\documentclass[a4paper, leqno, 12pt]{article}
\usepackage{amsmath,amsthm,stmaryrd}
\usepackage{amsfonts}
\usepackage{amssymb,latexsym}
\usepackage{enumerate}
\usepackage{color}
\usepackage{array}
\usepackage[all]{xy}
\usepackage{apptools}
\usepackage{mathrsfs}
\usepackage{tikz}

\usepackage{amsrefs}

\usepackage[left=2.5cm,right=2.5cm,top=3.0cm,bottom=3.0cm]{geometry}
\usepackage{hyperref}

\newtheorem{thm}{Theorem}[section]
\newtheorem{cor}[thm]{Corollary}
\newtheorem{lm}[thm]{Lemma}
\newtheorem{prop}[thm]{Proposition}
\newtheorem{con}[thm]{Conjecture}

\theoremstyle{definition}
\newtheorem{df}[thm]{Definition}
\newtheorem{rem}[thm]{Remark}
\newtheorem{exa}[thm]{Example}
\newtheorem{prob}[thm]{Problem}
\newtheorem{dis}[thm]{Discussion} 

\numberwithin{equation}{section}

\def\NN{{\mathbb{N}}}
\def\ZZ{{\mathbb{Z}}}
\def\AA{{\mathbb{A}}}
\def\DD{{\mathbb{D}}}
\def\EE{{\mathbb{E}}}

\def\CA{{\cal A}}
\def\CC{{\cal C}}
\def\CB{{\cal B}}
\def\CD{{\cal D}}
\def\CE{{\cal E}}
\def\CF{{\cal F}}
\def\CG{{\cal G}}

\def\CT{{\cal T}}

\def\CL{{\cal L}}

\def\CR{{\cal R}}
\def\CS{{\cal S}}

\def\CX{{\cal X}}

\def\res{\textnormal{res}} \def\wt{\widetilde} \def\und#1{\underline{#1}}

\def\epv {{ $\mbox{}$\hfill \qedsymbol }}
\def\wh#1{\widehat{#1}}

\def\Im{\textnormal{Im}}

\def\snull{\textnormal{ }}

\def\ov{\overline}
\def\wt{\widetilde}
\def\ra{\rightarrow}

\def\lmod{\mbox{{\rm -mod}}}
\def\lMod{\mbox{{\rm -Mod}}}
\def\lMOD{\mbox{{\rm -MOD}}}

\def\mod{\mbox{{\rm mod}}}

\def\ind{\mbox{{\rm ind}}}

\newcommand{\Rep}{\mathrm{Rep}}
\def\Mod{\mbox{{\rm Mod}}}

\def\Coker{\mbox{{\rm Coker}}}  \def\Coim{\mbox{{\rm Coim}}}
\def\supp{\mbox{\rm supp}}

\def\MOD{{\rm MOD}} \def\Ind{{\rm Ind}} \def\ind{{\rm ind}}

\def\Add{\textnormal{Add}} \def\res{\textnormal{res}}

\DeclareMathOperator{\ob}{ob}

\let\mod=\undefined
\DeclareMathOperator{\KG}{KG}
\DeclareMathOperator{\Ker}{Ker}
\DeclareMathOperator{\op}{op}

\DeclareMathOperator{\End}{End}
\DeclareMathOperator{\Hom}{Hom}
\DeclareMathOperator{\mod}{mod}

\begin{document}

\baselineskip=17pt


\title{Covering theory, functor categories and the Krull-Gabriel dimension}

\author{Grzegorz Pastuszak${}^{*}$}




\date{}

\maketitle

\begin{center}Dedicated to the memory of Professor Andrzej Skowro\'nski\end{center}

\renewcommand{\thefootnote}{}
\footnote{${}^{*}$Faculty of Mathematics and Computer Science, Nicolaus Copernicus University, Chopina 12/18, 87-100 Toru\'n, Poland, e-mail: past@mat.umk.pl.}
\footnote{MSC 2020: Primary 16G20; Secondary 16B50, 18E10.}
\footnote{Key words and phrases: Krull-Gabriel dimension, Galois coverings, functor categories, tensor product bifunctor, representation type, domestic algebra.}

\renewcommand{\thefootnote}{\arabic{footnote}}
\setcounter{footnote}{0}

\begin{abstract} Assume that $K$ is an algebraically closed field, $R$ a locally bounded $K$-category, $G$ an admissible group of $K$-linear automorphisms of $R$ and $F:R\ra A\cong R\slash G$ the associated Galois $G$-covering functor. In the first part of the paper we show that $\KG(R)\leq\KG(A)$ where $\KG$ denotes the \emph{Krull-Gabriel dimension}. In other words, we obtain a general fact that the \emph{Galois covering functors do not increase the Krull-Gabriel dimension.} This result is proved by developing the \emph{Galois covering theory of functor categories}, based on the existence of the \emph{tensor product bifunctor} for categories of modules over small $K$-categories. We understand this theory as the theory of the left and the right adjoint functors $\Phi,\Theta:\MOD(\CR)\ra\MOD(\CA)$ to the pull-up functor $\Psi=(F_{\lambda})_{\bullet}:\MOD(\CA)\ra\MOD(\CR)$, along the classical push-down functor $F_{\lambda}:\CR\ra\CA$ where $\CR=\mod(R)$, $\CA=\mod(A)$ and $(F_{\lambda})_{\bullet}=(-)\circ F_{\lambda}$. In the case $F_{\lambda}$ is dense, $\Phi$ and $\Theta$ are natural generalizations of the classical push-down functors $F_{\lambda}$ and $F_{\rho}$ to the level of functor categories. Generally, $\Phi$ and $\Theta$ restrict to categories $\CF(R),\CF(A)$ of \emph{finitely presented functors} and the restricted functors $\Phi,\Theta:\CF(R)\ra\CF(A)$ coincide. The second part of the paper is devoted to studying the functor $\Phi:\CF(R)\ra\CF(A)$. First we show that $\Phi$ is a Galois $G$-precovering of functor categories. Then we consider an important special case when the category $R$ is simply connected and locally representation-finite. In this case, $\Phi:\CF(R)\ra\CF(A)$ may be studied in terms of classical covering theory which allows to give an example of $\Phi$ which is not dense. This justifies the introduction of the \emph{functors of the first kind} and the \emph{functors of the second kind}, following the terminology of Dowbor and Skowro\'nski. In the final part of the paper, we give applications of our results in determining the Krull-Gabriel dimension and the representation type of some classes of algebras. Last but not least, we discuss applications to the conjecture of M. Prest, relating the Krull-Gabriel dimension of an algebra with its representation type.
\end{abstract}

\tableofcontents

\section{Introduction}

Throughout, $K$ is a fixed algebraically closed field. We denote by $\MOD(K)$ and $\mod(K)$ the categories of all $K$-vector spaces and all finite dimensional $K$-vector spaces, respectively. By an \textit{algebra} $A$ we mean a finite dimensional associative basic $K$-algebra with a unit. By an \textit{$A$-module} we mean a right $A$-module. We denote by $\MOD(A)$ and $\mod(A)$ the categories of all $A$-modules and all finitely generated $A$-modules, respectively. 

This paper links three theories which play a prominent role in the representation theory of finite dimensional $K$-algebras: the \emph{covering theory of locally bounded $K$-categories}, the \emph{theory of functor categories} and the \emph{theory of Krull-Gabriel dimension}.

Covering techniques in representation theory emerged in the papers of C. Riedtmann \cite{Rie}, P. Gabriel \cite{Ga}, K. Bongartz and P. Gabriel \cite{BoGa}, R. Martinez-Villa and J. A. de la Pe\~na \cite{MP}, and P. Dowbor and A. Skowro\'nski \cite{DoSk}. The goal was natural for the classical representation theory, namely, to determine and compute the indecomposable modules over finite dimensional algebras. Indeed, it turns our that when the field $K$ is algebraically closed, then any algebra $A$ is isomorphic with a \emph{bound quiver $K$-algebra} which in turn can be viewed as a finite \emph{bound quiver $K$-category} \cite{Ga1,Ga2}, see also Section 2.1. Bound quiver $K$-categories are the main examples of \emph{locally bounded $K$-categories}. A \emph{covering functor} is some $K$-linear functor $F:R\ra A$ between locally bounded $K$-categories which closely relates representation theories of $R$ and $A$. Similarly to the covering theory of topological spaces, the quiver of $R$ has less loops which makes representation theory of $R$ simpler then that of $A$. Moreover, for any algebra $A$, there is always the \emph{universal covering functor} $R\ra A$ \cite{MP}. This is a particular case of a \emph{Galois covering}, relating the categories $R$ and $A$ even closer, see Section 2.2 for definitions. Therefore, covering theory allows to describe indecomposable modules over an algebra in terms of indecomposable modules over some bound quiver $K$-category which is usually easier to study. This property is a source of fruitful applications of covering techniques in the representation theory of finite dimensional $K$-algebras.

Functor categories have been widely used in the representation theory of finite dimensional algebras. Pioneer work in this direction was done by M. Auslander in \cite{Au0} where the author proves that if $A$ is an algebra, then a functor $S:\mod(A)\ra\mod(K)$ is finitely presented and simple if and only if there exists a right minimal almost split map $g:M\ra N$ in $\mod(A)$ such that the sequence of functors $${}_{A}(-,M)\xrightarrow{{}_{A}(-,g)}{}_{A}(-,N)\ra S\ra 0$$ is a minimal projective presentation of $S$, see \cite[Chapter IV]{AsSiSk} for details. Furthermore, in \cite[Corollary 3.14]{Au} Auslander proved that representation-finite algebras are exactly those algebras for which all finitely presented functors are of finite length. These two results initiated the \textit{functorial approach} to the representation theory of algebras. This approach is now commonly used in the field. 

A natural continuation of the line of research started by Auslander is the study of the \textit{Krull-Gabriel filtration} of the category of finitely presented functors and the \textit{Krull-Gabriel dimension}. Assume that $A$ is an algebra and let $\CF(A)$ be the category of all finitely presented contravariant $K$-linear functors $\mod(A)\ra\mod(K)$. Recall that a functor $T:\mod(A)\ra\mod(K)$ is \emph{finitely presented} if and only if $T\cong\Coker{}_{A}(-,f)$ where $f$ is an $A$-module homomorphism. We refer to Section 2.4 for details of this definition. The associated \textit{Krull-Gabriel filtration} $(\CF(A)_{\alpha})_{\alpha}$ \cite{Po} is the filtration $$\CF(A)_{-1}\subseteq\CF(A)_{0}\subseteq\CF(A)_{1}\subseteq\hdots\subseteq\CF(A)_{\alpha}\subseteq\CF(A)_{\alpha+1}\subseteq\hdots$$ of $\CF(A)$ by Serre subcategories, defined recursively as follows: 
\begin{enumerate}[\rm(1)]
	\item $\CF(A)_{-1}=0$,
	\item $\CF(A)_{\alpha+1}$ is the Serre subcategory of $\CF(A)$ formed by all functors having finite length in the quotient category $\CF(A)\slash\CF(A)_{\alpha}$, for any ordinal number $\alpha$,
	\item $\CF(A)_{\beta}=\bigcup_{\alpha<\beta}\CF(A)_{\alpha}$, for any limit ordinal $\beta$.
\end{enumerate} Following \cite{Ge2,Geigle1985}, the \textit{Krull-Gabriel dimension} $\KG(A)$ of $A$ is defined as the smallest ordinal number $\alpha$ such that $\CF(A)_{\alpha}=\CF(A)$, if such a number exists. We set $\KG(A)=\infty$ if this is not the case. If $\KG(A)\in\NN$, then the Krull-Gabriel dimension of $A$ is \textit{finite}. If $\KG(A)=\infty$, then the Krull-Gabriel dimension of $A$ is \textit{undefined}. We refer to Section 2.3 for the definition of Krull-Gabriel dimension of locally bounded $K$-categories and small abelian categories in general.

A fundamental motivation to study the Krull-Gabriel dimension comes from the fact that the Krull-Gabriel filtration of the category $\CF(A)$ leads to a hierarchy of exact sequences in $\mod(A)$ where the almost split sequences form the lowest level \cite{Geigle1985,Sch3,Sch2}. Another one is the following conjecture due to M. Prest \cite{Pr2}. This conjecture fits naturally into the line of research initiated by Auslander in \cite{Au0}.

\begin{con}\label{00c1} Assume that $K$ is an algebraically closed field. A finite dimensional $K$-algebra $A$ is of domestic representation type if and only if the Krull-Gabriel dimension $\KG(A)$ of $A$ is finite.\footnote{We recall that Prest made a similar conjecture, relating the representation type of an algebra $A$ with existence of \emph{super-decomposable pure-injective $A$-modules}, see \cite{Pr3,Zi}.}
\end{con}

We refer to Chapter XIX of \cite{SiSk3} for precise definitions of finite, tame and wild representation type of an algebra, and the stratification of tame representation type into domestic, polynomial and nonpolynomial growth \cite{SkBC}. We refer to Introduction of \cite{P4} for more motivations to study the Krull-Gabriel dimension.

It is clear that the known results support the conjecture of Prest. We list them here for the sake of completeness. As already mentioned, M. Auslander proves in \cite{Au} that $\KG(A)=0$ if and only if the algebra $A$ is of finite type. H. Krause shows in \cite[11.4]{Kr2} that $\KG(A)\neq 1$, for any algebra $A$. If $A$ is a tame hereditary algebra, then $\KG(A)=2$ by the result of W. Geigle \cite[4.3]{Ge2}. A. Skowro\'nski shows in \cite[Theorem 1.2]{Sk5} that if $A$ is a cycle-finite algebra \cite{AsSk1} of domestic representation type, then $\KG(A)=2$. If $A$ is a strongly simply connected algebra \cite{Sk1}, then $A$ is of domestic type if and only if $\KG(A)$ is finite, by the result of M. Wenderlich \cite{We}. R. Laking, M. Prest and G. Puninski show in \cite{LPP} that string algebras \cite{SkWa} of domestic representation type have finite Krull-Gabriel dimension. Many important classes of algebras are known to have the Krull-Gabriel dimension undefined. Indeed, this holds for strictly wild algebras and wild algebras (see \cite[Chapter 13]{Pr} and \cite{GP,P5}), non-domestic string algebras (see \cite[Proposition 2]{Sch2}), tubular algebras (see \cite{Geigle1985,Ge2}), pg-critical algebras \cite{NoSk}, strongly simply connected algebras of nonpolynomial growth \cite{Sk2} and some algebras with strongly simply connected Galois coverings \cite{Sk3} (see \cite{KaPa1,KaPa2,KaPa3}).\footnote{The results of \cite{KaPa1,KaPa2,KaPa3} show that the width of the lattice of all pointed modules \cite{KaPa2} over any of these algebras is infinite and hence their Krull-Gabriel dimension is undefined, see for example \cite{Pr2}.}

Now we briefly describe the main results from our previous work \cite{P4,P6} where we show some first connections between covering theory, functor categories and Krull-Gabriel dimension. Our goal in these papers was to determine the Krull-Gabriel dimension for the standard self-injective algebras of polynomial growth \cite{Sk4}, partially by applying Skowro\'nski's characterization in \cite[Theorem 1.2]{Sk5}. The very nature of this class of algebras suggested to consider Galois $G$-coverings $R\ra A\cong R\slash G$ such that $R$ is a locally bounded \emph{locally support-finite} \cite{DoSk,DoLeSk} \emph{intervally finite} \cite[2.1]{BoGa} $K$-category and $G$ an admissible torsion-free group of $K$-linear automorphisms of $R$, see Section 2 for most of definitions. We have proved that in this special situation the Krull-Gabriel dimension is preserved by such Galois coverings, that is, $\KG(R)=\KG(A)$ \cite[Theorem 1.4]{P6}.\footnote{In \cite{P6} we correct an ill argument from Theorem 6.3 of \cite{P4}. In fact, the faultless results of \cite{P4} can be seen as showing only that $\KG(R)\leq\KG(A)$ if $R$ is locally support-finite and the group $G$ is torsion-free whereas \cite{P6} shows that $\KG(A)\leq\KG(R)$, under seemingly necessary assumptions that $R$ is locally support-finite and intervally finite.} We refer to Section 4.3 for details about this result, see in particular Theorem \ref{t7} and \ref{t8}, as well as a sketch of the proof. 

Because of abundance of covering techniques in representation theory, \cite[Theorem 1.4]{P6} is very useful in the study of Krull-Gabriel dimension. Indeed, in \cite{P4} we calculate the Krull-Gabriel dimension of the tame locally support-finite repetitive categories \cite{AsSk4} (see \cite[Theorem 7.3]{P4}) and then we deduce the dimension for standard self-injective algebras of polynomial growth (see \cite[Theorem 8.1]{P4}), as expected. Moreover, in \cite{J-PP1} we apply the theorem to relate Krull-Gabriel dimensions of repetitive categories, cluster repetitive categories and cluster-tilted algebras \cite{As} (see \cite[Theorem 3.4]{J-PP1}). Finally, in \cite[Theorem 3.6]{J-PP1} we characterize the Krull-Gabriel dimension for these classes. Recent applications for the weighted surface algebras \cite{ErSk1} are given in \cite{EJ-PP}.

We observe in \cite{P4,P6} that a Galois $G$-covering $F:R\ra A$ naturally induces some functors on the level of the functor categories. Indeed, assume that $\CR=\mod(R)$ and let $\mod(\CR)$ be the category of all contravariant $K$-linear functors $\CR\ra\mod(K)$. Let $F_{\lambda}:\mod(R)\ra\mod(A)$ be one of the \emph{push-down functors}, see Section 2.2 for the definition. Then we define $\Psi:\CF(A)\ra\mod(\CR)$ as $(-)\circ F_{\lambda}$ and $\Phi:\CF(R)\ra\CF(A)$ as $$\Phi(\Coker{}_{R}(*,f))=\Coker{}_{A}(-,F_{\lambda}(f)),$$ for any $R$-module homomorphism $f$, see \cite[Section 5]{P4} or \ref{r3} for more details. We show in \cite[Theorem 5.5]{P4} that $\Phi:\CF(R)\ra\CF(A)$ is a \emph{Galois $G$-precovering of functor categories} \ref{0d1}. In particular, we show in \cite[Theorem 5.5 (3)]{P4} that $\Phi$ is the right adjoint to $\Psi$.\footnote{In fact, the functor $\Phi:\CF(R)\ra\CF(A)$ is both the right and the left adjoint to $\Psi$.} This result, together with \cite[Theorem 1.3]{P6} (see Section 4.3), yields that $\KG(R)=\KG(A)$.

The present paper develops the above ideas much further. Namely, we introduce the \emph{Galois covering theory of functor categories} which is a vast generalization of covering techniques that appeared in \cite{P4,P6}. On the one hand, we believe it is the appropriate general setting for covering theory of functor categories. For this setting, our previous results appear as a very special case. On the second hand, it allows to draw important general conclusions for the theory of Krull-Gabriel dimension, with many potential applications. We show in the paper a flavor of these applications, particularly in the context of the conjecture of Prest \ref{00c1}. Finally, we stress that the paper contains a number of open problems, usually formulated as remarks with some preliminary thoughts.

Now we describe the organization of the paper, as well as our main results, in more detail. Section 2 is preliminary and contains some basic information about modules over locally bounded $K$-categories, Galois coverings of locally bounded $K$-categories, Krull-Gabriel dimension of abelian categories and functor categories.

Section 3 starts with recalling the \emph{general tensor product bifunctor} \cite{FPN,Mi} over arbitrary $K$-categories. In order to describe our main results, we introduce some notation, described in detail in Section 3.1. Namely, if $\CA$ is a small $K$-category, then we denote by $\MOD(\CA)$ the category of all \emph{right $\CA$-modules}, that is, the category of all contravariant $K$-linear functors $\CA\ra\MOD(K)$. Assume that $F:R\ra A$ is a given Galois $G$-covering functor and set $\CR=\mod(R)$, $\CA=\mod(A)$. Let $\Psi=(F_{\lambda})_{\bullet}=(-)\circ F_{\lambda}$ be the pull-up functor along $F_{\lambda}:\mod(R)\ra\mod(A)$. Based on the tensor product bifunctor, we obtain the following theorem which summarizes the main results of Section 3.

\begin{thm}\label{00t1} Assume that $F:R\ra A$ is a Galois $G$-covering functor. The following assertions hold:
\begin{enumerate}[\rm(1)]
  \item \textnormal{\ref{c1}} The pull-up functor $\Psi=(F_{\lambda})_{\bullet}=(-)\circ F_{\lambda}:\MOD(\CA)\ra\MOD(\CR)$ has the functor $$\Phi:=?\otimes_{\CR}[{}_{A}(-,F_\lambda(*))]:\MOD(\CR)\ra\MOD(\CA)$$ as the left adjoint and the functor $$\Theta:={}_{\CR}({}_{A}(F_{\lambda}(*),-),?):\MOD(\CR)\ra\MOD(\CA)$$ as the right adjoint.
  \item \textnormal{\ref{t2}} We have $\Phi(\CF(R))\subseteq\CF(A)$, $\Theta(\CF(R))\subseteq\CF(A)$ and $\Phi|_{\CF(R)}\cong\Theta|_{\CF(R)}$, hence both functors $\Phi,\Theta:\CF(R)\ra\CF(A)$ are exact. Moreover, these functors are also faithful.
  \item \textnormal{\ref{t3}} We have $\KG(R)\leq\KG(A)$.
\end{enumerate}
\end{thm} The assertion of $(1)$ is a consequence of more general Theorem \ref{t1}, already proved in \cite[Proposition 6.1]{Bu}. As we show in \ref{r1.1}, the classical covering functors $F_{\bullet}$, $F_{\lambda}$ and $F_{\rho}$ fall into this scheme. The assertion of $(2)$ is a core result, partially based on technical facts proved in Lemma \ref{l1} and Proposition \ref{p1}. Furthermore, $(2)$ directly implies $(3)$, stating that Galois covering functors do not increase the Krull-Gabriel dimension. Observe that $\Phi$ and $\Theta$ behave as classical push-down functors $F_{\lambda},F_{\rho}:\MOD(R)\ra\MOD(A)$, as they also restrict to categories of finite dimensional modules and they coincide on $\mod(R)$.\footnote{This property does not hold in general, but it holds for the Galois coverings.} 

In Section 4 we study the restricted functor $\Phi:\CF(R)\ra\CF(A)$. We also consider a special case when the push-down functor $F_{\lambda}:\mod(R)\ra\mod(A)$ is dense. At the beginning we show that the action of $G$ on $\mod(R)$ (see Section 2.2) induces an action of $G$ on $\MOD(\CR)$. Proposition \ref{p2} gives the main properties of this action. In Proposition \ref{p3} we recall that the category $\CF(R)$ of finitely presented functors is an abelian Krull-Schmidt $K$-category. The main results of the section are as follows.

\begin{thm}\label{00t2} Assume that $F:R\ra A$ is a Galois $G$-covering functor. The following assertions hold.
\begin{enumerate}[\rm(1)]
  \item \textnormal{\ref{t4}} If the group $G$ is torsion-free, then $\Phi:\CF(R)\ra\CF(A)$ is a Galois $G$-precovering of functor categories.
  \item \textnormal{\ref{t5}} There are natural equivalences $$\Psi(\Phi(\cdot))\cong\bigoplus_{g\in G}g(\cdot)\textnormal{ and }\Psi(\Theta(\cdot))\cong\prod_{g\in G}g(\cdot).$$
  \item \textnormal{\ref{c2}, \ref{t6}} Assume that the push-down functor $F_{\lambda}:\mod(R)\ra\mod(A)$ is dense. Then the functor $\Phi:\MOD(\CR)\ra\MOD(\CA)$ is a subfunctor of $\Theta:\MOD(\CR)\ra\MOD(\CA)$. Moreover, both functors $\Phi,\Theta$ have descriptions similar to the case of classical covering functors $F_{\lambda},F_{\rho}:\MOD(R)\ra\MOD(A)$, respectively.
\end{enumerate}
\end{thm} The assertion of $(1)$ is a significant generalization of \cite[Theorem 5.5]{P4} where we assume that the group $G$ is torion-free and $R$ is locally support-finite. Recall that in this case the push-down functor $F_{\lambda}$ is dense which greatly simplifies considerations. The assertion of $(2)$ is an analog of equivalences $$(F_{\bullet}\circ F_{\lambda})\cong\bigoplus_{g\in G}(-)^{g}\textnormal{ and }(F_{\bullet}\circ F_{\rho})\cong\prod_{g\in G}(-)^{g}$$ for the classical covering functors. These properties follow from the fact that any Galois covering is surjective on objects, see Section 2.2 for the definition. The classical covering functors $F_{\bullet}=(-)\circ F^{\op}:\MOD(A)\ra\MOD(R)$ and $F_{\lambda},F_{\rho}:\MOD(R)\ra\MOD(A)$ are described in detail in Section 2.2. In turn, \ref{c2} and \ref{t6} describe in detail the functors $\Phi,\Theta:\MOD(\CR)\ra\MOD(\CA)$ when $F_{\lambda}$ is dense. Analogies between $\Phi,\Theta$ and $F_{\lambda},F_{\rho}$ are easily seen in this case. Main properties of $\Phi$ and $\Theta$, for a dense $F_{\lambda}$, are summarized in Theorem \ref{t6}. Remark \ref{r3} discusses equivalent definitions of $\Phi:\CF(R)\ra\CF(A)$ in the dense case, mainly based on \cite{P4}. In \ref{r4} we consider possible generalizations of some results of the section. Finally, in Section 4.3 we give a comprehensive survey of our previous results from \cite{P4,P6}. In our opinion such a reminder is both convenient and natural, because the case when $R$ is locally support-finite $G$ is torsion-free is the main one where $F_{\lambda}$ is a dense functor.  

Section 5 begins with introducing in Definition \ref{d1} the functors of the first kind and the functors of the second kind, following terminology of Dowbor and Skowro\'nski from \cite{DoSk}, see also \cite{DoLeSk}. Functors of the first kind are indecomposable functors in $\CF(R)$, lying in the image of the functor $\Phi:\CF(R)\ra\CF(A)$. Functors of the second kind are the remaining indecomposable functors. The rest of the section is devoted to motivate Definition \ref{d1}. Indeed, we give in \ref{e1} an example of a Galois $G$-covering functor $F:R\ra A$ such that the functor $\Phi$ is not dense. In this way we confirm our conjecture formulated in \cite[Remark 5.6]{P4}. In the example, $R$ is a locally representation-finite simply connected \cite{AsSk3} locally bounded $K$-category and $G$ is an admissible torsion-free group of $K$-linear automorphisms of $R$. We show that in this special situation $\Phi$ may be studied in terms of classical covering theory. Indeed, first we show in Proposition \ref{p4} that locally representation-finiteness of $R$ implies that the category $\ind(R)$ is locally bounded. Then we recall in Theorem \ref{t10} a classical fact that in this case $A\cong R\slash G$ is representation-finite and $F_{\lambda}:\ind(R)\ra\ind(A)$ is a Galois $G$-covering of locally bounded $K$-categories, inducing a Galois $G$-covering on the level of Auslander-Reiten quivers of $R$ and $A$. We conclude in Proposition \ref{p5} that $\Phi\cong(F_\lambda)_{\lambda}$ in this case. Furthermore, the simple connectedness of $R$ implies that for any $C\in\{R,A\}$ we have $\ind(C)\cong K(\Gamma_{C})$ where $K(\Gamma_{C})$ denotes the \textit{mesh-category} \cite{Rie} of the Auslander-Reiten quiver $\Gamma_{C}$ of $C$. Naturally, these equivalences are given by $\phi_{C}:\ind(C)\ra K(\Gamma_{C})$ such that $\phi_{C}(X)=[X]$, for any $X\in\ind(C)$, where $[X]$ denotes the vertex in $\Gamma_{C}$ corresponding to the module $X$. The main result of the section is presented below.

\begin{thm}\textnormal{\ref{t11}}\label{00t3} Assume that $R$ is a locally representation-finite simply connected locally bounded $K$-category, $G$ an admissible torsion-free group of $K$-linear automorphisms of $R$ and $F:R\ra A\cong R\slash G$ the Galois covering. There exists a Galois covering $F_{\lambda}^{\Gamma}:K(\Gamma_{R})\ra K(\Gamma_{A})$ such that the following diagram $$\xymatrix{\mod(K(\Gamma_{R}))\ar[rr]^{(F_{\lambda}^{\Gamma})_{\lambda}}\ar[d]^{\widetilde{\phi}_{R}}&&\mod(K(\Gamma_{A}))\ar[d]^{\widetilde{\phi}_{A}}\\\mod(\ind(R))\ar[rr]^{(F_{\lambda})_\lambda}\ar[d]^{\eta_R}&&\mod(\ind(A))\ar[d]^{\eta_{A}}\\\CF(R)\ar[rr]^{\Phi_{F}}&& \CF(A),}$$ where $\widetilde{\phi}_{R}=(-)\circ\phi_{R}$ and $\widetilde{\phi}_{A}=(-)\circ\phi_{A}$, is a commutative diagram whose columns are equivalences.
\end{thm} The above theorem serves as a theoretical basis for Example \ref{e1}. Some further examples are given in \ref{e2} and \ref{e3}. The section is concluded with Remark \ref{r7} where we briefly discuss cases where $R$ is not locally representation-finite. Moreover, we consider in \ref{r7} a candidate for a generalization of the definition of locally support-finite locally bounded $K$-category to the level of finitely presented functors. 

In Section 6 we show some special applications of Theorem \ref{00t1} (3) (Theorem \ref{t3}), stating that Galois coverings do not increase the Krull-Gabriel dimension. Observe that the inequality $\KG(R)\leq\KG(A)$ implies that $\KG(R)$ is finite if $\KG(A)$ is finite, and $\KG(A)$ is undefined if so if $\KG(R)$. We use these facts in calculating the Krull-Gabriel dimension or determining the representation type for some classes of tame algebras with strongly simply connected Galois coverings \cite{Sk3} and the weighted surface algebras \cite{ErSk1}. We recall in the section some necessary facts concerning strongly simply connected algebras \cite{Sk4}, algebras with strongly simply connected Galois coverings \cite{NoSk} and weighted surface algebras. In particular, we define some special two families $D(\lambda)^{(1)}$ and $D(\lambda)^{(2)}$ \cite{ErSk2} of weighted surface algebras. The main results of the section are as follows.

\begin{thm}\label{00t4} The following assertions hold.
\begin{enumerate}[\rm(1)]
  \item \textnormal{\ref{t13}} Assume that $A$ is a tame algebra having a strongly simply connected Galois covering. If $\KG(A)$ is finite, then $A$ is of domestic type. 
  \item \textnormal{\ref{t14}} We have $\KG(D(\lambda)^{(1)})=\KG(D(\lambda)^{(2)})=\infty$.
\end{enumerate}   
\end{thm} We argue in Remark \ref{r9} that the converse of the assertion $(1)$ holds and so the conjecture of Prest is valid for this class of algebras. The second assertion of the above theorem gives a flavor of our main results from \cite{EJ-PP} where we prove, in particular, that all weighted surface algebras have Krull-Gabriel dimension undefined. We stress that in these theorems we cannot apply our previous results from \cite{P4,P6}, see Theorem \ref{t8}, because in general we do not consider torsion-free groups or dense push-down functors. 

The final Section 7 of the paper contains our general thoughts on the conjecture of M. Prest \ref{00c1}, in the context of our main results. In a sense, it is a continuation of Section 6.1. Indeed, we believe that it is reasonable to restrict the generality of this conjecture to the case considered in Theorem \ref{t15} which is used indirectly in Section 6.1. This theorem is in turn based on the renowned result of \cite[Theorem 3.6]{DoSk}. We formulate in \ref{pro1} and \ref{pro2} some related open problems and share our ideas in \ref{dis1}. We finish with Remark \ref{r10} which is particularly interesting in our opinion.  

We are grateful to P. Dowbor and S. Kasjan for discussions about problems considered in the paper. Results of this paper have been presented on several conferences, including ARTA VII, Toru\'n, Poland, 2022; ARTA VIII, Kingston, Canada, 2023; Functor and Tensor Categories, Models And Systems, Santiago de Compostela, Spain, 2024; ICRA 21, Shanghai, China, 2024.

\section*{Acknowledgements} Andrzej Skowro\'nski gave us in October 2015 the problem of describing the Krull-Gabriel dimension of the standard self-injective algebras of polynomial growth. This was the beginning of our serious interest in the theory of Krull-Gabriel dimension and related questions about functor categories. These interests resulted in the papers \cite{P4,P6}, \cite{J-PP1,EJ-PP}, as well as the present one. We view the results of this paper as a (partial) culmination of that work, which would have not been obtained without Andrzej's encouragement and published work.       

Andrzej Skowro\'nski passed away on 22.10.2020. We dedicate this paper to his memory.

\section{Background}

This section is devoted to recall some basic information related to categories of modules over locally bounded $K$-categories. This includes Galois coverings of locally bounded $K$-categories and associated functors between module categories. We also give brief introduction to Krull-Gabriel dimension and functor categories. The notation and results we introduce are freely used in the paper.  

\subsection{Modules over locally bounded $K$-categories}

Assume that $R$ is a $K$-category and let $\ob(R)$ be the class of all objects of $R$. If $x,y\in\ob(R)$, then $R(x,y)$ denotes the space of all morphisms from $x$ to $y$. Following \cite{Ga,BoGa}, we say that a $K$-category $R$ is \textit{locally bounded} if and only if 

\begin{itemize}
	\item distinct objects of $R$ are not isomorphic,
	\item the algebra $R(x,x)$ is local, for any $x\in\ob(R)$,
	\item $\sum_{y\in\ob(R)}\dim_{K}R(x,y)<\infty$, $\sum_{y\in\ob(R)}\dim_{K}R(y,x)<\infty$, for any $x\in\ob(R)$.
\end{itemize}

The main examples of locally bounded $K$-categories are given by the \textit{bound quiver $K$-categories}. We recall necessary definitions below.

Assume that $Q=(Q_{0},Q_{1})$ is a \emph{quiver} where $Q_{0}$ is the set of vertices and $Q_{1}$ the set of arrows. Then $Q$ is \textit{finite} if both sets $Q_{0}$ and $Q_{1}$ are finite. The quiver $Q$ is \textit{locally finite} if the number of arrows in $Q_{1}$ starting or ending in any vertex is finite. If $\alpha\in Q_{1}$ is an arrow, then $s(\alpha)$ denotes its starting vertex and $t(\alpha)$ its terminal vertex. Assume that $x,y\in Q_{0}$. By a \textit{path} from vertex $x$ to vertex $y$ in $Q$ we mean a sequence $c_{1}\dots c_{n}$ in $Q_{1}$ such that $s(c_{1})=x$, $t(c_{n})=y$ and $t(c_{i})=s(c_{i+1})$, for $1\leq i<n$.\footnote{We compose arrows in the opposite direction to functions, consistently with \cite{AsSiSk}.} We associate the \textit{stationary path} $e_{x}$ to each vertex $x\in Q_{0}$ and we set $s(e_{x})=t(e_{x})=x$.

The \textit{path $K$-category} $\und{KQ}$ of a locally finite quiver $Q$ is a $K$-category whose objects are the vertices of $Q$ and the $K$-linear space $\und{KQ}(x,y)$ of morphisms from $x$ to $y$ is generated by all paths from $y$ to $x$. The composition in $\und{KQ}$ is defined by concatenation of paths in $Q$. If $I$ is an \emph{admissible ideal} in $\und{KQ}$ \cite{Pog,BoGa}, then the pair $(Q,I)$ is called the \textit{bound quiver}. The associated quotient $K$-category $\und{KQ}\slash I$ is locally bounded and called the \textit{bound quiver $K$-category}. It is shown in \cite{BoGa} that any locally bounded $K$-category over an algebraically closed field $K$ is isomorphic with some bound quiver $K$-category.


Throughout the paper, we consider both covariant and contravariant functors. Contravariant functors $\CC\ra\CD$ are usually denoted as covariant functors $\CC^{\op}\ra\CD$ where $\CC^{\op}$ is the category opposite to $\CC$. Recall that a covariant functor $\CC^{\op}\ra\CD$ induces a contravariant one $\CC\ra\CD$ in a natural way, and vice versa.

Assume that $R$ is a locally bounded $K$-category. A \textit{right $R$-module} $M$ (or simply an \textit{$R$-module}) is a $K$-linear covariant functor $M:R^{\op}\ra\MOD(K)$. An $R$-module $M$ is \textit{finite dimensional} if and only if $\sum_{x\in\ob(R)}\dim_{K} M(x)<\infty$ and \emph{locally finite dimensional} if and only if $\dim_{K} M(x)<\infty$, for any $x\in\ob(R)$. Assume that $M,N:R^{\op}\ra\MOD(K)$ are $R$-modules. An \textit{$R$-module homomorphism} $f:M\ra N$ is a natural transformation of functors $(f_{x})_{x\in\ob(R)}$ where $f_{x}:M(x)\ra N(x)$ is a vector space homomorphism, for any $x\in\ob(R)$. The space of all homomorphisms from $M$ to $N$ is denoted by $\Hom_{R}(M,N)$. We usually write ${}_{R}(M,N)$ instead of $\Hom_{R}(M,N)$. Analogous notation for hom-spaces is used for any additive categories.

We denote by $\MOD(R)$, $\Mod(R)$ and $\mod(R)$ the categories of all $R$-modules, all locally finite dimensional $R$-modules and all finite dimensional $R$-modules, respectively. The full subcategories of $\Mod(R)$ and $\mod(R)$, formed by representatives of isomorphism classes of all indecomposable $R$-modules, are denoted by $\Ind(R)$ and $\ind(R)$, respectively. As usual, $\Gamma_{R}$ is the \emph{Auslander-Reiten quiver} of the category $\ind(R)$. 

We denote by $R\lMOD$ the category $\MOD(R^{\op})$ of all \textit{left $R$-modules}, that is, $K$-linear covariant functors $M:R\ra\MOD(K)$. Clearly $R\lMod$ and $R\lmod$ are the left counterparts of $\Mod(R)$ and $\mod(R)$, respectively. Recall that the functor $$D=\Hom_{K}(-,K):\mod(R)\ra R\lmod$$ is a duality. If $A$ is another locally bounded $K$-category, then an \emph{$R$-$A$-bimodule} is a $K$-bilinear bifunctor $A^{\op}\times R\ra\MOD(K)$.

Assume that $R=\und{KQ}\slash I$ is a bound quiver $K$-category. It is well known that the category $\MOD(R)$ of $R$-modules is equivalent with the category $\Rep_{K}(Q,I)$ of \emph{$K$-linear representations} of the bound quiver $(Q,I)$. Moreover, if $(Q,I)$ is finite, then there is an equivalence of $\Rep_{K}(Q,I)$ and the category of all modules over the \textit{bound quiver $K$-algebra} $KQ\slash I$, restricting to equivalence of categories of finite dimensional modules, see \cite[III 1.6]{AsSiSk} for details.\footnote{We denote the admissible ideal in the $K$-category by the same letter $I$ as the corresponding admissible ideal in the $K$-algebra. In the case of finite $(Q,I)$, we identify the bound quiver $K$-category $R=\und{KQ}\slash I$ with the bound quiver $K$-algebra $KQ\slash I$.} More generally, any module $M\in\MOD(R)$ can be viewed as a direct sum $\widehat{M}=\bigoplus_{x\in\ob(R)}M(x)$ with a right action of $\widehat{R}=\bigoplus_{x,y\in\ob(R)}R(x,y)$ such that, for any $m\in\widehat{M}$ and $r\in\widehat{R}$, the element $m\cdot r$ is defined as in \cite[III 1.6]{AsSiSk}. In this language, $f:M\ra N$ is a homomorphism of $R$-modules, in the sense of a natural transformation of functors, if and only if $\widehat{f}:\widehat{M}\ra\widehat{N}$ is a homomorphism of $R$-modules in the usual sense, that is $\widehat{f}(m\cdot r)=\widehat{f}(m)\cdot r$. The same works for left $R$-modules and $R$-$A$-bimodules, and so we have the tensor product bifunctor $$-\otimes_{R}-:\MOD(R)\times R\lMOD\ra\MOD(K)$$ with the usual properties. In the paper we use freely all three equivalent descriptions of module categories over locally bounded $K$-categories.    

Assume that $R$ is a locally bounded $K$-category and let $x\in\ob(R)$. Then $P_{x}=R(-,x)$ and $I_{x}=D(R(-,x))$ denote the indecomposable projective and the indecomposable injective (right) $R$-module, associated with the vertex $x$, respectively. Observe that these modules are finite dimensional, because $R$ is locally bounded. 

\subsection{Galois coverings}

Here we recall the notion of a Galois $G$-covering functor \cite{BoGa} and some related concepts and facts. For a general definition of a covering functor the reader is referred to \cite{Ga}.

Assume that $R,A$ are locally bounded $K$-categories, $F:R\ra A$ is a $K$-linear functor and $G$ a group of $K$-linear automorphisms of $R$ acting freely on the objects of $R$ (i.e. $gx=x$ if and only if $g=1$, for any $g\in G$ and $x\in\ob(R)$). Then $F:R\ra A$ is a \textit{Galois G-covering} (or simply a \emph{Galois covering}) if and only if 
\begin{itemize}
	\item the functor $F:R\ra A$ induces isomorphisms $$\bigoplus_{g\in G}R(gx,y)\cong A(F(x),F(y))\cong\bigoplus_{g\in G}R(x,gy)$$ of vector spaces, for any $x,y\in\ob(R)$,
	\item the functor $F:R\ra A$ induces a surjective function $\ob(R)\ra\ob(A)$,
	\item $Fg=F$, for any $g\in G$,
	\item for any $x,y\in\ob(R)$ such that $F(x)=F(y)$ there is $g\in G$ such that $gx=y$. 
\end{itemize}
It is well known that $F:R\ra A$ satisfies the above conditions if and only if $F$ induces an isomorphism $A\cong R\slash G$ where $R\slash G$ is the \textit{orbit category}, see \cite{BoGa}.

Assume that $F:R\ra A\cong R\slash G$ is a Galois covering. Then the \textit{pull-up} functor $F_{\bullet}:\MOD(A)\ra\MOD(R)$ associated with $F$ is the functor $(-)\circ F^{\op}$. The pull-up functor is exact and has the left adjoint $F_{\lambda}:\MOD(R)\ra\MOD(A)$ and the right adjoint $F_{\rho}:\MOD(R)\ra\MOD(A)$ which are called the \textit{push-down} functors. We recall the description of the push-down functors below.

Assume that $M:R^{\op}\ra\MOD(K)$ is an $R$-module. We define the $A$-module $F_{\lambda}(M):A^{\op}\ra\MOD(K)$ in the following way. Assume that $a\in\ob(A)$ and $a=F(x)$, for some $x\in\ob(R)$. Then we have $$F_{\lambda}(M)(a)=\bigoplus_{g\in G}M(gx).$$ Assume that $\alpha\in A(b,a)$ and $a=F(x),b=F(y)$, for some $x,y\in\ob(R)$. Since $F$ induces an isomorphism $$\bigoplus_{g\in G}R(gy,x)\cong A(F(y),F(x)),$$ there are $\alpha_{g}:gy\ra x$, for $g\in G$, such that $\alpha=\sum_{g\in G}F(\alpha_{g})$. Then the homomorphism $$F_{\lambda}(M)(\alpha):F_{\lambda}(M)(a)\ra F_{\lambda}(M)(b)$$ is defined by homomorphisms $M(g\alpha_{g^{-1}h}):M(gx)\ra M(hy)$, for any $g,h\in G$.\footnote{We use the standard matrix notation for homomorphisms between finite direct sums, see Section 3 for details.} Assume that $f:M\ra N$ is an $R$-module homomorphism and $f=(f_{x})_{x\in\ob(R)}$, $f_{x}:M(x)\ra N(x)$. Then $F_{\lambda}(f):F_{\lambda}(M)\ra F_{\lambda}(N)$, $F_{\lambda}(f)=(\hat{f}_{a})_{a\in\ob(A)}$ and $\hat{f}_{a}:F_{\lambda}(M)(a)\ra F_{\lambda}(N)(a)$ is defined by homomorphisms $f_{gx}:M(gx)\ra N(gx)$, for any $g\in G$. For the module $F_{\rho}(M):A^{\op}\ra\MOD(K)$ we have $$F_{\rho}(M)(a)=\prod_{g\in G}M(gx)$$ and the rest of the definition is similar to the case of $F_{\lambda}$. Observe that $F_{\lambda}$ is a subfunctor of $F_{\rho}$ and both functors coincide on the category of finite dimensional $A$-modules.\footnote{General covering functors do not have this property.} Moreover, if an $R$-module $M$ is finite dimensional, then $F_{\lambda}(M)$ is finite dimensional. Hence the functor $F_{\lambda}$ restricts to a functor $\mod(R)\ra\mod(A)$. This functor is also denoted by $F_{\lambda}$. 

Assume that $R$ is a locally bounded $K$-category, $G$ is a group of $K$-linear automorphisms of $R$ acting freely on the objects of $R$ and $g\in G$. Given $R$-module $M$ we denote by ${}^{g}M$ the module $M\circ g^{-1}$. Given $R$-module homomorphism $f:M\ra N$ we denote by ${}^{g}f$ the $R$-module homomorphism ${}^{g}M\ra {}^{g}N$ such that ${}^{g}f_{x}=f_{g^{-1}x}$, for any $x\in\ob(R)$. This defines an action of $G$ on $\MOD(R)$. It is easy to see that the map $f\mapsto {}^{g}f$ defines isomorphism of vector spaces ${}_{R}(M,N)\cong{}_{R}({}^{g}M,{}^{g}N)$.

We say that the group $G$ is \textit{admissible} if and only if $G$ acts freely on the objects of $R$ and there are only finitely many $G$-orbits. In this case the orbit category $R\slash G$ is finite and we often treat it as an algebra. If $G$ is admissible, then we say that $G$ \textit{acts freely on $\ind(R)$} if and only if ${}^{g}M\cong M$ implies that $g=1$, for any $M\in\ind(R)$ and $g\in G$.

The \textit{support} $\supp (M)$ of a module $M\in\MOD(R)$ is the full subcategory of $R$ formed by all objects $x$ in $R$ such that $M(x)\neq 0$. The category $R$ is \textit{locally support-finite} \cite{DoSk} if and only if for any $x\in\ob(R)$ the union of the sets $\supp(M)$, where $M\in\ind(R)$ and $M(x)\neq 0$, is finite.

The main properties of the push-down functor $F_{\lambda}:\mod(R)\ra\mod(A)$ are summarized in the following theorem, based on \cite{Ga,BoGa,MP,DoSk}.

\begin{thm}\label{0t1} Assume that $R$ is a locally bounded $K$-category, $G$ an admissible group of $K$-linear automorphisms of $R$ and $F:R\ra A$ the Galois covering. Then the functor $F_{\lambda}:\mod(R)\ra\mod(A)$ satisfies the following assertions.
\begin{enumerate}[\rm(1)]
 \item 	There are isomorphisms $F_{\lambda}({}^{g}M)\cong F_{\lambda}(M)$ and $F_{\lambda}({}^{g}f)\cong F_{\lambda}(f)$, for any $R$-module $M$, $R$-homomorphism $f$ and $g\in G$.
\item There is an isomorphism $F_{\bullet}(F_{\lambda}(M))\cong\bigoplus_{g\in G}{}^{g}M$, for any $R$-module $M$. Moreover, if $X,Y\in\ind(R)$, then $F_{\lambda}(X)\cong F_{\lambda}(Y)$ implies $Y\cong {}^{g}X$, for some $g\in G$.
    \item If the group $G$ is torsion-free, then $G$ acts freely on $\ind(R)$. If the latter condition holds, then $F_{\lambda}:\mod(R)\ra\mod(A)$ preserves indecomposability.
\item The functor $F_{\lambda}$ induces the following isomorphisms of vector spaces $$\bigoplus_{g\in G}{}_{R}({}^{g}X,Y)\cong{}_{A}(F_{\lambda}(X),F_{\lambda}(Y))\cong\bigoplus_{g\in G}{}_{R}(X,{}^{g}Y),$$ for any $X,Y\in\mod(R)$.
	\item Assume that the group $G$ is torsion-free and $R$ is locally support-finite. The the push-down functor $F_{\lambda}:\mod(R)\ra\mod(A)$ is dense. This means that for any $M\in\mod(A)$ there is $X\in\mod(R)$ such that $F_{\lambda}(X)\cong M$.
\end{enumerate}
\end{thm}

\begin{rem}\label{0r1} We give more details about Theorem \ref{0t1} (4). Assume that $X,Y\in\mod(R)$. Observe that there are only finitely many $g\in G$ such that ${}_{R}(X,{}^{g}Y)\neq 0$, because $G$ acts freely on the objects of $R$. Then \ref{0t1} (4) follows easily by the fact that $(F_{\lambda},F_{\bullet})$ is an adjoint pair. Moreover, these isomorphisms are natural in both variables $X,Y$ and so the bifunctors ${}_{A}(F_{\lambda}(-),F_{\lambda}(\cdot))$, $\bigoplus_{g\in G}{}_{R}(-,{}^{g}(\cdot))$ and $\bigoplus_{g\in G}{}_{R}({}^{g}(-),\cdot)$ are equivalent. More specifically, the push-down functor $F_{\lambda}:\mod(R)\ra\mod(A)$ induces natural isomorphisms $$\nu_{X,Y}:\bigoplus_{g\in G}{}_{R}({}^{g}X,Y)\ra{}_{A}(F_{\lambda}(X),F_{\lambda}(Y))$$ of vector spaces, for any $X,Y\in\mod(R)$, given by $$\nu_{X,Y}((f_{g})_{g\in G})=\sum_{g\in G}F_{\lambda}(f_{g})$$ where $f_{g}:{}^{g}X\ra Y$, for any $g\in G$. In this description we identify $F_{\lambda}({}^{g}X)$ with $F_{\lambda}(X)$, for any $R$-module $X$ and $g\in G$. This isomorphism is used freely in the paper.
\end{rem}

We recall that if the group $G$ is torsion-free, then the functor $F_{\lambda}:\mod(R)\ra\mod(A)$ preserves right and left minimal almost split homomorphisms, Auslander-Reiten sequences and induces an injection $\ind(R)\slash G\hookrightarrow\ind(A)$, see \cite{BoGa,DoSk}.

In the paper we assume the following definition of a Galois $G$-precovering functor between additive Krull-Schmidt $K$-categories.

\begin{df}\label{0d1} Assume that $F:\CC\ra\CD$ is an additive functor between additive Krull-Schmidt $K$-categories $\CC$ and $\CD$. Assume that $G$ is a group acting freely on the isomorphism classes of indecomposable objects of $\CC$. We say that $F:\CC\ra\CD$ is a \emph{Galois $G$-precovering} if and only if the following assertions hold.
\begin{enumerate}[\rm(1)]
\item There are isomorphisms $F({}^{g}C)\cong F(C)$ and $F({}^{g}f)\cong F(f)$, for any object $C$ and morphism $f$ in $\CC$.
\item If $X,Y$ are indecomposable objects of $\CC$, then $F(X)\cong F(Y)$ implies $Y\cong {}^{g}X$, for some $g\in G$.
    \item The functor $F:\CC\ra\CD$ preserves indecomposability.
\item The functor $F$ induces the following isomorphisms of vector spaces $$\bigoplus_{g\in G}{}_{\CC}({}^{g}X,Y)\cong{}_{\CD}(F(X),F(Y))\cong\bigoplus_{g\in G}{}_{\CC}(X,{}^{g}Y),$$ for any objects $X,Y$ of $\CC$.
\end{enumerate} If a Galois $G$-precovering is a dense functor, we call it a \emph{Galois $G$-covering.} \epv
\end{df} 
Theorem \ref{0t1} implies that if the group $G$ acts freely on $\ind(R)$, then the push-down functor $F_{\lambda}$ is a Galois $G$-precovering of module categories. Moreover, if $G$ is torsion-free and $R$ is locally support-finite, then $F_{\lambda}$ is a Galois $G$-covering. However, $F_{\lambda}$ is not dense in general which leads to the following definitions, introduced in \cite{DoSk}. Let $F:R\ra A$ be a Galois covering. An indecomposable module $M\in\mod(A)$ is of the \emph{first kind} if and only if $M$ lies in the image of $F_\lambda$. Otherwise, $M$ is of the \emph{second kind}. We emphasize that \cite{DoSk} gives important characterizations of the modules of the first and the second kind for a wide class of Galois coverings. 

\subsection{Abelian categories and Krull-Gabriel dimension}


Throughout the paper we assume that any abelian category $\CC$ is \textit{skeletally small} which means that the class of all isomorphism classes of objects of $\CC$ is a set.

Assume that $\CC$ is an abelian category and $X,Y$ are objects of $\CC$. We denote by ${}_{\CC}(X,Y)$ the abelian group $\Hom_{\CC}(X,Y)$ of all \emph{homomorphisms} (or \emph{morphisms}) from $X$ to $Y$. If $f\in{}_{\CC}(X,Y)$, then $\Ker_{\CC}(f)$, $\Im_{\CC}(f)$, $\Coker_{\CC}(f)$ and $\Coim_{\CC}(f)$ denote \textit{kernel}, \textit{image}, \textit{cokernel} and \textit{coimage} of $f$ in $\CC$, respectively. We refer to \cite{Po} for precise definitions of these notions in abelian categories. 

Assume that $X_{1},\dots,X_{n}$ and $Y_{1},\dots,Y_{m}$ are objects of an abelian catefgory $\CC$. We use the standard \textit{matrix notation} for a morphism $$f\in{}_{\CC}(\bigoplus_{i=1}^{n}X_{i},\bigoplus_{j=1}^{m}Y_{j}),$$ that is, we write $f=[f_{ji}]_{i=1,\dots,n}^{j=1,\dots,m}$ where $p_{j}:\bigoplus_{k=1}^{n}Y_{k}\ra Y_{j}$ is the split epimorphism, $u_{i}:X_{i}\ra\bigoplus_{k=1}^{m}X_{k}$ is the split monomorphism and $f_{ji}=p_{j}fu_{i}$, for any $i=1,\dots,n$, $j=1,\dots,m$. We say in this case that $f$ \textit{is defined by homomorphisms $f_{ji}$}.

Assume that $\CC,\CD$ are abelian categories. A functor $F:\CC\ra\CD$ is \textit{exact} if and only if for any exact sequence $0\ra X\stackrel{f}{\rightarrow}Z\stackrel{g}{\rightarrow}Y\ra 0$ in $\CC$ the sequence $$0\ra F(X)\stackrel{F(f)}{\longrightarrow}F(Z)\stackrel{F(g)}{\longrightarrow}F(Y)\ra 0$$ is exact in $\CD$. Equivalently, the functor $F$ preserves finite limits, finite colimits, and is both left and right exact, see \cite[VIII]{McL} for details. In particular, an exact functor $F:\CC\ra\CD$ preserves kernels, images, cokernels and coimages of homomorphisms in $\CC$.

Assume that $\CC$ is an abelian category. A full subcategory $\CS$ of $\CC$ is a \textit{Serre subcategory} if and only if for any exact sequence $0\ra X\ra Z\ra Y\ra 0$ in $\CS$ we have $Z\in\CS$ if and only if $X,Y\in\CS$. Equivalently, $\CS$ is closed under subobjects, quotients and extensions. We recall that a Serre subcategory is an abelian subcategory. Important examples of Serre subcategories are given by kernels of exact functors.

Assume that $\CS$ is a Serre subcategory of $\CC$. We refer the reader to \cite{Po} for the precise definition of the \textit{quotient category} $\CC\slash\CS$. For our purposes we only need to know that the quotient category $\CC\slash\CS$ is abelian and there exists an exact \textit{quotient functor} $q_{\CS}:\CC\ra\CC\slash\CS$ such that $q_{\CS}(X)=X$, for any object $X$ of $\CC$.

Assume that $\CC$ is an abelian category. An object $S$ of $\CC$ is \textit{simple} in $\CC$ if and only if $S\neq 0$ and any subobject of $S$ in $\CC$ is either $S$ or $0$. An object $T$ of $\CC$ has \textit{finite length} in $\CC$ if and only if there exists a chain of subobjects $$0=T_{0}\subseteq T_{1}\subseteq\dots\subseteq T_{n}=T$$ of $T$ such that $T_{i+1}\slash T_{i}$ is simple in $\CC$, for any $i=0,\dots,n-1$. We call such a chain a \textit{composition series} of $T$. If $T$ has finite length, then the number $n$ is unique, called the \textit{length} of $T$ and denoted by $l_{\CC}(T)$. It is convenient to recall that if $0\ra X\ra Z\ra Y\ra 0$ is a short exact sequence in $\CC$, then $l_{\CC}(Z)=l_{\CC}(X)+l_{\CC}(Y)$ and hence $Z$ has finite length if and only if $X,Y$ have finite length.

Assume that $\CC$ is an abelian category. Following \cite{Ge2,Geigle1985}, we define the \textit{Krull-Gabriel filtration} $(\CC_{\alpha})_{\alpha}$ of $\CC$ recursively as follows: 
\begin{enumerate}[\rm(1)]
	\item $\CC_{-1}=0$,
	\item $\CC_{\alpha+1}$ is the Serre subcategory of $\CC$ formed by all objects having finite length in the quotient category $\CC\slash\CC_{\alpha}$, for any ordinal number $\alpha$,
	\item $\CC_{\beta}=\bigcup_{\alpha<\beta}\CC_{\alpha}$, for any limit ordinal $\beta$.
\end{enumerate} 

Observe that $\CC_{0}$ is the Serre subcategory of $\CC$ formed by all objects of $\CC$ of finite length and $\CC_{\alpha}$ is a Serre subcategory of $\CC_{\alpha+1}$, for any ordinal $\alpha\geq -1$.

Assume that $(\CC_{\alpha})_{\alpha}$ is the Krull-Gabriel filtration of $\CC$. Then the \textit{Krull-Gabriel dimension} $\KG(\CC)$ of $\CC$ as the smallest ordinal number $\alpha$ such that $\CC_{\alpha}=\CC$, if such a number exists. Otherwise, we set $\KG(\CC)=\infty$ and take a technical assumption that $\alpha<\infty$, for any ordinal number $\alpha$. If $\KG(\CC)\in\NN$, then the Krull-Gabriel dimension of $\CC$ is \textit{finite}. If $\KG(\CC)=\infty$, then the Krull-Gabriel dimension of $\CC$ is \textit{undefined}.

The following standard lemma is useful in the study of Krull-Gabriel dimension.

\begin{lm}\label{0l1} Assume that $\CC,\CD$ are abelian categories and $F:\CC\ra\CD$ is an exact functor.
\begin{enumerate}[\rm(1)]
  \item If $F$ is faithful, then $\KG(\CC)\leq\KG(\CD)$.
  \item If $F$ is full and dense, then $\KG(\CD)\leq\KG(\CC)$.
\end{enumerate}
\end{lm}

\begin{proof} We only give a sketch of the proof of (1), referring to \cite[Appendix B]{Kr} for details. Let $(\CC_{\alpha})_{\alpha}$ and $(\CD_{\alpha})_{\alpha}$ be the Krull-Gabriel filtrations of $\CC,\CD$, respectively. Assume that $F:\CC\ra\CD$ is exact and faithful. Then we conclude that if $0=T_{0}\subseteq T_{1}\subseteq\dots\subseteq T_{n}=T$ is a chain of subobjects of $T$ which is not a composition series of $T$ in $\CC_{\alpha}$, then $$0=F(T_{0})\subseteq F(T_{1})\subseteq\dots\subseteq F(T_{n})=F(T)$$ is a chain of subobjects of $F(T)$ which is not a composition series of $F(T)$ in $\CD_{\alpha}$, for any ordinal $\alpha$. This shows that $F(T)\in\CD_{\alpha}$ implies $T\in\CC_{\alpha}$, for any $T\in\CC$ and any ordinal $\alpha$. Hence $\KG(\CC)\leq\KG(\CD)$ by \cite[Lemma 3.2]{P4}. Analogous arguments work for (2).  
\end{proof}

\begin{cor}\label{0c1} Assume that $\CS$ is a Serre subcategory of $\CC$. Then $\KG(\CC\slash\CS)\leq\KG(\CC)$ and $\KG(\CS)\leq\KG(\CC)$. Moreover, $\CC$ has finite Krull-Gabriel dimension if and only if $\CS$ and $\CC\slash\CS$ have finite Krull-Gabriel dimensions.
\end{cor}
\begin{proof} The quotient functor $q_{\CS}:\CC\ra\CC\slash\CS$ is full, dense and exact and the inclusion functor $\CS\hookrightarrow\CC$ is faithful and exact. Hence the first part follows from Lemma \ref{0l1}. For the second part we refer to \cite[Section 1]{Ge2}. 
\end{proof}

\subsection{Finitely presented functors}

Assume that $R$ is a locally bounded $K$-category. Set $\CR=\mod(R)$ and denote by $\mod(\CR)$ the category of all contravariant $K$-linear functors from $\mod(R)$ to the category $\mod(K)$ of finite dimensional $K$-vector spaces.\footnote{In \cite{P4} the category $\mod(\CR)$ is denoted by $\CG(R)$.}  

Assume that $F,G,H\in\mod(\CR)$ and let $F\stackrel{u}{\rightarrow}G\stackrel{v}{\rightarrow}H$ be a sequence of homomorphisms of functors. Recall that this sequence is exact in $\mod(\CR)$ if and only if it induces an exact sequence $$F(X)\stackrel{u_{X}}{\rightarrow}G(X)\stackrel{v_{X}}{\rightarrow}H(X)$$ of vector spaces, for any $X\in\mod(R)$. If $F=0$, then $v:G\ra H$ is a monomorphism of functors, that is, $v_{X}:G(X)\ra H(X)$ is a monomorphism of vector spaces, for any $X\in\mod(R)$. In this case, $G$ is a subobject of $H$. If $H=0$, then $u:F\ra G$ is an epimorphism of functors, that is, $v_{X}:F(X)\ra G(X)$ is an epimorphism of vector spaces, for any $X\in\mod(R)$. In this case, $G$ is a quotient of $F$.

Let $M$ be an $R$-module. Then a \textit{contravariant hom-functor} represented by $M$ is the functor $H_{M}:\mod(R)\ra\MOD(K)$ such that $H_{M}(X)={}_{R}(X,M)$, for any $X\in\mod(R)$, and if $f\in{}_{R}(X,Y)$, then $H_{M}(f):{}_{R}(Y,M)\ra{}_{R}(X,M)$ where $H_{M}(f)(g)=gf$, for any $g\in{}_{R}(Y,M)$. The functor $H_{M}:\mod(R)\ra\MOD(K)$ is denoted by ${}_{R}(-,M)$.\footnote{The hom-functor is also denoted as ${}_{R}(*,M)$, ${}_{R}(?,M)$ etc. We agree that the domain of a hom-functor can be enlarged to $\Mod(R)$ or $\MOD(R)$.}

Recall that a homomorphism $f\in{}_{R}(M,N)$ induces a homomorphism of hom-functors ${}_{R}(-,f):{}_{R}(-,M)\ra{}_{R}(-,N)$ such that ${}_{R}(X,f):{}_{R}(X,M)\ra{}_{R}(X,N)$ is defined by ${}_{R}(X,f)(g)=fg$, for any $g\in{}_{R}(X,M)$. The Yoneda lemma implies that the function $f\mapsto{}_{R}(-,f)$ defines an isomorphism $${}_{R}(M,N)\ra{}_{\CR}({}_{R}(-,M),{}_{R}(-,N))$$ of vector spaces. Moreover, this yields $M\cong N$ if and only if ${}_{R}(-,M)\cong{}_{R}(-,N)$. 

Assume that $F\in\mod(\CR)$. The functor $F$ is \textit{finitely generated} if and only if there exists an epimorphism of functors ${}_{R}(-,N)\ra F$, for some $N\in\mod(R)$. The functor $F$ is \textit{finitely presented} if and only if there exists an exact sequence of functors $${}_{R}(-,M)\xrightarrow{{}_{R}(-,f)}{}_{R}(-,N)\ra F\ra 0,$$ for some $M,N\in\mod(R)$ and $R$-module homomorphism $f:M\ra N$. This means that $F\cong\Coker{}_{R}(-,f)$ and thus $F(X)$ is isomorphic to the cokernel of the map ${}_{R}(X,f):{}_{R}(X,M)\ra{}_{R}(X,N)$.

The full subcategory of $\mod(\CR)$, formed by all finitely presented functors, is denoted as $\CF(R)$.\footnote{Although this notation is unusual, we would like to remain consistent with the notation from our previous papers, mainly with \cite{P4,P6,J-PP1}.} Obviously ${}_{R}(-,M)\in\CF(R)$, for any $M\in\mod(R)$. Moreover, the functor ${}_{R}(-,M)$ is a projective object of the category $\CF(R)$ and any projective object of $\CF(R)$ is a hom-functor, see \cite{Au0}.

Assume that $R$ is a locally bounded $K$-category. Following \cite{Au}, we define the \textit{Krull-Gabriel dimension of $R$} as the Krull-Gabriel dimension $\KG(\CF(R))$ of the category $\CF(R)$. We denote $\KG(\CF(R))$ as $\KG(R)$. 

We also need relative versions of the above definitions. Let $\CX$ be a full subcategory of $\mod(R)$. If $M\in\CX$, then ${}_{\CX}(-,M):\CX\ra\MOD(K)$ denotes the functor ${}_{R}(-,M)$ restricted to the category $\CX$. If $M,N\in\CX$ are $R$-modules, then any $R$-homomorphism $f:M\ra N$ induces a homomorphism of functors ${}_{\CX}(-,f):{}_{\CX}(-,M)\ra{}_{\CX}(-,N)$. A functor $F:\CX^{\op}\ra\mod(K)$ is \textit{finitely generated} if and only if there exists an exact sequence of functors ${}_{\CX}(-,N)\ra F\ra 0$, for some $N\in\CX$. A functor $F:\CX^{\op}\ra\mod(K)$ is \textit{finitely presented} if and only if there exists an exact sequence of functors $${}_{\CX}(-,M)\xrightarrow{{}_{\CX}(-,f)}{}_{\CX}(-,N)\ra F\ra 0,$$ for some $M,N\in\CX$ and $R$-module homomorphism $f:M\ra N$. We denote by $\CF(\CX)$ the category of all finitely presented functors $\CX^{\op}\ra\mod(K)$. 

Assume that $\CX$ is a full subcategory of $\mod(R)$ which is closed under isomorphisms and direct summands. Then $\CX$ is \textit{contravariantly finite} \cite{AuRe} if and only if any module $M\in\mod(R)$ has a \textit{right $\CX$-approximation}. This means that there exists a module $X_{M}\in\CX$ and $R$-homomorphism $\alpha:X_{M}\ra M$ such that for any module $X\in\CX$ and any $R$-homomorphism $\beta:X\ra M$ there is $\gamma:X\ra X_{M}$ such that $\alpha\gamma=\beta$. Equivalently, there exists an exact sequence of functors ${}_{\CX}(-,X_{M})\ra {}_{\CX}(-,M)\ra 0$. Thus $\CX$ is contravariantly finite if and only if the functor ${}_{\CX}(-,M)$ is finitely generated, for any $M\in\mod(R)$. If $\CX$ is contravariantly finite, then the category $\CF(\CX)$ is abelian \cite{Ge2}.

A full subcategory $\CX$ of $\mod(R)$ is a contravariantly finite if and only if for any functor $F\in\CF(R)$ the restriction $F|_{\CX}:\CX\ra\mod(K)$ is a finitely presented functor, see for example \cite{Au}. Then we define the \textit{restriction functor} $r_{\CX}:\CF(R)\ra\CF(\CX)$ such that $r_{\CX}(F)=F|_{\CX}$, for any $F\in\CF(R)$. The restriction functor $r_{\CX}:\CF(R)\ra\CF(\CX)$ is exact, full and dense \cite[Section 2]{Ge2}. Hence we conclude from \ref{0l1} that $\KG(\CX)\leq\KG(R)$. We apply this fact in a very useful lemma below. First we need to recall some notions. 

Assume that $R$ is a locally bounded $K$-category. A full subcategory $B$ of $R$ is \textit{convex} if and only if for any $n\geq 1$ and objects $x,z_{1},\dots,z_{n},y$ of $R$ the following condition is satisfied: if $x,y$ are objects of $B$ and the vector spaces of morphisms $$R(x,z_{1}), R(z_{1},z_{2}),\dots,R(z_{n-1},z_{n}),R(z_{n},y)$$ are nonzero, then $z_{1},\dots,z_{n}$ are objects of $B$. If $R=\und{KQ}/I$ is a path $K$-category of a bound quiver $(Q,I)$ and $B$ is a convex subcategory of $R$, then $B=\und{KQ'}/I'$, for some convex subquiver $Q'$ of the quiver $Q$. Recall that a full subquiver $Q'$ of $Q$ is \textit{convex} if and only if for any path $c_{1}\dots c_{n}$ from the vertex $x$ to the vertex $y$ in $Q$ such that $x,y\in Q'_{0}$ we have $s(c_{i})\in Q'_{0}$, for $i=1,\dots,n-1$.

Assume that $R$ is a locally bounded $K$-category. If $B$ is a convex subcategory of $R$, then we denote by $$\CE_{B}:\mod(B)\ra\mod(R)\subseteq\Mod(R)$$ the \textit{functor of extension by zeros}. Recall that if $M\in\mod(B)$, then $\CE_{B}(M)$ coincides with $M$ on $B$ and for $x\notin\ob(B)$ we have $\CE_{B}(M)(x)=0$. The convexity of $B$ yields that $\CE_{B}(M)$ is an $R$-module and that any $R$-module whose support is contained in $B$ is a $B$-module. The functor $\CE_{B}$ is fully faithful and exact, hence we often identify $M$ with $\CE_{B}(M)$.

\begin{lm}\label{0l2} Assume that $R$ is a locally bounded $K$-category. 
\begin{enumerate}[\rm(1)]
  \item If $B$ is a convex subcategory of $R$, then the category $\mod(B)$ is a contravariantly finite subcategory of $\mod(R)$.
  \item If $B$ is a factor category of $R$, then the category $\mod(B)$ is a contravariantly finite subcategory of $\mod(R)$.
\end{enumerate} In both cases, $\KG(B)\leq\KG(R)$.
\end{lm}

\begin{proof} (1) Assume that $B$ is a convex subcategory of $R$. Then any $B$-module is also an $R$-module via the extension by zeros. Let $M\in\mod(R)$ and assume that $M'$ is the $R$-module generated by all modules of the form $\Im(f)\subseteq M$ where $f:X\ra M$ is an $R$-homomorphism and $X\in\mod(B)$. Observe that, for any such $f:X\ra M$, $\Im(f)$ is an $R$-module whose support is contained in $B$, and thus so is $M'$. The convexity of $B$ implies that $M'$ is a $B$-module. This shows that the inclusion $M'\hookrightarrow M$ of $R$-modules is the right $\mod(B)$-approximation of $M$, because we have $\Im(\beta)\subseteq M'$, for any $X\in\mod(B)$ and $\beta:X\ra M$.

(2) Assume that $B$ is a factor category of $R$. Hence there is a $K$-linear epimorphism $\pi:R\ra B$ of locally bounded $K$-categories. The epimorphism $\pi$ induces a fully faithful and exact functor $\iota:\mod(B)\ra\mod(R)$ such that $\iota(X)=X\circ\pi$. Let $M\in\mod(R)$. Similarly as in (1), assume that $M'$ is the $R$-module generated by all modules of the form $\Im(f)\subseteq M$ where $f:\iota(X)\ra M$ is an $R$-homomorphism and $X\in\mod(B)$. We show that $\Im(f)$ is a $B$-module. Indeed, since $\iota(X)$ is annihilated by $\Ker(\pi)$, we have $$\Im(f)\cdot\Ker(\pi)=f(\iota(X))\cdot\Ker(\pi)=f(\iota(X)\cdot\Ker(\pi))=f(0)=0,$$ so $\Im(f)$ is a $B$-module. This implies that $M'$ is a $B$-module, because it is an $R$-module annihilated by $\Ker(\pi)$. As in (1), this shows that the inclusion $M'\hookrightarrow M$ of $R$-modules is the right $\mod(B)$-approximation of $M$.
\end{proof}

\section{General covering theory of functor categories}
In this section we prove our first main results. We start with recalling the existence of the \emph{tensor product bifunctor} \cite{FPN,Mi} for categories of modules over arbitrary small $K$-categories, abelian $K$-categories in particular. Then we determine in Theorem \ref{t1} the left and the right adjoint to the pull-up functor. This fact is already proved in \cite[Proposition 6.1]{Bu}, but it seems not to be widely recognized. As a direct consequence of Theorem \ref{t1} we get Corollary \ref{c1}, describing the left and the right adjoint to the pull-up along $F_{\lambda}:\mod(R)\ra\mod(A)$, denoted by $\Psi$, where $F:R\ra A$ is the given Galois $G$-covering. We believe that Corollary \ref{c1} provides the natural setting for studying Galois coverings of functor categories. Then we prove the crucial Theorem \ref{t2}, showing that the left and the right adjoints to $\Psi$ coincide on the category of finitely presented functors, which implies that the restricted functor is exact. Since it is also faithful, we deduce in Theorem \ref{t3} the main result stating that the Galois coverings do not increase the Krull-Gabriel dimension. This assertion means that if $F:R\ra A$ is a Galois covering of locally bounded $K$-categories $R$ and $A$, then $\KG(R)\leq\KG(A)$.

\subsection{Tensor products and adjoints to the pull-up functor}
Assume that $\CA$ is a small $K$-category. We denote by $$\CA(-,-):\CA^{\op}\times\CA\ra\MOD(K)$$ the \emph{hom-bifunctor}. Assume that $\CA,\CB$ are small $K$-categories (note that $\CA$ and $\CB$ may be abelian). We denote by $\CA\lMOD$ the category of all \emph{left $\CA$-modules}, that is, the category consisting of covariant functors $\CA\ra\MOD(K)$ as objects and natural transformations as morphisms. The category of all \emph{right $\CA$-modules}, denoted by $\MOD(\CA)$, is the category consisting of covariant functors $\CA^{\op}\ra\MOD(K)$ (equivalently, contravariant functors $\CA\ra\MOD(K)$) as objects and natural transformations as morphisms. Since a left $\CA$-module is a right $\CA^{\op}$-module in a natural way, we sometimes identify $\MOD(\CA^{\op})$ with $\CA\lMOD$. An \emph{$\CA$-$\CB$-bimodule} is a functor $\CB^{\op}\times\CA\ra\MOD(K)$, that is, a functor which is contravariant in the first variable and covariant in the second.

We usually denote objects of $\CA$ with lowercase letters whereas objects of $\MOD(\CA)$ or $\CA\lMOD$ with uppercase letters. Observe that $\CA(-,-)$ is an $\CA$-$\CA$-bimodule and if $a\in\CA$, then $\CA(a,-)\in\CA\lMOD$ and $\CA(-,a)\in\MOD(A)$. We also denote by $${}_{\CA}(-,-):\MOD(\CA)^{\op}\times\MOD(\CA)\ra\MOD(K)$$ the \emph{hom-bifunctor}. Hence ${}_{\CA}(-,-)=\MOD(\CA)(-,-)$ but we do not use the latter notation. Note that ${}_{\CA^{\op}}(-,-)$ is the hom-bifunctor defined for the category of left $\CA$-modules. We emphasize that the notation introduced above is consistent with that for modules over locally bounded $K$-categories, see Section 2.

We recall from \cite{FPN}, see also \cite{Mi}, that there exists a \emph{tensor product bifunctor} $$-\otimes_{\CA}-:\MOD(\CA)\times\CA\lMOD\ra\MOD(K)$$ such that, for any $\CA$-$\CB$-bimodule ${}_{\CA}M_{\CB}$ and $\CB$-$\CA$-bimodule ${}_{\CB}N_{\CA}$:
\begin{itemize}
  \item the functor $-\otimes{}_{\CA}M_{\CB}:\MOD(\CA)\ra\MOD(\CB)$ is the left adjoint to the functor ${}_{\CB}({}_{\CA}M_{\CB},-):\MOD(\CB)\ra\MOD(\CA)$,
  \item the functor ${}_{\CB}N_{\CA}\otimes_{\CA}-:\CA\lMOD\ra\CB\lMOD$ is the left adjoint to the functor ${}_{\CB^{\op}}({}_{\CB}N_{\CA},-):\CB\lMOD\ra\CA\lMOD$.
\end{itemize} Moreover, for any $M\in\MOD(\CA)$, $N\in\CA\lMOD$ and $a\in\CA$ we have the following natural \emph{co-Yoneda isomorphisms}: $M\otimes_{\CA}\CA(a,-)\cong M(a)$ and $\CA(-,a)\otimes_{\CA} N\cong N(a)$. Equivalently: $$M\otimes_{\CA}\CA(-,-)\cong M\textnormal{ and }\CA(-,-)\otimes_{\CA} N\cong N.$$ As the left adjoints to the appropriate hom-functors, the tensor product functors $-\otimes{}_{\CA}M_{\CB}$ and ${}_{\CB}N_{\CA}\otimes_{\CA}-$ are unique up to natural equivalence. Hence we conclude that for locally bounded $K$-categories $\CA$ and $\CB$ these functors coincide with the usual tensor products for modules, see also Section 2.1.

Assume that $G:\CA\ra\CB$ is a $K$-linear functor. The \emph{pull-up functor} along $G$ is the functor $G_{\bullet}:\MOD(\CB)\ra\MOD(\CA)$ defined as $(-)\circ G^{\op}$. This functor is often called the \emph{restriction functor}, but we view the naming coming from covering theory more suitable in general. The following theorem describes the left and the right adjoint functors to the pull-up functor, see \cite[Proposition 6.1]{Bu}.

\begin{thm} \label{t1}
 Assume that $G:\CA\ra\CB$ is a $K$-linear functor. The pull-up functor $G_{\bullet}:\MOD(\CB)\ra\MOD(\CA)$ has the left adjoint given by the functor $$G_{L}=?\otimes_{\CA}\CB(-,G(*)):\MOD(\CA)\ra\MOD(\CB)$$ and the right adjoint given by the functor $$G_{R}={}_{\CA}(\CB(G(*),-),?):\MOD(\CA)\ra\MOD(\CB).$$
\end{thm}

\begin{proof}  Observe that the bifunctor $\CB(-,G(*)):\CB^{\op}\times\CA\ra\MOD(K)$ is an $\CA$-$\CB$-bimodule and the bifunctor $\CB(G(*),-):\CA^{\op}\times\CB\ra\MOD(K)$ is a $\CB$-$\CA$-bimodule, and hence the functors $G_{L}$ and $G_{R}$ are well-defined.

Assume that $T\in\MOD(\CA)$ and $U\in\MOD(\CB)$. We show that $(G_L,G_\bullet)$ and $(G_\bullet,G_R)$ are adjoint pairs. Indeed, the hom-tensor adjunction gives the following natural isomorphism: $${}_{\CB}(G_{L}(T),U)={}_{\CB}(T\otimes_{\CA}\CB(-,G(*)),U)\cong{}_{\CA}(T,{}_{\CB}(\CB(-,G(*)),U)).$$ The Yoneda lemma yields the natural isomorphism ${}_{\CB}(\CB(-,G(*)),U)\cong U(G(*))=G_\bullet(U)$ and consequently we get the natural isomorphism ${}_{\CB}(G_{L}(T),U)\cong{}_{\CA}(T,G_\bullet(U))$. This shows that $G_L$ is the left adjoint to $G_\bullet$. Similarly, the hom-tensor adjunction yields the following natural isomorphism:

$${}_{\CB}(U,G_R(T))={}_{\CB}(U,{}_{\CA}(\CB(G(*),-),T))\cong{}_{\CA}(U\otimes{}_{\CB\snull}\CB(G(*),-),T).$$ The co-Yoneda lemma gives the natural isomorphism $U\otimes{}_{\CB\snull}\CB(G(*),-)\cong U(G(*))=G_\bullet(U)$ and so we get the natural isomorphism ${}_{\CB}(U,G_R(T))\cong{}_{\CA}(G_\bullet(U),T)$. This shows that $G_R$ is the right adjoint to $G_\bullet$.
\end{proof}

\begin{rem}\label{r1} Let $A$ be a locally bounded $K$-category, $B$ a full subcategory of $A$ and $\iota:B\hookrightarrow A$ the inclusion functor. Then the pull-up functor $\iota_{\bullet}:\MOD(A)\ra\MOD(B)$ along $\iota$ is, rightly, called the \emph{restriction functor}. Since $$B(-,\iota(*))=B(-,*):A^{\op}\times B\ra\MOD(K)$$ and $$B(\iota(*),-)=B(*,-):B^{\op}\times A\ra\MOD(K),$$ we conclude that $$\iota_{L}=?\otimes_{B}B(-,*)\textnormal{ and }\iota_{R}={}_{B}(B(*,-),?).$$ It is known that both functors $\iota_{L},\iota_{R}$ are fully faithful and hence they are called the \emph{embedding functors}. We use the standard notation for these functors, namely $\res_{B}:=\iota_{\bullet}$, $T_{B}:=\iota_{L}$ and $L_{B}:=\iota_{R}$. Assume that $B$ is finite and $\ob(B)=\{x_{1},x_{2},\dots,x_{n}\}$. Then we have an isomorphism $$B(-,*)\cong B(-,x_1)\oplus\dots\oplus B(-,x_n)\cong P_{x_{1}}\oplus\dots\oplus P_{x_{n}}$$ of right $A$-modules and an isomorphism $$B(*,-)\cong B(x_1,-)\oplus\dots\oplus B(x_n,-)\cong D(I_{x_{1}})\oplus\dots\oplus D(I_{x_{n}})$$ of left $A$-modules. If $A$ is also finite, we may view both $A$ and $B$ as bound quiver $K$-algebras. Then there is an idempotent $e\in A$ such that $B=eAe$ and we set the notation: $\res_{B}=\res_{e}=(-)e$, $T_{B}=T_{e}=(-)\otimes_{B}eA$ and $L_{B}=L_{e}={}_{B}(Ae,-)$, consistently with \cite[I.6]{AsSiSk}. \epv
\end{rem}

\begin{rem}\label{r1.1} Theorem \ref{t1} necessarily describes the case of classical Galois covering functors. Namely, assume that $F:R\ra A$ is a Galois $G$-covering. According to \ref{t1}, the pull-up functor $F_{\bullet}:\MOD(A)\ra\MOD(R)$, defined as $(-)\circ F^{\op}$, has the left adjoint given by the functor $$F_{L}=?\otimes_{R}A(-,F(*)):\MOD(R)\ra\MOD(A)$$ and the right adjoint given by the functor $$F_{R}={}_{}(A(F(*),-),?):\MOD(R)\ra\MOD(A).$$  Then it follows from the definition of $F:R\ra A$ that $F_{L}\cong F_{\lambda}$ and $F_{R}\cong F_{\rho}$. Indeed, assume that $M$ is an $R$-module and take $a\in\ob(A)$. Since $F$ is surjective on objects, we get that $a=F(x)$, for some $x\in\ob(R)$. Therefore we obtain: $$F_{L}(M)(a)=F_{L}(M)(F(x))=M\otimes_{R}A(F(x),F(*))\cong M\otimes_{R}(\bigoplus_{g\in G}R(gx,*))\cong$$$$\cong\bigoplus_{g\in G}M\otimes_{R}R(gx,*)\cong\bigoplus_{g\in G}M(gx)=F_{\lambda}(M)(F(x))=F_{\lambda}(M)(a),$$ by the isomorphism $\bigoplus_{g\in G}R(gx,*)\cong A(F(x),F(*))$ an the co-Yoneda lemma. Similar calculations work in the case of $F_{R}\cong F_{\rho}$. We refer to Section 2 of \cite{Bu} for the details, as well as more general considerations.
\end{rem}

Let $F:R\ra A$ be a Galois $G$-covering. Applying the above theorem to the push-down functor $F_{\lambda}:\mod(R)\ra\mod(A)$, we get the following result which we believe initiates the \emph{genral Galois covering theory of functor categories}. In the sequel we set $\CR=\mod(R)$ and $\CA=\mod(A)$. 

\begin{cor}\label{c1} Assume that $F:R\ra A$ is a Galois $G$-covering and let $$\Psi:=(F_{\lambda})_{\bullet}:\MOD(\CA)\ra\MOD(\CR)$$ be the pull-up functor along $F_{\lambda}:\mod(R)\ra\mod(A)$. The functor $$\Phi:=?\otimes_{\CR}[{}_{A}(-,F_\lambda(*))]:\MOD(\CR)\ra\MOD(\CA)$$ is the left adjoint to $\Psi$ and hence right exact. The functor $$\Theta:={}_{\CR}({}_{A}(F_{\lambda}(*),-),?):\MOD(\CR)\ra\MOD(\CA)$$ is the right adjoint to $\Psi$ and hence left exact.
\end{cor}

\begin{proof} Since $\Phi=(F_{\lambda})_L$ and $\Theta:=(F_{\lambda})_R$, the assertion follows from a direct application of Theorem \ref{t1} to the functor $F_{\lambda}:\CR\ra\CA$. Note that the bifunctor ${}_{A}(-,F_\lambda(*))$ is a $\CR$-$\CA$-bimodule whereas ${}_{A}(F_{\lambda}(*),-)$ is an $\CA$-$\CR$-bimodule.
\end{proof}

\subsection{Galois coverings do not increase Krull-Gabriel dimension}

Assume that $R$ is a locally bounded $K$-category. We write $C\preceq R$ if $C$ is a full subcategory of $R$. Assume $C\preceq R$ and define $\wt{C}_{\triangleleft}$, $\wt{C}_{\triangleright}$ and $\wt{C}$ as the full subcategories of $R$ such that $$\ob(\wt{C}_{\triangleleft})=\{x\in\ob(R)\mid \exists_{c\in\ob(C)\snull}R(c,x)\neq 0\},\textnormal{ } \ob(\wt{C}_{\triangleright})=\{x\in\ob(R)\mid \exists_{c\in\ob(C)\snull}R(x,c)\neq 0\}$$ and $\ob(\wt{C})=\ob(\wt{C}_{\triangleleft})\cup\ob(\wt{C}_{\triangleright})$. Observe that $\wt{C}_{\triangleleft},\wt{C}_{\triangleright},\wt{C}$ are finite full subcategories of $R$ containing $C$.

Assume that $X\in\Mod(R)$, $M\in\mod(R)$ and $C,D\preceq R$ are subcategories such that $\supp(M)\subseteq C\preceq D\preceq R$. Assume that $h:\res_{D}(X)\ra M$ is a $D$-homomorphism. We define a collection of $K$-linear maps $$h^{[0]}=(h^{[0]}_{a}:X(a)\ra M(a))_{a\in\ob(R)}$$ in such a way that $h^{[0]}_{a}=h_{a}$ if $a\in\supp(M)$ and $h^{[0]}_{a}=0$ otherwise. Obviously, $h^{[0]}$ looks like an extension of $h$ by zeros, but it is not an $R$-homomorphism $X\ra M$ in general. The following lemma shows that this is the case if $D=\wt{C}_{\triangleleft}$.


\begin{lm}\label{l1} Assume that $R$ is a locally bounded $K$-category, $X\in\Mod(R)$, $M\in\mod(R)$ and $\supp(M)\subseteq C\preceq R$. Set $D=\wt{C}_{\triangleleft}$. If $h:\res_{D}(X)\ra M$ is a $D$-homomorphism, then $h^{[0]}:X\ra M$ is an $R$-homomorphism.
\end{lm}
\begin{proof} We have to show that for any $a,b\in\ob(R)$ and any morphism $\sigma\in R(a,b)$ the following diagram $$\xymatrix{X(b)\ar[rr]^{X(\sigma)}\ar[d]^{h_{b}^{[0]}}&&{}X(a)\ar[d]^{h_{a}^{[0]}}\\ M(b)\ar[rr]^{M(\sigma)}&& M(a)}$$ commutes, i.e. $h_{a}^{[0]}X(\sigma)=M(\sigma)h_{b}^{[0]}$. We consider several cases. First, observe that if $a\notin\ob(D)$, then $a\notin\supp(M)$, because $D$ contains $C$. This yields $h_{a}^{[0]}=0$ and $M(\sigma)=0$, so the diagram commutes. Thus assume that $a\in\ob(D)$. If $b\in\ob(D)$, then commutativity of the diagram follows from the fact that $h$ is a $D$-homomorphism. Hence assume that $b\notin\ob(D)$. This yields $M(b)=0$ and thus $h_{b}^{[0]}=0$. If $a\notin\ob(C)$, then $M(a)=0$ and $h_{a}^{[0]}=0$, so the diagram commutes. Finally, if $a\in\ob(C)$, then the definition of $D=\wt{C}_{\triangleleft}$ yields $\sigma=0$, because $b\notin\ob(D)$. Hence $X(\sigma)=0$ and since $h_{b}^{[0]}=0$, the diagram also commutes in this case.
\end{proof}

The above lemma is applied in the following Proposition \ref{p1}. This proposition is an important ingredient of the proof of the subsequent Theorem \ref{t2}, stating that the functors $\Phi,\Theta:\MOD(\CR)\ra\MOD(\CA)$ restrict to the categories of finitely presented functors and that these restrictions coincide. Theorem \ref{t2} implies a crucial fact that the functor $\Phi:\CF(R)\ra\CF(A)$ is exact. 

\begin{prop}\label{p1} Assume that $R$ is a locally bounded $K$-category and let $X\in\Mod(R)$. Assume that $M,N\in\mod(R)$ and $f:M\ra N$, $\alpha:X\ra N$ are $R$-homomorphisms. If $\Im{}_{R}(*,\alpha)\subseteq\Im{}_{R}(*,f)$, then $\alpha$ factorizes through $f$.
\end{prop}

\begin{proof} Observe that the condition $\Im{}_{R}(*,\alpha)\subseteq\Im{}_{R}(*,f)$ states that for any $L\in\mod(R)$ and any $R$-homomorphism $\gamma:L\ra X$ there is $h:L\ra M$ such that $\alpha\gamma=fh$. We show that this yields $\alpha$ factorizes through $f$. For that purpose we construct a suitable full subcategory $D$ of $R$, a module $L\in\mod(R)$ and an $R$-homomorphism $\gamma:L\ra X$ such that $\gamma_{x}$ is an identity if $x\in\ob(D)$, $\alpha_{x}=0$ if $x\notin\ob(D)$ and for any $h:L\ra M$ with $\alpha\gamma=fh$, the homomorphism $\res_{D}(h)^{[0]}:X\ra M$ exists and provides the factorization.

Let $C$ be the full subcategory of $R$ such that $\ob(C)=\supp(M)\cup\supp(N)$. Set $D=\wt{C}_{\triangleleft}$ and $\wt{X}=\res_{D}(X)$. Let $R(*,-)$ be the hom-bifunctor with $*\in R$ and $-\in D$. Denote by $\varphi:\wt{X}\otimes_{D}R(*,-)\ra X$ the $R$-homomorphism such that $\varphi(x\otimes f)=x\cdot f$ where $x\in\wt{X}$ is treated as an element of the module $X$. Note that for $r\in\ob(R)$ and $a\in\ob(D)$ we have $R(r,a)\neq 0$ only if $r\in D_{\triangleright}$, so the module $\wt{X}\otimes_{D}R(*,-)$ is finite dimensional with $\supp(\wt{X}\otimes_{D}R(*,-))\subseteq D_{\triangleright}$. Moreover, if $f\in R(r,a)$ with $r,a\in\ob(D)$, then $\varphi(x\otimes f)=\varphi((x\cdot f)\otimes 1_{D})=x\cdot f$  which shows that $\res_{D}(\varphi)$ is an isomorphism. In other words, $$\varphi_{a}:\wt{X}\otimes_{D}R(a,-)\ra X(a)=\wt{X}(a)$$ is an isomorphism of vector spaces, for any $a\in\ob(D)$. We conclude that there is a module $L\in\mod(R)$ and an $R$-isomorphism $\iota:L\ra\wt{X}\otimes_{D}R(*,-)$ such that if $\gamma=\varphi\iota:L\ra X$, then $\gamma_{a}$ is an identity, for any $a\in\ob(D)$.\footnote{The module $L$ can be obtained by appropriate base changes of domains and codomains of linear maps on the arrows of $\wt{X}\otimes_{D}R(*,-)$.} In particular, we get $\res_{D}(L)=\res_{D}(X)=\wt{X}$.

We know by assumption that there is an $R$-homomorphism $h:L\ra M$ such that $\alpha\gamma=fh$. Observe that $$\res_{D}(h):\res_{D}(L)=\res_{D}(X)=\wt{X}\ra M$$ is a $D$-homomorphism with $D=\wt{C}_{\triangleleft}$ and $\supp(M)\subseteq C$. Thus by Lemma \ref{l1} there exists an $R$-homomorphism $h^{'}=\res_{D}(h)^{[0]}:X\ra M$ and we show that $\alpha=fh^{'}$. Indeed, if $x\in\ob(D)$, then $$(fh')_{x}=f_{x}h'_{x}=f_{x}h_{x}=(fh)_{x}=(\alpha\gamma)_{x}=\alpha_{x}\gamma_{x}=\alpha_{x}.$$ On the other hand, if $x\notin\ob(D)$, then $\alpha_{x}=0$ and $h'_{x}=0$ since $x$ is not contained in $\supp(N)\cup\supp(M)$. Therefore $$(fh')_{x}=f_{x}h'_{x}=0=\alpha_{x}$$ and the proof is complete.
\end{proof}


\begin{thm}\label{t2} Assume that $F:R\ra A$ is a Galois $G$-covering. Then the following assertions hold.
\begin{enumerate}[\rm(1)]
  \item We have $\Phi(\CF(R))\subseteq\CF(A)$, $\Theta(\CF(R))\subseteq\CF(A)$ and $\Phi|_{\CF(R)}\cong\Theta|_{\CF(R)}$. The functors $\Phi,\Theta:\CF(R)\ra\CF(A)$ are exact functors.
  \item The functors $\Phi,\Theta:\CF(R)\ra\CF(A)$ are faithful.
\end{enumerate}
\end{thm}

\begin{proof} $(1)$ Assume that $M\in\mod(R)$. Viewing the right $\CR$-module ${}_{R}(*,M)\in\CF(R)$ as the representable functor in the co-Yoneda isomorphism, we get the following natural isomorphism: $$\Phi({}_{R}(*,M))={}_{R}(*,M)\otimes_{\CR}[{}_{A}(-,F_\lambda(*))]\cong{}_{A}(-,F_\lambda(M)).$$ The same works if we replace the module $M$ with any $R$-module homomorphism $f$. It follows that if $T=\Coker_{R}(*,f)$, then $\Phi(T)\cong\Coker_{A}(-,F_{\lambda}(f))$, because $\Phi$ is right exact. Hence $\Phi(\CF(R))\subseteq\CF(A)$. To show that $\Theta(\CF(R))\subseteq\CF(A)$ as well, we prove that there is a natural equivalence of functors $\varphi^{T}:\Phi(T)\ra\Theta(T)$, for any $T\in\CF(R)$. Moreover, we show that this equivalence extends to a natural equivalence $$\varphi=(\varphi^{T})_{T\in\CF(R)}:\Phi|_{\CF(R)}\ra\Theta|_{\CF(R)}.$$ In what follows we fix a unique projective presentation, for any finitely presented functor $T\in\CF(R)$. Assume that $f:M\ra N\in\mod(R)$ and set $T=\Coker_{R}(*,f)$, that is, we have an exact sequence of functors of the form $${}_{R}(*,M)\xrightarrow{{}_{R}(*,f)}{}_{R}(*,N)\xrightarrow{\pi} T\ra 0.$$ Then we obtain $$\Phi(T)=\Coker_{A}(-,F_{\lambda}(f))\cong \Coker_{A}(-,F_{\rho}(f))\cong\Coker_{R}(F_{\bullet}(-),f)$$ since $F_{\lambda}=F_{\rho}$ on $\mod(R)$. Furthermore, we get $$\Theta(T)={}_{\CR}({}_{A}(F_{\lambda}(*),-),T)\cong {}_{\CR}({}_{R}(*,F_{\bullet}(-)),T)$$ and we show below that there exists a natural isomorphism of functors $$\Coker_{R}(F_{\bullet}(-),f)\cong{}_{\CR}({}_{R}(*,F_{\bullet}(-)),T).\footnote{Observe that $F_{\bullet}(\mod(A))\subseteq\Mod(R)$ so we cannot conclude that $\Coker_{R}(F_{\bullet}(-),f)\cong T(F_{\bullet}(-))$ nor can we apply the Yoneda lemma to state that ${}_{\CR}({}_{R}(*,F_{\bullet}(-)),T)\cong T(F_{\bullet}(-))$.}$$ Indeed, let us define the natural transformation of functors $$\varphi^{T}=(\varphi^{T}_{X})_{X\in\mod(A)}:\Coker_{R}(F_{\bullet}(-),f)\ra{}_{\CR}({}_{R}(*,F_{\bullet}(-)),T)$$ in the following way. Assume that $X\in\mod(A)$ and $\alpha:F_{\bullet}(X)\ra N$ is an $R$-homomorphism. We define the $K$-linear map $$\varphi^{T}_{X}:\Coker_{R}(F_{\bullet}(X),f)\cong\frac{{}_{R}(F_{\bullet}(X),N)}{\Im{}_{R}(F_{\bullet}(X),f)}\ra{}_{\CR}({}_{R}(*,F_{\bullet}(X)),T)$$ by the formula: $$\varphi^{T}_{X}(\alpha+\Im{}_{R}(F_{\bullet}(X),f))=\pi\circ{}_{R}(*,\alpha).$$ First we show that $\varphi^{T}_{X}$ is well-defined. Indeed, assume that we have $\alpha-\alpha'\in\Im{}_{R}(F_{\bullet}(X),f)$, for some $\alpha':F_{\bullet}(X)\ra N$. It follows that $\alpha-\alpha'=f\beta$, for some $\beta:F_{\bullet}(X)\ra M$ and thus $$\Im{}_{R}(*,\alpha-\alpha')=\Im{}_{R}(*,f\beta)\subseteq\Im{}_{R}(*,f)=\Ker(\pi).$$ Hence $0=\pi\circ{}_{R}(*,\alpha-\alpha')=\pi\circ{}_{R}(*,\alpha)-\pi\circ{}_{R}(*,\alpha')$ which yields $\pi\circ{}_{R}(*,\alpha)=\pi\circ{}_{R}(*,\alpha')$ and so $\varphi^{T}_{X}$ is a well-defined $K$-linear map.

We show that $\varphi^{T}$ is a natural transformation. Assume that $g:X\ra Y\in\mod(A)$ and observe that for any $R$-homomorphism $\alpha:F_{\bullet}(Y)\ra N$ we have: $$\varphi^{T}_{X}(\alpha F_{\bullet}(g)+\Im_{R}(F_{\bullet}(X),f))=\pi\circ {}_{R}(*,\alpha F_{\bullet}(g))=$$$$=(\pi\circ {}_{R}(*,\alpha))\circ {}_{R}(*,F_{\bullet}(g))=\varphi^{T}_{Y}(\alpha+\Im_{R}(F_{\bullet}(Y),f))\circ{}_{R}(*,F_{\bullet}(g)).$$ It is easy to see that these calculations yield $\varphi^{T}$ is natural. Now we show that $\varphi^{T}$ is a natural equivalence. Assume that $X\in\mod(A)$. We prove that $\varphi^{T}_{X}$ is an isomorphism of vector spaces. Let $\eta:{}_{R}(*,F_{\bullet}(X))\ra T$ be a homomorphism of functors. Since the functor ${}_{R}(*,F_{\bullet}(X))$ is projective, we conclude that there is an $R$-homomorphism $\alpha:F_{\bullet}(X)\ra N$ such that $\eta=\pi\circ{}_{R}(*,\alpha)$. This yields $\varphi^{T}_{X}$ is an epimorphism. To show that $\varphi^{T}_{X}$ is a monomorphism as well, assume that $\alpha:F_{\bullet}(X)\ra N$ is an $R$-homomorphism. Observe that if $\pi\circ{}_{R}(*,\alpha)=0$, then $$\Im_{R}(*,\alpha)\subseteq\Ker(\pi)=\Im_{R}(*,f)$$ and Proposition \ref{p1} yields $\alpha$ factorizes through $f$, so $\alpha+\Im{}_{R}(F_{\bullet}(X),f))=0$. This shows that $\varphi^{T}_{X}$ is a monomorphism and, consequently, an isomorphism. 

We conclude that $\Phi(T)\cong\Theta(T)$ and thus $\Theta(\CF(R))\subseteq\CF(A)$, because $\Phi(\CF(R))\subseteq\CF(A)$, as shown earlier. We shall denote the functors $\Phi|_{\CF(R)},\Theta|_{\CF(R)}:\CF(R)\ra\CF(A)$ as $\Phi,\Theta:\CF(R)\ra\CF(A)$, respectively.

In the final step of the proof we show that $\varphi=(\varphi^{T})_{T\in\CF(R)}:\Phi\ra\Theta$ is a natural transformation of functors. Thus assume that  $T_{i}=(T_{i},f_{i},\pi_{i})$ where $f_{i}:M_{i}\ra N_{i}$, for $i=1,2$. That is, we have exact sequences of functors of the form $${}_{R}(*,M_{i})\xrightarrow{{}_{R}(*,f_{i})}{}_{R}(*,N_{i})\xrightarrow{\pi} T_{i}\ra 0.$$ Let $\iota:T_{1}\ra T_{2}$ be a morphism of functors. We have to show that the following diagram commutes $$\xymatrix{\Coker_{R}(F_{\bullet}(-),f_{1})\cong\Phi(T_{1})\ar[dd]^{\Phi(\iota)}\ar[rr]^{\varphi^{T_{1}}}&&\Theta(T_{1})\cong
{}_{\CR}({}_{R}(*,F_{\bullet}(-)),T_{1})\ar[dd]^{\Theta(\iota)}\\\\\Coker_{R}(F_{\bullet}(-),f_{2})\cong
\Phi(T_{2})\ar[rr]^{\varphi^{T_{2}}}&&\Theta(T_{2})\cong{}_{\CR}({}_{R}(*,F_{\bullet}(-)),T_{2}),}$$ that is, $\varphi_{X}^{T_{2}}\circ\Phi(\iota)_{X}=\Theta(\iota)_{X}\circ\varphi_{X}^{T_{1}}$, for any $X\in\mod(A)$. Assume that $\iota:T_{1}\ra T_{2}$ is defined by $g:M_{1}\ra M_{2}$ and $h:N_{1}\ra N_{2}$. In particular, we have $\iota\pi_{1}=\pi_{2}\circ{}_{R}(*,h)$. The $K$-linear homomorphism $$\Phi(\iota)_{X}:\Coker_{R}(F_{\bullet}(X),f_{1})\ra\Coker_{R}(F_{\bullet}(X),f_{2})$$ is given by the formula $$\Phi(\iota)_{X}(\alpha+\Im_{R}(F_{\bullet}(X),f_{1}))=h\alpha+\Im_{R}(F_{\bullet}(X),f_{2}),$$ for any $R$-homomorphism $\alpha:F_{\bullet}(X)\ra N_{1}$. The $K$-linear homomorphism $\Theta(\iota)$ is induced by $\iota$, that is, $\Theta(\iota)=\iota\circ(-)$. Assume now that $\alpha:F_{\bullet}(X)\ra N_{1}$ is an $R$-homomorphism. For simplicity, set $\ov{\alpha}=\alpha+\Im_{R}(F_{\bullet}(X),f_{1})$ and $\ov{h\alpha}=h\alpha+\Im_{R}(F_{\bullet}(X),f_{2})$. Then the following equalities hold: $$(\varphi_{X}^{T_{2}}\circ\Phi(\iota)_{X})(\ov{\alpha})=\varphi_{X}^{T_{2}}(\ov{h\alpha})=
\pi_{2}\circ{}_{R}(*,h\alpha)=$$$$=\pi_{2}\circ({}_{R}(*,h)\circ{}_{R}(*,\alpha))=(\pi_{2}\circ{}_{R}(*,h))\circ{}_{R}(*,\alpha)=
(\iota\pi_{1})\circ{}_{R}(*,\alpha)=\iota(\pi_{1}\circ{}_{R}(*,\alpha))=$$$$=\Theta_{X}(\pi_{1}\circ{}_{R}(*,\alpha))=
\Theta_{X}(\varphi_{X}^{T_{1}}(\ov{\alpha}))=(\Theta_{X}\circ\varphi_{X}^{T_{1}})(\ov{\alpha})$$ and hence the above diagram commutes. This implies that $\varphi=(\varphi_{T})_{T\in\CF(R)}$ is a natural transformation of functors and thus an equivalence, because all the $\varphi_{T}$ are equivalences. Since $\Phi$ is right exact whereas $\Theta$ is left exact (see Corollary \ref{c1}) we conclude that $\Phi\cong\Theta$ is exact on $\CF(R)$.

$(2)$ Assume that $T_{1},T_{2}\in\CF(R)$ and $\iota:T_{1}\ra T_{2}$ is a morphism of functors such that $\Phi(\iota)=0$. We show that this yields $\iota=0$. Indeed, assume that $T_{i}=(T_{i},f_{i},\pi_{i})$, $f_{i}:M_{i}\ra N_{i}$, for $i=1,2$. Moreover, let $\iota$ be defined by $g:M_{1}\ra M_{2}$ and $h:N_{1}\ra N_{2}$. This means that we have the following commutative diagram with exact rows: $$\xymatrix{{}_{R}(*,M_{1})\ar[rr]^{{}_{R}(*,f_{1})}\ar[dd]^{{}_{R}(*,g)}&&{}_{R}(*,N_{1})\ar[dd]^{{}_{R}(*,h)}\ar[rr]^{\pi_{1}}&&T_{1}\ar[dd]^{\iota}\ar[rr]&&0\\\\
{}_{R}(*,M_{2})\ar[rr]^{{}_{R}(*,f_{2})}&&{}_{R}(*,N_{2})\ar[rr]^{\pi_{2}}&&T_{2}\ar[rr]&&0.}$$ If $\Phi(\iota)=0$, then $$0=\Phi(\iota)\Phi(\pi_{1})=\Phi(\iota\pi_{1})=\Phi(\pi_{2}\circ{}_{R}(*,h))=\Phi(\pi_{2})\circ\Phi({}_{R}(*,h))$$ which yields $$\Im(\Phi({}_{R}(*,h)))\subseteq\Ker\Phi(\pi_{2})=\Im(\Phi({}_{R}(*,f_{2})))$$ and thus $\Im({}_{A}(-,F_{\lambda}(h)))\subseteq\Im({}_{A}(-,F_{\lambda}(f_{2})))$. Since the functor ${}_{A}(-,F_{\lambda}(N_{1}))$ is projective, we conclude there is an $A$-homomorphism $\alpha:F_{\lambda}(N_{1})\ra F_{\lambda}(M_{2})$ such that the following diagram is commutative with exact rows:$$\xymatrix{{}_{A}(-,F_{\lambda}(M_{1}))\ar[rr]^{{}_{A}(-,F(f_{1}))}\ar[dd]^{{}_{A}(-,F_{\lambda}(g))}&&{}_{A}(-,F_{\lambda}(N_{1}))\ar[dd]^{{}_{A}(-,F_{\lambda}(h))}\ar[rr]^{\Phi(\pi_{1})}\ar[ddll]^{{}_{A}(-,\alpha)}&&\Phi(T_{1})\ar[dd]^{\Phi(\iota)}\ar[rr]&&0\\\\
{}_{A}(-,F_{\lambda}(M_{2}))\ar[rr]^{{}_{A}(-,F_{\lambda}(f_{2}))}&&{}_{A}(-,F(N_{2}))\ar[rr]^{\Phi(\pi_{2})}&&\Phi(T_{2})\ar[rr]&&0.}$$ Consequently $F_{\lambda}(f_{2})\alpha=F_{\lambda}(h)$, so $F_{\bullet}(F_{\lambda}(f_{2})\alpha)=F_{\bullet}(F_{\lambda}(h))$ and thus we get $$(\bigoplus_{g\in G}{}^{g}f_{2})F_{\bullet}(\alpha)=\bigoplus_{g\in G}{}^{g}h.$$ Assuming $$F_{\bullet}(\alpha)=[\alpha_{h,g}]:\bigoplus_{g\in G}{}^{g}N_{1}\ra\bigoplus_{h\in G}{}^{h}M_{2}$$ we obtain $f_{2}\alpha_{e,e}=h$ which yields ${}_{R}(*,f_{2})\circ{}_{R}(*,\alpha_{e,e})={}_{R}(*,h)$. This in turn implies that $\iota=0$. Hence $\Phi:\CF(R)\ra\CF(A)$ is a faithful functor and the same holds for the functor $\Theta:\CF(R)\ra\CF(A)$, by $(1)$. \end{proof}

The following theorem is one of the main results of the paper. In a slogan form, the theorem states that \emph{Galois $G$-coverings do not increase the Krull-Gabriel dimension}.

\begin{thm}\label{t3} Assume that $F:R\ra A$ is a Galois $G$-covering. Then $\KG(R)\leq\KG(A)$. In particular:
\begin{enumerate}[\rm(1)]
  \item If $\KG(R)$ is undefined, then $\KG(A)$ is undefined.
  \item If $\KG(A)$ is finite, then $\KG(R)$ is finite.
\end{enumerate}
\end{thm}

\begin{proof} It follows from Theorem \ref{t2} that $\Phi:\CF(R)\ra\CF(A)$ is exact and faithful, so $\KG(R)\leq\KG(A)$ follows from \ref{0l1}.
\end{proof}

\section{Covering theory for finitely presented functors. The dense case}

Assume that $F:R\ra A$ is a Galois $G$-covering. This section deals with further properties of the induced functor $\Phi:\CF(R)\ra\CF(A)$. The first main result is Theorem \ref{t4} which states that if the group $G$ is torsion-free, then the functor $\Phi:\CF(R)\ra\CF(A)$ is a Galois $G$-precovering of functor categories. This is a vast generalization of \cite[Theorem 5.5]{P4} where it is shown that the functor $\Phi:\CF(R)\ra\CF(A)$ is a Galois $G$-precovering, under the assumption that $R$ is locally support-finite. Then we study the special case when the push-down functor $F_{\lambda}:\mod(R)\ra\mod(A)$ is dense. Theorem \ref{t6} is the final result of this part of the section, see also Theorem \ref{t5} and Corollary \ref{c2}. We show in \ref{t6} that in this case the functors $\Phi,\Theta:\MOD(\CR)\ra\MOD(\CA)$ have descriptions analogous to the classical covering functors, see Section 2.2. In particular, $\Phi$ is a subfunctor of $\Theta$. In Remark \ref{r3} we discuss equivalent definitions of $\Phi:\CF(R)\ra\CF(A)$ under the assumption that $F_{\lambda}$ is dense. The section ends with a reminder on the important special case when $R$ is locally support-finite and $G$ is torsion-free, hence $F_\lambda$ is dense. Indeed, we prove in \cite{P4,P6} that then the Krull-Gabriel dimension is preserved by Galois $G$-coverings. We believe that recalling these results, together with sketches of the proofs and some additional comments, is of benefit to the reader. It is naturally also within the scope of this section.

\subsection{The functor $\Phi:\CF(R)\ra\CF(A)$ as a precovering}

The action of $G$ on $\mod(R)$ induces an action of $G$ on $\MOD(\CR)$. Indeed, given a functor $T\in\MOD(\CR)$ and $g\in G$ we define $gT\in\MOD(\CR)$ to be as follows: $(gT)(X)=T({}^{g^{-1}}X)$ and $(gT)(f)=T({}^{g^{-1}}f)$, for any module $X\in\mod(R)$ and homomorphism $f\in\mod(R)$. Furthermore, given a morphism of functors $\iota:T_{1}\ra T_{2}$ and $g\in G$ we define $g\iota:gT_{1}\ra gT_{2}$ as $(g\iota)_{X}=\iota_{X^{-g}}$, for any $X\in\mod(R)$.

It is easy to see that the above is indeed an action of $G$ on $\MOD(\CR)$. We recall in the following proposition that this action restricts to $\CF(R)$ and that it is free, provided $G$ is torsion-free. The proposition summarizes some facts proved already in \cite{P4}. Namely, the assertions $(2)$ and $(3)$ correspond to \cite[Lemma 5.1]{P4} and \cite[Lemma 5.4]{P4}, respectively, whereas $(1)$ is implicitly contained in the proof of \cite[Theorem 5.5 (1)]{P4}.   

\begin{prop}\label{p2} Assume that $T=\Coker{}_{R}(*,f)\in\CF(R)$, for some $R$-homomorphism $f:M\ra N$ and $T\neq 0$. The following assertions hold.
\begin{enumerate}[\rm(1)]
  \item We have $gT\cong\Coker{}_{R}(*,{}^{g}f)\in\CF(R)$, for any $g\in G$ and hence the action of $G$ on $\MOD(\CR)$ restricts to $\CF(R)$.
  \item If $X\in\mod(R)$, then $T({}^{g}X)\neq 0$ only for finite number of $g\in G$.
  \item If the group $G$ is torsion-free, then $G$ acts freely on $\CF(R)$.
\end{enumerate}
\end{prop}

\begin{proof} $(1)$ Assume that $g\in G$ and $X,Y\in\mod(R)$ are arbitrary modules. Since we have $$(g{}_{R}(*,X))(Y)={}_{R}({}^{g^{-1}}Y,X)\cong{}_{R}(Y,{}^{g}X)$$ and the later isomorphism is natural, we conclude that there is a natural isomorphism $g({}_{R}(*,X))\cong{}_{R}(*,{}^{g}X)$. This yields $g({}_{R}(*,f))\cong{}_{R}(*,{}^{g}f)$ and so we conclude $$gT=g(\Coker{}_{R}(*,f))\cong\Coker{}_{R}(*,{}^{g}f)\in\CF(R).$$

$(2)$ Note that the vector space $T({}^{g}X)$ is a quotient of ${}_{R}({}^{g}X,N)$, for any $g\in G$, and there is only a finite number of $g\in G$ such that ${}_{R}({}^{g}X,N)\neq 0$. Indeed, since $G$ acts freely on the objects of $R$, the function $g\mapsto gx$ is injective, for a fixed $x\in\ob(R)$. This implies that $\supp({}^{g}X)\cap\supp(N)\neq\emptyset$ only for finite number of $g\in G$ and the claim follows, because the supports of $X$ and $N$ are finite.

$(3)$ We show that $gT\cong T$ implies $g=1$, for any $g\in G$. Indeed, let $X$ be an $R$-module such that $T(X)\neq 0$. Since $gT\cong T$, we get $g^{n}T\cong T$ and thus $T(X)\cong T({}^{g^{-n}}X)\neq 0$, for any $n\in\NN$. If $g\neq 1$, the elements $g^{-n}$ are pairwise different, because $G$ is torsion-free. Hence by $(2)$ we conclude that $g=1$ and consequently $G$ acts freely on $\CF(R)$.
\end{proof}

The following straightforward fact shows that $\CF(R)$ is not only an abelian $K$-category, but also Krull-Schmidt. In fact, all the assertions are folklore, but we present their short proofs for convenience.

\begin{prop}\label{p3} The following assertions hold for any locally bounded $K$-category $R$.

\begin{enumerate}[\rm(1)]
  \item The category $\CF(R)$ is hom-finite.
  \item A functor $T\in\CF(R)$ has local endomorphism algebra if and only if $T$ is indecomposable.
  \item The category $\CF(R)$ is an abelian Krull-Schmidt $K$-category.
\end{enumerate}
\end{prop}

\begin{proof} $(1)$ Assume that $T\in\CF(R)$ and $X\in\mod(R)$. Observe that the Yoneda lemma yields ${}_{\CR}({}_{R}(*,X),T)\cong T(X)$ and hence ${}_{\CR}({}_{R}(*,X),T)$ is of finite dimension. Furthermore, assume that $T_{1},T_{2}\in\CF(R)$. Since $T_{1}$ is finitely presented, there exists an epimorphism $\pi:{}_{R}(*,Y)\ra T_{1}$, for some $Y\in\mod(R)$. Then the map ${}_{\CR}(T_{1},T_{2})\ra{}_{\CR}({}_{R}(*,Y),T_{2})$ defined as $(-)\circ\pi$ is a monomorphism of vector spaces, because $\iota\circ\pi=0$ yields $\iota=0$, for any $\iota\in{}_{\CR}(T_{1},T_{2})$ (since $\pi$ is an epimorphism). This implies that the hom-space ${}_{\CR}(T_{1},T_{2})$ is of finite dimension.

$(2)$ It is enough to show that any indecomposable functor $T\in\CF(R)$ has local endomorphism algebra, because the converse is true in general. Hence assume that $T\in\CF(R)$ is indecomposable. Then $(1)$ yields $\End_{\CR}(T)$ is a finite dimensional $K$-algebra. Therefore if $\End_{\CR}(T)$ is not local, there exist two non-zero idempotents $e_{1},e_{2}=1_{T}-e_{1}\in\End_{\CR}(T)$, see for example \cite[I 4.6]{AsSiSk}. It follows that $T\cong\Im(e_{1})\oplus\Im(e_{2})$ and $\Im(e_{1}),\Im(e_{2})$ are finitely presented, because $\CF(R)$ is abelian. This contradicts the assumption about $T$.

$(3)$ Assume that $T\in\CF(R)$. An easy inductive argument based on $(1)$ yields $T$ decomposes into a finite direct sum of indecomposable finitely presented functors. Since these functors have also local endomorphism algebras by $(2)$, we conclude from Krull-Schmidt theorem (see for example Appendix E of \cite{Pr2}) that this decomposition is unique, up to isomorphism of the direct summands and their order.
\end{proof}

The following theorem shows fundamental properties of the functor $\Phi:\CF(R)\ra\CF(A)$ induced by the given Galois $G$-covering $F:R\ra A$. Namely, it states that $\Phi$ is a Galois $G$-precovering of functor categories provided $G$ is torsion-free. The torsion-freeness of $G$ is necessary to show that $\Phi$ preserves indecomposable functors, but the remaining properties are independent of this assumption. 

The theorem is a generalization of \cite[Theorem 5.5]{P4} where it is shown that the functor $\Phi:\CF(R)\ra\CF(A)$ is a Galois $G$-precovering under the assumption that $R$ is locally support-finite. Recall that in this case the push-down functor $F_{\lambda}:\mod(R)\ra\mod(A)$ is dense, and this fact was heavily used in the proof of \cite[Theorem 5.5]{P4}.

\begin{thm}\label{t4} Assume that $R$ is a locally bounded $K$-category, $G$ an admissible group of $K$-linear automorphisms of $R$ and $F:R\ra A\cong R\slash G$ the Galois $G$-covering. The functor $\Phi:\CF(R)\ra\CF(A)$ has the following properties.

\begin{enumerate}[\rm(1)]
	\item There exists an isomorphism $\Phi(T)\cong\Phi(gT)$, for any $T\in\CF(R)$ and $g\in G$.
	\item There exists an isomorphism $\Psi(\Phi(T))\cong\bigoplus_{g\in G}gT$, for any $T\in\CF(R)$. Moreover, if functors $T_{1},T_{2}\in\CF(R)$ are indecomposable, then $\Phi(T_{1})\cong\Phi(T_{2})$ implies $T_{1}\cong gT_{2}$, for some $g\in G$.
\item If $G$ is torsion-free, then $\Phi$ preserves indecomposability.
	\item Assume $T,T'\in\CF(R)$ and $U\in\CF(A)$. There are natural isomorphisms of vector spaces $${}_{\CA}(\Phi(T),U)\cong {}_{\CR}(T,\Psi(U))\textnormal{ and }{}_{\CR}(\Psi(U),T)\cong {}_{\CA}(U,\Phi(T))$$ which induce natural isomorphisms $$\bigoplus_{g\in G}{}_{\CR}(gT,T')\cong{}_{\CA}(\Phi(T),\Phi(T'))\cong\bigoplus_{g\in G}{}_{\CR}(T,gT').$$ The latter isomorphisms are given by $$(\iota_{g})_{g\in G}\mapsto\sum_{g\in G}\Phi(\iota_{g})\mapsfrom(g^{-1}\iota_{g})_{g\in G}$$ where $\iota_{g}:gT_{1}\ra T_{2}$ is a homomorphism of functors, for any $g\in G$.
\item Consequently, if the group $G$ is torsion-free, then the functor $\Phi:\CF(R)\ra\CF(A)$ is a Galois $G$-precovering of functor categories.
\end{enumerate}
\end{thm}

\begin{proof}
(1) Assume $g\in G$ and $T=\Coker{}_{R}(*,f)$, for some $R$-homomorphism $f:M\ra N$. Then $gT\cong\Coker{}_{R}(*,{}^{g}f)$, by Proposition \ref{p2} (1) and since $F_{\lambda}(f)\cong F_{\lambda}({}^{g}f)$, we obtain $$\Phi(gT)\cong\Coker{}_{A}(-,F_{\lambda}({}^{g}f))\cong\Coker{}_{A}(-,F_{\lambda}(f))\cong\Phi(T).$$

(2) Assume that $T=\Coker{}_{R}(*,f)$, for some $R$-homomorphism $f:M\ra N$. It is easy to see that ${}_{R}(*,\bigoplus_{g\in G}{}^{g}L)\cong\bigoplus_{g\in G}{}_{R}(*,L)$, for any $L\in\mod(R)$. Moreover, recall that $F_{\bullet}(F_{\lambda}(f))\cong\bigoplus_{g\in G}{}^{g}f$. Therefore we obtain the following natural isomorphisms: $$\Psi(\Phi(T))=\Phi(T)\circ F_{\lambda}\cong\Coker_{R}(F_{\lambda}(*),F_{\lambda}(f))\cong\Coker_{R}(*,F_{\bullet}(F_{\lambda}(f)))\cong$$$$\cong
\Coker_{R}(*,\bigoplus_{g\in G}{}^{g}f)\cong\Coker_{R}(\bigoplus_{g\in G}(*,{}^{g}f))\cong\bigoplus_{g\in G}\Coker_{R}(*,{}^{g}f)\cong\bigoplus_{g\in G}gT.$$ This shows the first part of the assertion. Now assume that $T_{1},T_{2}\in\CF(R)$ are indecomposable and $\Phi(T_{1})\cong\Phi(T_{2})$. Then $\Psi(\Phi(T_{1}))\cong\Psi(\Phi(T_{2}))$ and hence there is an isomorphism $\bigoplus_{g\in G}gT_{1}\cong\bigoplus_{g\in G}gT_{2}$. Proposition \ref{p3} (2) yields $T_{1}$ and $T_{2}$ have local endomorphism algebras. Note that, for any $g\in G$, $gT_{1}$ and $gT_{2}$ have local endomorphism algebras as well, because $\End_{\CR}(T)\cong\End_{\CR}(gT)$, for any $T\in\CF(R)$. Hence the Krull-Schmidt theorem yields $T_{1}\cong gT_{2}$, for some $g\in G$.

(3) Assume that $G$ is torsion-free and $T\in\CF(R)$ is indecomposable. We show that $\Phi(T)$ is indecomposable. Indeed, assume to the contrary that $\Phi(T)\cong U\oplus V$, for some functors $U,V\in\CF(A)$ such that $U$ is indecomposable. Then it follows by $(2)$ that $\bigoplus_{g\in G}gT\cong\Psi(U)\oplus\Psi(V)$ and so the Krull-Schmidt theorem yields $\Psi(U)\cong\bigoplus_{h\in H}hT$, for some nonempty set $H\subseteq G$. Observe that $g\Psi(U)=g(U\circ F_{\lambda})\cong U\circ F_{\lambda}$, because $F_{\lambda}$ is $G$-invariant, and thus $$g\Psi(U)\cong\bigoplus_{h\in H}(gh)T\cong\bigoplus_{h\in H}hT,$$ for any $g\in G$. Since $G$ acts freely on $\CF(R)$ by Proposition \ref{p2} (3), we conclude that $gH\subseteq H$, for any $g\in G$. This easily implies that $H=G$, thus $V=0$ and so $\Phi(T)\cong U$ is indecomposable.

(4) The first two natural isomorphisms follow from the fact that $(\Phi,\Psi)$ and $(\Psi,\Theta)$ are adjoint pairs and $\Phi|_{\CF(R)}\cong\Theta|_{\CF(R)}$, see Corollary \ref{c1} and Theorem \ref{t2} (1), respectively. Now we show the latter isomorphisms. Set $U=\Phi(T')$. Then by $(2)$ we obtain $${}_{\CA}(\Phi(T),\Phi(T'))\cong {}_{\CR}(T,\Psi(\Phi(T')))\cong{}_{\CR}(T,\bigoplus_{g\in G}{}gT').$$ Since the vector space ${}_{\CA}(\Phi(T),\Phi(T'))$ is of finite dimension by Proposition \ref{p3} (1), we conclude that there are only finitely many $g\in G$ such that ${}_{\CR}(T,gT')\neq 0.$\footnote{A direct argument for this to hold is the following. Since $T,T'\in\CF(R)$, there are $N,N'\in\mod(R)$ and epimorphisms ${}_{R}(*,N)\ra T$ and ${}_{R}(*,N)\ra T'$. Hence if ${}_{\CR}(T,gT')\neq 0$, then ${}_{\CR}({}_{R}(*,N),{}_{R}(*,{}^{g}N'))\cong{}_{R}(N,{}^{g}N')\neq 0$ and ${}_{R}(N,{}^{g}N')\neq 0$ holds only for finite number of $g\in G$.} Moreover, it is easy to see that ${}_{\CR}(T,gT')\cong{}_{\CR}(g^{-1}T,T')$, for any $g\in G$ and thus we obtain $${}_{\CR}(T,\bigoplus_{g\in G}{}gT')\cong\bigoplus_{g\in G}{}_{\CR}(T,gT')\cong\bigoplus_{g\in G}{}_{\CR}(gT,T').$$ This shows the first part of the assertion. In the sequel we prove that the isomorphism $\bigoplus_{g\in G}{}_{\CR}(gT,T')\cong{}_{\CA}(\Phi(T),\Phi(T'))$ is given by $$(\iota_{g})_{g\in G}\stackrel{\mu}{\mapsto}\sum_{g\in G}\Phi(\iota_{g})$$ where $\iota_{g}:gT_{1}\ra T_{2}$ are homomorphisms of functors. Indeed, it is easy to see that $\mu$ is a $K$-linear map. Since we know that the appropriate vector spaces have the same finite dimension, it suffices to prove that the map is an epimorphism. Hence assume that $T=\Coker_{R}(*,f)$, $T'=\Coker_{R}(*,f')$ where $f:M\ra N$ and $f':M'\ra N'$ and let $\theta:\Phi(T)\ra\Phi(T')$ be a morphism of functors defined by $\alpha:F_{\lambda}(M)\ra F_{\lambda}(M')$ and $\beta:F_{\lambda}(N)\ra F_{\lambda}(N')$. Recall that there are $s_{g}:{}^{g}M\ra M'$ and $t_{g}:{}^{g}N\ra N'$, for $g\in G$, such that $\alpha=\sum_{g\in G}F_{\lambda}(s_{g})$ and $\beta=\sum_{g\in G}F_{\lambda}(t_{g})$. Let $\iota_{g}:gT\ra T'$ be defined by $(s_{g},t_{g})$, for $g\in G$, which means there is the following commutative diagram with exact rows: $$\xymatrix{{}_{R}(*,{}^{g}M)\ar[rr]^{{}_{R}(*,{}^{g}f)}\ar[dd]^{{}_{R}(*,s_g)}&&{}_{R}(*,{}^{g}N)\ar[dd]^{{}_{R}(*,t_g)}\ar[rr]&&gT\ar[dd]^{\iota_{g}}\ar[rr]&&0\\\\
{}_{R}(*,M')\ar[rr]^{{}_{R}(*,f')}&&{}_{R}(*,N')\ar[rr]&&T'\ar[rr]&&0.}$$ Then $\Phi(\iota_{g})$ is defined by $(F_{\lambda}(s_{g}),F_{\lambda}(t_{g}))$ and since $F_{\lambda}$ is $G$-invariant, we obtain the following commutative diagram with exact rows: $$\xymatrix{{}_{A}(-,F_{\lambda}(M))\ar[rr]^{{}_{A}(-,F_{\lambda}(f))}\ar[dd]^{{}_{A}(-,F_{\lambda}(s_g))}&&{}_{A}(-,F_{\lambda}(N))\ar[dd]^{{}_{A}(-,F_{\lambda}(t_g))}\ar[rr]&&\Phi(gT)\cong\Phi(T)\ar[dd]^{\Phi(\iota_{g})}\ar[rr]&&0\\\\
{}_{A}(-,F_{\lambda}(M'))\ar[rr]^{{}_{A}(-,F_{\lambda}(f'))}&&{}_{A}(-,F_{\lambda}(N'))\ar[rr]&&\Phi(T')\ar[rr]&&0.}$$ Observe that for any $A$-modules $X,Y$ and $A$-homomorphisms $h_{1},h_{2}:X\ra Y$ we have $${}_{A}(-,h_{1}+h_{2})={}_{A}(-,h_{1})+{}_{A}(-,h_{2}).$$ This fact easily implies that the morphism of functors $\sum_{g\in G}\Phi(\iota_{g}):\Phi(T)\ra\Phi(T')$ is defined by $$(\sum_{g\in G}F_{\lambda}(s_{g}),\sum_{g\in G}F_{\lambda}(t_{g}))=(\alpha,\beta)$$ and hence it equals the initial morphism $\theta:\Phi(T)\ra\Phi(T')$. This yields $$\mu((\iota_{g})_{g\in G})=\sum_{g\in G}\Phi(\iota_{g})=\theta$$ and the claim follows. Finally, observe that the function $(\iota_{g})_{g\in G}\mapsto(g^{-1}\iota_{g})_{g\in G}$ induces an isomorphism $\bigoplus_{g\in G}{}_{\CR}(gT,T')\cong\bigoplus_{g\in G}{}_{\CR}(T,gT')$. This finishes the proof of $(4)$.

(5) This follows directly from all the previous facts. \end{proof}

\begin{rem}\label{r2} We show in Theorem \ref{t2} (2) that the functor $\Phi:\CF(R)\ra\CF(A)$ is faithful. Recall that this property, together with the exactness of $\Phi$, is crucial for the inequality $\KG(R)\leq\KG(A)$ to hold (see Theorem \ref{t3}). It it worth to notice that the faithfulness of $\Phi:\CF(R)\ra\CF(A)$ follows also from the fact that the isomorphism $$\bigoplus_{g\in G}{}_{\CR}(gT,T')\cong{}_{\CA}(\Phi(T),\Phi(T')),$$ proved above, is given by $(\iota_{g})_{g\in G}\mapsto\sum_{g\in G}\Phi(\iota_{g})$. Indeed, by composing this isomorphism with the inclusion monomorphism ${}_{\CR}(T,T')\hookrightarrow\bigoplus_{g\in G}{}_{\CR}(gT,T')$, we obtain a monomorphism which gets the form $\iota\mapsto\Phi(\iota)$. This shows the faithfulness of the functor $\Phi$.
\end{rem}

\subsection{The case of dense push-down functor}

Throughout the section we assume that $R$ is a locally bounded $K$-category, $G$ an admissible group of $K$-linear automorphisms of $R$ and $F:R\ra A\cong R\slash G$ the Galois covering. 

The following Theorem \ref{t5} widely generalizes a part of the assertion $(2)$ of Theorem \ref{t4}, stating only the existence of an isomorphism $\Psi(\Phi(T))\cong\bigoplus_{g\in G}gT$, for any finitely presented $T\in\CF(R)$. This result is further applied to show that if the push-down functor $F_{\lambda}:\mod(R)\ra\mod(A)$ is dense, then the functors $\Phi,\Theta:\MOD(\CR)\ra\MOD(\CA)$ have descriptions analogous to the classical covering functors, see Corollary \ref{c2}. In the dense case, this also yields a simpler proof of the fact that $\Phi$ and $\Theta$ are equivalent on the category $\CF(R)$ of finitely presented functors, proved in general in Theorem \ref{t2} (1), and hence that they are exact functors, see Theorem \ref{t6}.

\begin{thm}\label{t5} For any functor $T\in\MOD(\CR)$ we have natural isomorphisms of functors $$\Psi(\Phi(T))=\Phi(T)\circ F_{\lambda}\cong\bigoplus_{g\in G}gT\textnormal{ and }\Psi(\Theta(T))=\Theta(T)\circ F_{\lambda}\cong\prod_{g\in G}gT$$ which extend to natural equivalences $$\Psi(\Phi(\cdot))\cong\bigoplus_{g\in G}g(\cdot)\textnormal{ and }\Psi(\Theta(\cdot))\cong\prod_{g\in G}g(\cdot).$$

\end{thm}

\begin{proof} Let $T\in\MOD(\CR)$ and $?\in\mod(R)$. We show that $\Psi(\Phi(T))\cong\bigoplus_{g\in G}gT$. Indeed, first observe we have the following natural isomorphisms: $$\Psi(\Phi(T))=\Phi(T)\circ F_{\lambda}=T\otimes_{\CR}[{}_{A}(F_{\lambda}(?),F_{\lambda}(*))]\cong T\otimes_{\CR}[{}_{R}(F_{\bullet}(F_{\lambda}(?)),*)]\cong$$ $$\cong T\otimes_{\CR}[{}_{R}(\bigoplus_{g\in G}{}^{g}?,*)]\cong T\otimes_{\CR}[\prod_{g\in G}{}_{R}({}^{g}?,*)].$$ Proposition \ref{p2} (2) yields $\prod_{g\in G}{}_{R}({}^{g}?,*)\cong\bigoplus_{g\in G}{}_{R}({}^{g}?,*)$, because for any $M,X\in\mod(R)$ there is only a finite number of $g\in G$ such that ${}_{R}({}^{g}X,M)\neq 0$. Therefore we obtain $$T\otimes_{\CR}[\prod_{g\in G}{}_{R}({}^{g}?,*)]\cong T\otimes_{\CR}[\bigoplus_{g\in G}{}_{R}({}^{g}?,*)]\cong\bigoplus_{g\in G}T\otimes_{\CR}[{}_{R}({}^{g}?,*)]\cong\bigoplus_{g\in G}T(^{g}?)=\bigoplus_{g\in G}gT,$$ which follows from the fact that the tensor product commutes with colimits and also from the co-Yoneda lemma, respectively. 

Now we show that $\Psi(\Theta(T))\cong\prod_{g\in G}gT$. Indeed, some standard arguments and the Yoneda lemma yield the natural isomorphisms: $$\Psi(\Theta(T))=\Theta(T)\circ F_{\lambda}={}_{\CR}({}_{A}(F_{\lambda}(*),F_{\lambda}(?)),T)\cong{}_{\CR}({}_{R}(*,F_{\bullet}(F_{\lambda}(?))),T)\cong$$ $$
\cong{}_{\CR}({}_{R}(*,\bigoplus_{g\in G}{}^{g}{?}),T)\cong{}_{\CR}(\bigoplus_{g\in G}{}_{R}(*,{}^{g}{?}),T)\cong\prod_{g\in G}{}_{\CR}({}_{R}(*,{}^{g}{?}),T)\cong\prod_{g\in G}T(^{g}?)=\prod_{g\in G}gT.$$ The second part of the thesis follows from similar arguments, because if $T_{1},T_{2}\in\MOD(R)$ and $\iota:T_{1}\ra T_{2}$ is a morphism of functors, then for any $M\in\mod(R)$ we have $$\iota\otimes_{\CR}[{}_{R}(M,*)]\cong\iota_{M}\textnormal{ and }{}_{\CR}({}_{R}(*,M),\iota)\cong\iota_{M},$$ by the co-Yoneda and Yoneda lemmas, respectively. In consequence, we obtain the isomorphisms $\Psi(\Phi(\iota))\cong\bigoplus_{g\in G}g\iota$ and $\Psi(\Theta(\iota))\cong\prod_{g\in G}g\iota$ which shows the claim.
\end{proof}


It is convenient to formulate the following straightforward corollary of Theorem \ref{t5}, giving an exact description of the functors $\Phi,\Theta:\MOD(\CR)\ra\MOD(\CA)$ on the category of $A$-modules of the first kind. This description is purely analogous to the case of classical covering functors. In some sense, it also generalizes and retrieves the construction of the functor $\Phi:\CF(R)\ra\CF(A)$, as given in Section 5 of \cite{P4}. 

\begin{cor}\label{c2} The functors $\Phi,\Theta:\MOD(\CR)\ra\MOD(\CA)$ have the following properties. 
\begin{enumerate}[\rm(1)]
  \item Assume that $T\in\MOD(\CR)$ and $M\in\mod(R)$. Then we have: $$\Phi(T)(F_{\lambda}(M))\cong\bigoplus_{g\in G}T({}^{g}M)\textnormal{ and }\Theta(T)(F_{\lambda}(M))\cong\prod_{g\in G}T({}^{g}M).$$
  \item Assume that $\alpha:F_{\lambda}(M)\ra F_{\lambda}(N)$ is an $A$-homomorphism such that $\alpha=\sum_{g\in G}F_{\lambda}(f_{g})$ where $f_{g}:{}^{g}M\ra N$, for any $g\in G$. Then we have: $$\Phi(T)(\alpha)\cong\sum_{g\in G}\bigoplus_{h\in G}T({}^{h}f_{g})\textnormal{ and }\Theta(T)(\alpha)\cong\sum_{g\in G}\prod_{h\in G}T({}^{h}f_{g}).$$
  \item Assume that $T_{1},T_{2}\in\MOD(\CR)$, $\iota:T_{1}\ra T_{2}$ is a morphism of functors and let $M\in\mod(R)$. Then we have: $$\Phi(\iota)_{F_{\lambda}(M)}\cong\bigoplus_{g\in G}\iota_{{}^{g}M}\textnormal{ and }\Theta(\iota)_{F_{\lambda}(M)}\cong\prod_{g\in G}\iota_{{}^{g}M}.$$
\end{enumerate}
\end{cor}

\begin{proof} We show that the assertions $(1)$, $(2)$ and $(3)$ follow directly from Theorem \ref{t5}. Indeed, observe that $$\Phi(T)(F_{\lambda}(M))=\Psi(\Phi(T))(M)\cong\bigoplus_{g\in G}gT(M)=\bigoplus_{g\in G}T({}^{g}M)\textnormal{, and moreover}$$ $$\Phi(T)(\alpha)=\Phi(T)(\sum_{g\in G}F_{\lambda}(f_{g}))=\sum_{g\in G}\Phi(T)(F_{\lambda}(f_{g}))\cong$$$$\cong\sum_{g\in G}\Psi(\Phi(T))(f_{g})\cong\sum_{g\in G}\bigoplus_{h\in G}hT(f_{g})=\sum_{g\in G}\bigoplus_{h\in G}T({}^{h}f_{g}).$$ This shows that the homomorphism $$\Phi(T)(\alpha):\bigoplus_{g\in G}T({}^{g}Y)\ra\bigoplus_{g\in G}T({}^{g}X)$$ is defined by $T({}^{h}f_{h^{-1}g}):T({}^{h}Y)\ra T({}^{g}X)$, for any $g,h\in G$. Similarly, we have $$\Phi(\iota)_{F_{\lambda}(M)}\cong\Psi(\Phi(\iota))_{M}\cong\bigoplus_{g\in G}(g\iota)_{M}=\bigoplus_{g\in G}\iota_{{}^{g}M}$$ which implies that the homomorphism $$\Phi(\iota)_{F_{\lambda}(M)}:\bigoplus_{g\in G}T_{1}({}^{g}M)\ra\bigoplus_{g\in G}T_{2}({}^{g}M)$$ is defined by $\iota_{{}^{g}M}:T_{1}({}^{g}M)\ra T_{2}({}^{g}M)$, for any $g\in G$. It is easy to see that all the isomorphisms for the functor $\Theta:\MOD(\CR)\ra\MOD(\CA)$ have analogous proofs. 
\end{proof}

The above facts show in particular that for any functor $T\in\MOD(\CR)$ the functor $\Phi(T)\in\MOD(\CA)$ is a subfunctor of the functor $\Theta(T)\in\MOD(\CA)$, if these functors are restricted to the category of $A$-modules of the first kind. We thus conclude the following result which also gives another proof of Theorem \ref{t2} (1) in the case the push-down $F_{\lambda}:\mod(R)\ra\mod(A)$ is dense.

\begin{thm}\label{t6} Assume that the push-down functor $F_{\lambda}:\mod(R)\ra\mod(A)$ is dense. Then $\Phi:\MOD(\CR)\ra\MOD(\CA)$ is a subfunctor of $\Theta:\MOD(\CR)\ra\MOD(\CA)$. Moreover, these functors restrict to categories of finitely presented functors and they coincide on $\CF(R)$. Consequently, both functors $\Phi,\Theta:\CF(R)\ra\CF(A)$ are exact.
\end{thm}

\begin{proof} It follows from Theorem \ref{t5} or Corollary \ref{c2} that $\Phi:\MOD(\CR)\ra\MOD(\CA)$ is a subfunctor of $\Theta:\MOD(\CR)\ra\MOD(\CA)$, because the functor $\bigoplus_{g\in G}g(\cdot)$ is a subfunctor of $\prod_{g\in G}g(\cdot)$. Moreover, if $T\in\CF(R)$ is finitely presented, then for any $M\in\mod(R)$ we have $T({}^{g}M)\neq 0$ only for finite number of $g\in G$, see Proposition \ref{p2} $(2)$. This yields $$\Phi(T)(F_{\lambda}(M))\cong\bigoplus_{g\in G}T({}^{g}M)\cong\prod_{g\in G}T({}^{g}M)\cong\Theta(T)(F_{\lambda}(M)),$$ so $\Phi(T)\cong\Theta(T)$ and thus $\Phi|_{\CF(R)}\cong\Theta|_{\CF(R)}$. Hence $\Phi$ and $\Theta$ are exact on $\CF(R)$.
\end{proof}

\begin{rem}\label{r3} For the sake of completeness of the exposition we recall the construction of $\Phi:\CF(R)\ra\CF(A)$ as given in Section 5 of \cite{P4}. 

Assume that $F_{\lambda}:\mod(R)\ra\mod(A)$ is dense. Denote by $\Add(\mod(R))$ the full subcategory of $\MOD(R)$ whose objects are arbitrary direct sums of finite dimensional $R$-modules. Assume that $T\in\CF(R)$ and let $\wh{T}:\Add(\mod(R))\ra\Mod(K)$ be the additive closure of $T$. This means that if $\bigoplus_{j\in J}M_{j}\in\Add(\mod(R))$ and $f:\bigoplus_{j\in J}M_{j}\ra\bigoplus_{i\in I}N_{i}$ is an $R$-homomorphism in $\Add(\mod(R))$ defined by $f_{ij}:M_{j}\ra N_{i}$, for $i\in I$, $j\in J$, then $\wh{T}(\bigoplus_{j\in J}M_{j})=\bigoplus_{j\in J}T(M_{j})$ and $$\wh{T}(f):\bigoplus_{i\in I}T(N_{i})\ra\bigoplus_{j\in J}T(M_{j})$$ is defined by $T(f_{ij}):T(N_{i})\ra T(M_{j})$, for $i\in I$, $j\in J$. Obviously $\wh{T}$ equals $T$ on $\mod(R)$.

In \cite[Section 5]{P4} we define the functor $\Phi:\CF(R)\ra\CF(A)$ as the functor $\wh{(\cdot)}\circ F_{\bullet}$ and prove that $$\Phi(T)=\wh{T}\circ F_{\bullet}\cong\Coker{}_{A}(-,F_{\lambda}(f)),$$ for any $T=\Coker{}_{R}(*,f)\in\CF(R)$. The density of $F_{\lambda}$ ensures that there are no $A$-modules of the second kind, hence $F_{\bullet}(\mod(A))\subseteq\Add(\mod(R))$ and the definition is correct. Since $\wh{(\cdot)}\circ F_{\bullet}$ is easily seen to be exact, this gives yet another way to show that $\Phi$ is an exact functor, under the assumption that the push-down functor $F_{\lambda}:\mod(R)\ra\mod(A)$ is dense.\footnote{Observe that in the dense case one can define $\Phi:\MOD(\CR)\ra\MOD(\CA)$ in the same way as we define here $\Phi:\CF(R)\ra\CF(A)$.} Note however that the above point of view is very restrictive (nevertheless, sufficient for purposes of \cite{P4} or \cite{J-PP1}). Indeed, although one can show the existence of the functor $\Phi:\CF(R)\ra\CF(A)$ such that $\Phi(\Coker{}_{R}(*,f))=\Coker{}_{A}(-,F_{\lambda}(f))$ quite easily without the assumption that $F_{\lambda}$ is dense, its exactness is not at all straightforward.\footnote{One could try to show that $\Phi$ is both left and right adjoint, but a rather natural proof of this fact given in \cite[Theorem 5.5 (3)]{P4} seems to work only for a dense $F_{\lambda}$.} Moreover, such a definition cannot be lifted to the level of arbitrary functors in any natural way.\footnote{Similarly, one cannot lift naturally the properties/definitions of functors $\Phi,\Psi:\MOD(\CR)\ra\MOD(\CA)$ as shown in Corollary \ref{c2} to $A$-modules of the second kind, existing when $F_{\lambda}$ is not dense.} On the other hand, general definitions based on the tensor products of functors, as given in Section 3, simply do the job but also seem to reveal the true nature of the functor $\Phi:\CF(R)\ra\CF(A)$ as the restriction of $\Phi:\MOD(\CR)\ra\MOD(\CA)$, being the left adjoint to $\Psi:\MOD(\CA)\ra\MOD(\CR)$, which in turn possesses the right adjoint $\Theta:\MOD(\CR)\ra\MOD(\CA)$. \epv
\end{rem}


\begin{rem}\label{r4} Theorem \ref{t4} shows that when $G$ is an admissible torsion-free group of $K$-linear automorphisms of $R$ and $F:R\ra A\cong R\slash G$ is the Galois covering, then the functor $\Phi:\CF(R)\ra\CF(A)$ is a Galois precovering of functor categories. This result does not require the assumption that $F_{\lambda}:\mod(R)\ra\mod(A)$ is dense. If the latter holds, then $\Phi$ is a subfunctor of $\Theta$, which essentially follows from Theorem \ref{t5}, see also Corollary \ref{c2} and Theorem \ref{t6}. Then we can conclude, as in Theorem \ref{t6}, that these functors coincide on the category of finitely presented functors and hence are both exact. It seems reasonable to expect that this situation can be generalized, at least to some extent, to the setting of arbitrary abelian Krull-Schmidt hom-finite categories with a $G$-action. To be more concrete, we presume that Galois precoverings of such categories with a torsion-free group $G$ induce Galois precoverings on the level of functor categories. The exactness of precoverings on the functor (higher) level can be deduced in the case of dense precoverings on the category (lower) level (i.e. in the case of coverings) similarly as in Theorem \ref{t6}. It is unlikely to expect such an exactness result for arbitrary precoverings, because the proof of Theorem \ref{t2} relies on very specific results, as Lemma \ref{l1} and Proposition \ref{p1}. We leave these interesting general questions for further research. \epv
\end{rem}

\subsection{When Galois coverings preserve Krull-Gabriel dimension}

Theorems \ref{t2} and \ref{t3} show that for any Galois $G$-covering $F:R\ra A$ the functor $\Phi:\CF(R)\ra\CF(A)$ is exact and faithful and hence $\KG(R)\leq\KG(A)$. It is shown in \cite{P6} that the opposite inequality $\KG(A)\leq\KG(R)$ holds in case $R$ is locally support-finite, intervally finite and the group $G$ is torsion-free. Since these conditions imply the density of the push-down $F_\lambda:\mod(R)\ra\mod(A)$, we believe it is natural to present here a sketch of the proof of the theorem. 



We introduce some terminology. Assume that $A,B$ are some locally bounded $K$-categories and let $\varphi\colon\mod(A)\ra\mod(B)$ be a $K$-linear additive covariant functor. We define $\Lambda_{\varphi}\colon\CF(B)\ra\CG(A)$ as the composition functor $(-)\circ\varphi$. Observe that if $U\in\CF(B)$ and ${}_{B}(-,X)\xrightarrow{_{B}(-,f)}{}_{B}(-,Y)\ra U\ra 0$, then we have
$${}_{B}(\varphi(-),X)\xrightarrow{_{B}(\varphi(-),f)}{}_{B}(\varphi(-),Y)\ra U\varphi=\Lambda_{\varphi}(U)\ra 0.$$ We call the functor $\varphi \colon \mod(A)\ra\mod(B)$ \textit{admissible} if and only if $\varphi$ is dense and $\Im(\Lambda_{\varphi})\subseteq\CF(A)$ (the terminology is introduced in \cite{J-PP1}). This means that $\Lambda_{\varphi}$ preserves finitely presented functors. Recall that if $\varphi\colon \mod(A)\ra\mod(B)$ is admissible, then $\KG(B)\leq\KG(A)$, because the functor $\Lambda_{\varphi}\colon \CF(B)\ra\CF(A)$ is exact and faithful, see the remark following \cite[Theorem 1.3]{P6} for details.

Assume that $G$ is an admissible group of $K$-linear automorphisms of $R$. A finite convex subcategory $B$ of the category $R$ is called \textit{fundamental domain} \cite{P6} if and only if for any $M\in\ind(R)$ there exists $g\in G$ such that $\supp({}^{g}M)\subseteq B$. We say that a locally bounded $K$-category $R$ is \textit{intervally finite} \cite[2.1]{BoGa} if and only if a convex hull of any finite full subcategory of $R$ is finite.


\begin{thm} \label{t7}
Assume that $R$ is a locally support-finite and intervally finite locally bounded $K$-category. Assume that $G$ is an admissible torsion-free group of $K$-linear automorphisms of $R$ and let $F:R\ra A\cong R\slash G$ be the Galois $G$-covering. The following assertions hold.
\begin{enumerate}[\rm(1)]
  \item There exists a fundamental domain $B$ of $R$.
  \item The functor $F_{\lambda}\circ\CE_{B}\colon\mod(B)\ra\mod(A)$ is admissible.
  \item We have $\KG(A)\leq\KG(B)\leq\KG(R)$.
\end{enumerate}

\end{thm}

\begin{proof} $(1)$ We refer to \cite[Proposition 1.2]{P6} for the proof of this assertion. We only mention that the interval finiteness of $R$ was crucial to conclude that a convex hull of some finite subcategory of $R$ is finite.\footnote{The existence of a fundamental domain $B$ of $R$ already implies that $R$ is locally support-finite. Indeed, if $R$ is not locally support-finite, then there are indecomposable finite dimensional $R$-modules with arbitrarily large finite supports, and thus there is $M\in\ind(R)$ such that $|\supp(M)|>|B|$. This yields there is no $g\in G$ with $\supp({}^{g}M)\subseteq B$ and so $B$ is not a fundamental domain of $R$.}

$(2)$ Set $G=F_{\lambda}\circ\CE_{B}$ and assume that $U=\Coker{}_{A}(-,\alpha)\in\CF(A)$, for some $A$-homomorphism $\alpha$. We show that $U\circ G\in\CF(B)$. Indeed, recall that $(F_{\lambda},F_{\bullet})$ is an adjoint pair, hence $U\circ F_{\lambda}\cong\Coker{}_{R}(*,F_{\bullet}(\alpha))$ and thus $$U\circ G=(U\circ F_{\lambda})\circ\CE_{B}\cong\Coker{}_{R}(\CE_{B}(*),F_{\bullet}(\alpha)).$$ Recall from \cite{P6} that the functor of extension by zeros $$\CE_{B}\colon\mod(B)\ra\mod(R)\subseteq\Mod(R)$$ has the right adjoint $(-)':\Mod(R)\ra\mod(B)$ such that for any $L\in\Mod(R)$ we have $$L'=\sum_{g:X\ra L}\Im(g)\subseteq L$$ where $X\in\mod(B)$ and $g:X\ra L$ is an $R$-homomorphism. We refer to the proof of \cite[Proposition 1.1]{P6} for the details about this fact. Since $F_{\bullet}(\alpha)$ is an $R$-homomorphism in $\Mod(R)$, we further obtain $$\Coker{}_{R}(\CE_{B}(*),F_{\bullet}(\alpha))\cong\Coker{}_{B}(*,F_{\bullet}(\alpha)')\in\CF(B)$$ which shows that $U\circ G\in\CF(B)$. Now we show that $G=F_{\lambda}\circ\CE_{B}$ is dense. Indeed, assume that $X\in\ind(A)$. Since $F_{\lambda}$ is dense, we get $X\cong F_{\lambda}(M)$, for some $M\in\ind(R)$. Since $B$ is a fundamental domain of $R$, there is $g\in G$ such that $\supp({}^{g}M)\subseteq B$. Thus we obtain isomorphisms $X\cong F_{\lambda}(M)\cong F_{\lambda}({}^{g}M)\cong F_{\lambda}(\CE_{B}({}^{g}M))$ which show that the functor $G$ is dense.

$(3)$  We get $\KG(A)\leq\KG(B)$ by $(2)$, because $G=F_{\lambda}\circ\CE_{B}\colon\mod(B)\ra\mod(A)$ is admissible and thus $\Lambda_{G}:\CF(A)\ra\CF(B)$ is exact and faithful. Moreover, we have $\KG(B)\leq\KG(R)$ by Lemma \ref{0l2}. Hence $\KG(A)\leq\KG(B)\leq\KG(R)$ which finishes the proof of the theorem.
\end{proof}

\begin{rem}\label{r5} Observe that the proof above does apply the torsion-freeness of the group only indirectly. Indeed, this condition seems to be important in the fundamental papers \cite{DoSk} or \cite{DoLeSk} to deduce the density of the push-down $F_\lambda:\mod(R)\ra\mod(A)$ (which is crucial in our argumentation) from the locally support-finiteness of $R$.

\end{rem}

We recall the following theorem which is one of main results of the papers \cite{P4,P6}.


\begin{thm}\label{t8} Assume that $R$ is a locally support-finite and intervally finite locally bounded $K$-category. Assume that $G$ is an admissible torsion-free group of $K$-linear automorphisms of $R$ and let $F:R\ra A\cong R\slash G$ be the Galois $G$-covering. Then there exists a fundamental domain $B$ of $R$ and we have $\KG(R)=\KG(B)=\KG(A)$.
\end{thm}

\begin{proof} The inequalities $\KG(R)\leq\KG(A)\leq\KG(B)\leq\KG(R)$ follow from Theorem \ref{t3} and Theorem \ref{t7}.
\end{proof}

Recall that $(\CF(R)_{\alpha})_{\alpha}$ and $(\CF(A)_{\alpha})_{\alpha}$ denote the Krull-Gabriel filtrations of $\CF(R)$ and $\CF(A)$, respectively. We finish this section with a theorem showing that, under the same assumptions as in Theorems \ref{t7} and \ref{t8}, the functor $\Phi:\CF(R)\ra\CF(A)$ induces Galois precoverings $\Phi:\CF(R)_{\alpha}\ra\CF(A)_{\alpha}$, for any ordinal number $\alpha$. All arguments are implicitly contained in \cite{P4,P6} (see especially \cite[Theorem 1.5]{P6}) but we believe it is convenient to be formulated directly.

\begin{thm}\label{t9} Assume that $R$ is a locally support-finite $K$-category, $G$ an admissible torsion-free group of $K$-linear automorphisms of $R$ and $F:R\ra A\cong R\slash G$ the Galois covering. Assume that $R$ is intervally finite and $B$ is a fundamental domain of $R$. The following assertions hold.
\begin{enumerate}[\rm(1)]
  \item We have $T\in\CF(R)_{\alpha}\Longleftrightarrow\Phi(T)\in\CF(A)_{\alpha}$, for any $T\in\CF(R)$ and ordinal number $\alpha$.
  \item The (restricted) functor $\Phi:\CF(R)_{\alpha}\ra\CF(A)_{\alpha}$ is a Galois precovering, for any ordinal number $\alpha$.
\end{enumerate}
\end{thm}

\begin{proof} (1) Since the functor $\Phi:\CF(R)\ra\CF(A)$ is exact and faithful, we conclude from \cite[Appendix B]{Kr} that $\Phi(T)\in\CF(A)_{\alpha}$ implies $T\in\CF(R)_{\alpha}$, for any $T\in\CF(R)$ and any ordinal $\alpha$, see also the proof of Lemma \ref{0l2}. We show below the converse implication.

Assume that $T\in\CF(R)$ and $T=\Coker{}_{R}(-,f)$ where $f:M\ra N$. Recall that Theorem 5.5 $(2)$ from \cite{P4} (or Theorems \ref{t4} (2) or \ref{t5}) yields $\Phi(T)F_{\lambda}\cong\bigoplus_{g\in G}gT$ and so we obtain $$(\Phi(T)\circ F_{\lambda})\circ\CE_{B}\cong(\bigoplus_{g\in G}gT)\circ\CE_{B}\cong\bigoplus_{g\in G}((gT)\circ\CE_{B}).$$ Observe that $(gT)\circ\CE_{B}\neq 0$ only for finite number of $g\in G$. Indeed, if $(gT)\circ\CE_{B}\neq 0$, then $0\neq {}_{R}(X,{}^{g}N)$, for some indecomposable $X\in\mod(R)$ with $\supp(X)\subseteq B$. Hence $\supp({}^{g}N)\cap B\neq\emptyset$, so the claim follows (see the remark on page 4 of \cite{P6}). Define $\Lambda:\CF(A)\ra\CF(B)$ as $\Lambda=(-)\circ (F_{\lambda}\circ\CE_{B})$ and observe that the above arguments yield a finite set $H\subseteq G$ such that $$\Lambda(\Phi(T))\cong(\bigoplus_{g\in H}gT)\circ\CE_{B}\cong r_{B}(\bigoplus_{g\in H}gT)$$ where $r_{B}:\CF(R)\ra\CF(B)$ is the restriction functor. Assume that $\alpha$ is an ordinal number and $\Phi(T)\notin\CF(A)_{\alpha}$. Since the functor $\Lambda$ is exact and faithful, we get $\Lambda(\Phi(T))\notin\CF(B)_{\alpha}$ and thus we obtain $r_{B}(\bigoplus_{g\in H}gT)\notin\CF(B)_{\alpha}$. Since the functor $r_{B}$ is exact, full and dense, we conclude from \cite[Appendix B]{Kr} that $\bigoplus_{g\in H}gT\notin\CF(R)_{\alpha}$. The category $\CF(R)_{\alpha}$ is closed under direct summands, so $gT\notin\CF(R)_{\alpha}$, for some $g\in H$, and finally $T\notin\CF(R)_{\alpha}$ since $\CF(R)_{\alpha}$ is closed under the action of $G$ by \cite[Proposition 6.1 (2)]{P4}.

(2) It follows from \cite[Proposition 6.1]{P4} that the Krull-Gabriel filtration of the category $\CF(R)$ is $G$-invariant. Theorem \ref{t4} implies that $\Phi:\CF(R)\ra\CF(A)$ is a Galois precovering. Moreover, the assertion of $(1)$ yields $\Phi(\CF(R)_{\alpha})\subseteq\CF(A)_{\alpha}$, for any ordinal number $\alpha$. These three fact altogether show the thesis.  
\end{proof}

\begin{rem}\label{r6} Theorem \ref{t9} (1) gives yet another proof that $\KG(R)=\KG(A)$. Indeed, we show in \cite[Lemma 3.2]{P4} a general fact that for any exact functor $F:\CC\ra\CD$ between abelian categories $\CC,\CD$ such that for any $U\in\CD$ there is an epimorphism $\epsilon:F(T)\ra U$ we have $\KG(\CC)=\KG(\CD)$ if the condition $$T\in\CC_{\alpha}\Longleftrightarrow F(T)\in\CD_{\alpha}$$ is satisfied, for any ordinal number $\alpha$.\footnote{Recall that $(\CC_{\alpha})_{\alpha},(\CD_{\alpha})_{\alpha}$ are Krull-Gabriel filtrations of $\CC,\CD$, respectively.} Since $F_{\lambda}$ is dense, we get that any $U\in\CF(A)$ is of the form $\Coker{}_{A}(-,\alpha)\in\CF(A)$ for some $\alpha:F_{\lambda}(X)\ra F_{\lambda}(Y)$. Thus there is an epimorphism $\epsilon:{}_{A}(-,F_{\lambda}(Y))\ra U$ and \cite[Lemma 3.2]{P4} directly yields $\KG(R)=\KG(A)$. 
\end{rem}

\section{The functors of the first and the second kind}

This section presents examples of Galois coverings $F:R\ra A$ for which the functor $\Phi:\CF(R)\ra\CF(A)$ is either dense or not. In these examples $R$ is a simply connected locally representation-finite locally bounded $K$-category. This situation simplifies considerations, because the locally representation-finiteness of $R$ implies that the categories $\ind(R),\ind(A)$ are locally bounded (see Proposition \ref{p4}), $F_{\lambda}:\ind(R)\ra\ind(A)$ is a Galois covering itself (see Theorem \ref{t10}) and $\Phi\cong (F_\lambda)_{\lambda}$ (see Proposition \ref{p5}). Furthermore, the simple connectedness of $R$ allows to view $F_{\lambda}:\ind(R)\ra\ind(A)$ as a Galois covering of the associated mesh-categories and so $(F_\lambda)_{\lambda}\cong\Phi$ becomes a Galois covering on the module level which is much better understood, see Theorem \ref{t11}. Based on Theorem \ref{t11} we show in Example \ref{e1} that the functor $\Phi:\CF(R)\ra\CF(A)$ is not dense in general. In this way we confirm the conjecture formulated in \cite[Remark 5.6]{P4}. 

Therefore it is natural to introduce the following terminology, originated in \cite{DoSk} for the case of modules over locally bounded $K$-categories.

\begin{df}\label{d1} Assume that $F:R\ra A$ is a Galois $G$-covering with a torsion-free group $G$ and let $\Phi_{F}:\CF(R)\ra\CF(A)$ be the associated Galois $G$-precovering of functor categories. We call an indecomposable functor $U\in\CF(A)$ a \emph{functor of the first kind} if and only if $U$ lies in the image of $\Phi_{F}$. Otherwise, we call $U$ a \emph{functor of the second kind}.
\end{df}

Recall from Proposition \ref{p3} that categories of finitely presented functors are abelian Krull-Schmidt categories. This implies that the density of $\Phi_{F}:\CF(R)\ra\CF(A)$ is equivalent with the condition that any functor $U\in\CF(A)$ is of the first kind.


\subsection{The locally representation-finite case}

We formulate the following well-known fact and give its simple proof for convenience.

\begin{prop}\label{p4} Assume that $R$ is a locally bounded $K$-category. Then $R$ is locally representation-finite if and only if the category $\ind(R)$ is locally bounded.
\end{prop}

\begin{proof} It is easy to see that $\ind(R)$ is locally bounded if and only if the hom-functors ${}_{R}(*,X)$ and ${}_{R}(X,*)$ have finite supports, for any $X\in\ind(R)$. 

Assume that $R$ is locally representation-finite and let $X\in\ind(R)$. We show that the functor ${}_{R}(*,X)$ has finite support (arguments for the functor ${}_{R}(X,*)$ are analogous). Indeed, let $R_a$ be the finite set of all indecomposable modules in $\mod(R)$ which are nonzero in $a\in\ob(R)$. Let $\CX$ be the union of all the sets $R_{a}$ such that $a\in\supp(X)$. If $Y\in\ind(R)$ and ${}_{R}(Y,X)\neq 0$, then $\supp(X)\cap\supp(Y)\neq\emptyset$ and hence $Y\in\CX$. Since $\CX$ is finite, we conclude that ${}_{R}(*,X)$ has finite support. Thus $\ind(R)$ is locally bounded.

On the other hand, assume that the hom-functors have finite supports. We prove that $R$ is locally representation-finite. Let $a\in\ob(R)$ and recall that $X(a)\cong{}_{R}(P_a,X)$, for any indecomposable module $X\in\mod(R)$, and hence $X(a)\neq 0$ if and only if ${}_{R}(P_a,X)\neq 0$. Since the support of the functor ${}_{R}(P_a,*)$ is finite, we conclude that there is a finite number of indecomposable $R$-modules $X\in\mod(R)$ with $X(a)\neq 0$. Thus $R$ is locally representation-finite. 
\end{proof}

We recall the following classical result of covering theory of locally bounded locally representation finite $K$-categories, see \cite{Ga,BoGa,BrGa,MP} for the details. We denote by $\Gamma_{R}$ the Auslander-Reiten quiver of a locally bounded $K$-category $R$ whose vertices are isomorphism classes of indecomposable $R$-modules, denoted by $[X]$, for any indecomposable module $X\in\mod(R)$.

\begin{thm}\label{t10} Assume that $R$ is a locally representation-finite locally bounded $K$-category, $G$ an admissible torsion-free group $K$-linear automorphisms of $R$ and $F:R\ra A$ the associated Galois covering. The following assertions hold.

\begin{enumerate}[\rm(1)]
 \item The category $A$ is representation-finite.
	\item The push-down functor $F_{\lambda}:\ind(R)\ra\ind(A)$ is a Galois covering of locally bounded $K$-categories which induces a Galois covering $\Gamma_{R}\ra\Gamma_{A}$ of translation quivers such that $[X]\mapsto[F_{\lambda}(X)]$, for any $X\in\ind(R)$.
\end{enumerate}
\end{thm}

If $F_{\lambda}:\ind(R)\ra\ind(A)$ is a Galois covering of locally bounded $K$-categories, as in the above theorem, $F_\lambda$ induces the push-down functor $$(F_{\lambda})_{\lambda}:\mod(\ind(R))\ra\mod(\ind(A)).$$ It turns out that the functor $(F_{\lambda})_{\lambda}$ is equivalent with the functor $\Phi:\CF(R)\ra\CF(A)$, denoted in the sequel by $\Phi_{F}$. This is shown in the next proposition. 

\begin{prop}\label{p5} Assume that $R$ is a locally representation-finite locally bounded $K$-category. The following assertions hold.
\begin{enumerate}[\rm(1)]
	\item There is an exact equivalence of categories $\eta_{R}:\mod(\ind(R))\ra\CF(R)$.
	\item Assume that $G$ is an admissible torsion-free group of $K$-linear automorphisms of $R$ and $F:R\ra A$ is the associated Galois covering. Then the following diagram $$\xymatrix{\mod(\ind(R))\ar[rr]^{(F_{\lambda})_{\lambda}}\ar[d]^{\eta_{R}}&&{}\mod(\ind(A))\ar[d]^{\eta_{A}}\\\CF(R)\ar[rr]^{\Phi_{F}}&& \CF(A)}$$ is commutative. In consequence, $(F_{\lambda})_{\lambda}\cong\Phi_{F}$.
\end{enumerate}
\end{prop}

\begin{proof}
$(1)$ It is well known that $\eta_{R}:\mod(\ind(R))\ra\CF(R)$ is the functor of additive closure, but we sketch the construction for convenience. Assume that $T\in\mod(\ind(R))$ and define $\eta_{R}(T):\mod(R)\ra\mod(K)$ as the unique contravariant $K$-linear additive functor such that $\eta_{R}(T)(X)=T(X)$, for any $X\in\ind(R)$. In particular, we have $$\eta_{R}(T)(X_{1}\oplus\dots\oplus X_{n})=T(X_{1})\oplus\dots\oplus T(X_{n}),$$ for any $n\in\NN$ and $X_{1},\dots,X_{n}\in\ind(R)$. The functor $\eta_{R}$ is defined on homomorphisms $f\in\mod(\ind(R))$ in a natural way and thus we obtain a functor $\eta_{R}:\mod(\ind(R))\ra\CG(R)$ which is easily seen to be exact. We show that $\Im(\eta_{R})\subseteq\CF(R)$. Indeed, observe that $\eta_{R}$ preserves projective objects, because $$\eta_{R}(\ind(R)(-,M))={}_{R}(-,M),$$ for any $M\in\mod(R)$. Since $\eta_{R}$ is exact, we conclude that it preserves projective presentations and so we get an exact functor $\eta_{R}:\mod(\ind(R))\ra\CF(R)$. This functor is an equivalence, because it possesses the quasi-inverse which sends a functor $T\in\CF(R)$ to the restriction $T|_{\ind(R)}$.


$(2)$ We conclude from Theorem \ref{t10} (1) that $A$ is representation-finite and so there exists the equivalence $\eta_{A}:\mod(\ind(A))\ra\CF(A)$. Assume that $T\in\mod(\ind(R))$. Then the minimal projective presentation $p$ of $T$ in $\mod(\ind(R))$ has the form $$p:\quad\ind(R)(-,M)\ra\ind(R)(-,N)\ra T\ra 0,$$ for some $M,N\in\mod(R)$. Since $(F_{\lambda})_{\lambda}$ is exact and preserves projectivity (as any push-down functor), we get that $$(F_{\lambda})_{\lambda}(p):\quad\ind(A)(-,F_{\lambda}(M))\ra\ind(A)(-,F_{\lambda}(N))\ra (F_{\lambda})_{\lambda}(T)\ra 0$$ is a projective presentation of the $\ind(A)$-module $(F_{\lambda})_{\lambda}(T)$ in $\mod(\ind(A))$. Applying the functors $\eta_{R}$ and $\eta_{A}$ to $p$ and $(F_{\lambda})_{\lambda}(p)$, respectively, we obtain the following two projective presentations: $$\eta_{R}(p):\quad{}_{R}(-,M)\ra{}_{R}(-,N)\ra\eta_{R}(T)\ra 0,$$ $$\eta_{A}((F_{\lambda})_{\lambda}(p)):\quad{}_{A}(-,F_{\lambda}(M))\ra{}_{A}(-,F_{\lambda}(N))\ra\eta_{A}((F_{\lambda})_{\lambda}(T))\ra 0$$ which immediately yields $$\Phi_{F}(\eta_{R}(T))=\eta_{A}((F_{\lambda})_{\lambda}(T)).$$ Similar arguments imply the commutativity of the diagram on morphisms, because any homomorphism $\iota:T_{1}\ra T_{2}$ in $\mod(\ind(R))$ can be lifted to a morphism of the associated minimal projective presentations.
\end{proof}

The above fact is most useful in the case when $R$ is simply connected. Our aim is to present a version of Proposition \ref{p5} in that special situation.

Recall that a locally bounded $K$-category $R$ is \textit{simply connected} \cite{AsSk3} provided that, for any presentation $R\cong\underline{KQ}\slash I$ as a bound quiver $K$-category: 
\begin{itemize}
  \item the quiver $Q$ is \textit{triangular}, i.e. $Q$ has no oriented cycles,
  \item the fundamental group $\Pi_{1}(Q,I)$ is trivial, i.e. $\Pi_{1}(Q,I)=\{1\}$.
\end{itemize} The reader is referred to \cite{Ga} for the appropriate definitions. It follows from \cite{SkBC} that $R$ is simply connected if and only if $R$ is triangular and has no proper Galois coverings.\footnote{$R$ is \emph{triangular} if and only if there is a presentation $R\cong\underline{KQ}\slash I$ with $Q$ triangular.} If $R$ is locally representation finite, then $R$ is simply connected if and only if the fundamental group $\Pi_{1}(\Gamma_{R})$ of $\Gamma_{R}$ is trivial.

For a locally bounded $K$-category $R$, we denote by $K(\Gamma_{R})$ the \textit{mesh-category} \cite{Rie} of the Auslander-Reiten quiver $\Gamma_{R}$ of $R$. Recall that $K(\Gamma_{R})=\underline{K(\Gamma_{R}^{\op})}\slash I$ where $I$ is the admissible ideal in the path category $\underline{K(\Gamma_{R}^{\op})}$ generated by all the mesh-relations appearing in $\Gamma_{R}^{\op}$. We say that $R$ is \textit{standard} if and only if there exists a $K$-linear equivalence of categories $\phi_{R}:\ind(R)\ra K(\Gamma_{R})$ such that $\phi_{R}(X)=[X]$, for any $X\in\ind(R)$. It follows from \cites{BrGa,BoGa} that $R$ is standard if and only if $R$ admits a simply connected Galois covering. In particular, if $R$ is simply connected, then $R$ is standard.

\begin{rem}\label{p6} If $R$ is locally representation-finite, then $\ind(R)$ is locally bounded by Proposition \ref{p4} and hence $\ind(R)$ has a presentation $\underline{KQ}\slash I$ as a bound quiver $K$-category. It is known that $Q=\Gamma_{R}^{\op}$ and $I$ contains all the mesh-relations. However $I$ may contain more relations, so in general $\ind(R)$ is not equivalent with the mesh category $K(\Gamma_{R})$. This holds if $R$ is standard and in the particular case when $R$ is simply connected.
\end{rem}


We get the following handy description of the functor $\Phi:\CF(R)\ra\CF(A)$ in the special case when $R$ is locally representation-finite and simply connected.



  

\begin{thm}\label{t11} Assume that $R$ is a locally representation-finite simply connected locally bounded $K$-category. Assume $G$ is an admissible torsion-free group of $K$-linear automorphisms of $R$ and $F:R\ra A$ is the associated Galois covering. The following assertions are satisfied.
\begin{enumerate}[\rm(1)]
  \item There exists a Galois covering $F_{\lambda}^{\Gamma}:K(\Gamma_{R})\ra K(\Gamma_{A})$ such that the following diagram $$\xymatrix{\ind(R)\ar[rr]^{F_{\lambda}}\ar[d]^{\phi_{R}}&& \ind(A)\ar[d]^{\phi_{A}}\\K(\Gamma_{R})\ar[rr]^{F_{\lambda}^{\Gamma}}&&{}K(\Gamma_{A})}$$ commutes.
  \item The following diagram $$\xymatrix{\mod(K(\Gamma_{R}))\ar[rr]^{(F_{\lambda}^{\Gamma})_{\lambda}}\ar[d]^{\widetilde{\phi}_{R}}&&\mod(K(\Gamma_{A}))\ar[d]^{\widetilde{\phi}_{A}}\\\mod(\ind(R))\ar[rr]^{(F_{\lambda})_\lambda}\ar[d]^{\eta_R}&&\mod(\ind(A))\ar[d]^{\eta_{A}}\\\CF(R)\ar[rr]^{\Phi_{F}}&& \CF(A),}$$ where $\widetilde{\phi}_{R}=(-)\circ\phi_{R}$ and $\widetilde{\phi}_{A}=(-)\circ\phi_{A}$, is a commutative diagram whose columns are equivalences.
\end{enumerate}

\end{thm}


\begin{proof} (1) First observe that both $R$ and $A$ are standard, because $R$ is simply connected. Hence there are $K$-linear equivalences $\phi_{R}:\ind(R)\ra K(\Gamma_{R})$ and $\phi_{A}:\ind(R)\ra K(\Gamma_{A})$, sending indecomposable modules to the associated vertices in the Auslander-Reiten quivers. Recall from Theorem \ref{t10} that the push-down functor $F_{\lambda}:\ind(R)\ra\ind(A)$ induces a Galois covering $\Gamma_{R}\ra\Gamma_{A}$ of translation quivers such that $[X]\mapsto[F_{\lambda}(X)]$, for any $X\in\ind(R)$. Hence we get a Galois covering $\Gamma_{R}^{\op}\ra\Gamma_{A}^{\op}$ of the opposite quivers and thus a Galois covering $K(\Gamma_{R})\ra K(\Gamma_{A})$ of the associated mesh-categories. Let us denote this Galois covering as $F_{\lambda}^{\Gamma}:K(\Gamma_{R})\ra K(\Gamma_{A})$. It follows from the definition that $F_{\lambda}^{\Gamma}([X])=[F_{\lambda}(X)]$, for any $X\in\ind(R)$, which is equivalent with $F_{\lambda}^{\Gamma}\circ\phi_R=\phi_{A}\circ F_{\lambda}$. This shows the commutativity.

(2) Since $\phi_R$ and $\phi_A$ are equivalences of $K$-categories, we conclude that $\widetilde{\phi}_{R}=(-)\circ\phi_{R}$ and $\widetilde{\phi}_{A}=(-)\circ\phi_{A}$ are $K$-linear equivalences on the module level. Then it follows from (1) that the following diagram $$\xymatrix{\mod(K(\Gamma_{R}))\ar[rr]^{(F_{\lambda}^{\Gamma})_{\lambda}}\ar[d]^{\widetilde{\phi}_{R}}&&\mod(K(\Gamma_{A}))\ar[d]^{\widetilde{\phi}_{A}}\\\mod(\ind(R))\ar[rr]^{(F_{\lambda})_{\lambda}}&&\mod(\ind(A))}$$ is commutative.\footnote{It is clear that an equivalence of Galois coverings induces an equivalence of the associated push-down functors. Thus we omit a bit tedious calculations showing that $F_{\lambda}^{\Gamma}\circ\phi_R=\phi_{A}\circ F_{\lambda}$ implies that $(F_\lambda)_{\lambda}\circ\widetilde{\phi}_{R}=\widetilde{\phi}_{A}\circ(F_{\lambda}^{\Gamma})_{\lambda}$.} This fact, together with Proposition \ref{p5} (2), implies that the diagram of (2) is a commutative diagram whose columns are equivalences.
\end{proof}

\subsection{Examples}

Theorem \ref{t11} shows that in some nice situations the functor $\Phi_{F}:\CF(R)\ra\CF(A)$ can be described in terms of classical covering theory. We use this fact below in Example \ref{e1} to show that $\Phi_{F}$ is not dense in general.\footnote{We are grateful to Stanisław Kasjan for drawing our attention to this example.} In this way we confirm our conjecture formulated in \cite[Remark 5.6]{P4}. Afterwards we give two other examples in which $\Phi_{F}$ is a dense functor. In these examples we only sketch the arguments.


\begin{exa}\label{e1} Assume that $R=\underline{KQ_{R}}\slash I_{R}$ where $Q_{R}$ is the following quiver: $$
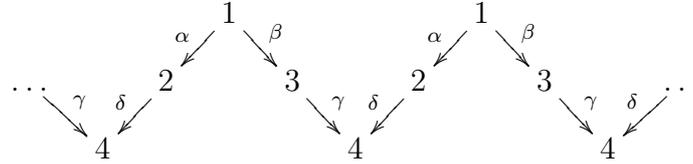
\xymatrix@C=1em@R=1em{&&&1\ar[dr]^{\beta}\ar[dl]_{\alpha}&&&&1\ar[dr]^{\beta}\ar[dl]_{\alpha}\\\dots\ar[dr]^{\gamma}&&2\ar[dl]_{\delta}&&3\ar[dr]^{\gamma}&&2\ar[dl]_{\delta}&&3\ar[dr]^{\gamma}&&\dots\ar[dl]_{\delta}\\&4&&&&4&&&&4}$$ and $I_{R}$ is generated by all zero relations of the form $\alpha\delta=0$ and $\beta\gamma=0$. The group $G=\mathbb{Z}$ acts on $R$ by the horizontal translation, indicated on the picture by labelling different elements of the same orbit by the same number or letter. Hence $R\slash G$ is the bound-quiver $K$-algebra $A=KQ_{A}\slash I_{A}$ where $Q_{A}$ is the quiver: $$
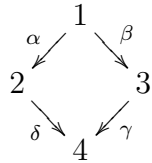
\xymatrix@C=1em@R=1em{&1\ar[dl]_{\alpha}\ar[dr]^{\beta}\\2\ar[dr]_{\delta}&&3\ar[dl]^{\gamma}\\&4}$$ and $I_{A}=\langle\alpha\delta,\beta\gamma\rangle$. We obtain a Galois covering functor $F:R\ra A$ with the covering group $G=\mathbb{Z}$. Observe that $R$ is locally representation-finite simply connected, as $Q_R$ is a quiver of type $\mathbb{A}^{\infty}_{\infty}$, and thus $A$ is standard (however, not simply connected). Below we present the Auslander-Reiten quivers $\Gamma_{R},\Gamma_{A}$ (one next to the other) and explain some additional data that we marked on these quivers.

$$
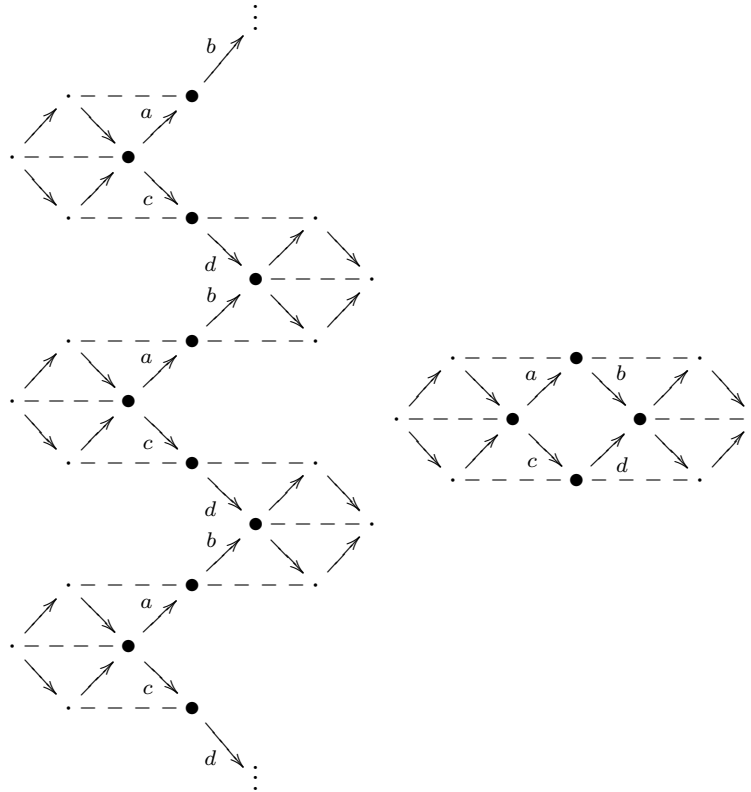
\xymatrix@C=1em@R=1em{&&&&\vdots\\&\cdot\ar[dr]\ar@{--}[rr]&&\bullet\ar[ur]^{b}\\\cdot\ar[ur]\ar[dr]\ar@{--}[rr]&&\bullet\ar[dr]_{c}\ar[ur]^{a}\\&\cdot\ar[ur]\ar@{--}[rr]&&\bullet\ar[dr]_{d}\ar@{--}[rr]&&\cdot\ar[dr]\\&&&&\bullet\ar[ur]\ar[dr]\ar@{--}[rr]&&\cdot\\&\cdot\ar[dr]\ar@{--}[rr]&&\bullet\ar[ur]^{b}\ar@{--}[rr]&&\cdot\ar[ur]\\\cdot\ar[ur]\ar[dr]\ar@{--}[rr]&&\bullet\ar[dr]_{c}\ar[ur]^{a}\\&\cdot\ar[ur]\ar@{--}[rr]&&\bullet\ar[dr]_{d}\ar@{--}[rr]&&\cdot\ar[dr]\\&&&&\bullet\ar[ur]\ar[dr]\ar@{--}[rr]&&\cdot\\&\cdot\ar[dr]\ar@{--}[rr]&&\bullet\ar[ur]^{b}\ar@{--}[rr]&&\cdot\ar[ur]\\\cdot\ar[ur]\ar[dr]\ar@{--}[rr]&&\bullet\ar[dr]_{c}\ar[ur]^{a}\\&\cdot\ar[ur]\ar@{--}[rr]&&\bullet\ar[dr]_{d}\\&&&&\vdots} \xymatrix@C=1em@R=1em{\\\\\\\\\\\\\\&\cdot\ar[dr]\ar@{--}[rr]&&\bullet\ar[dr]^{b}\ar@{--}[rr]&&\cdot\ar[dr]\\\cdot\ar[ur]\ar[dr]\ar@{--}[rr]&&\bullet\ar[ur]^{a}\ar[dr]_{c}&&\bullet\ar[ur]\ar@{--}[rr]\ar[dr]&&\cdot\\&\cdot\ar[ur]\ar@{--}[rr]&&\bullet\ar[ur]_{d}\ar@{--}[rr]&&\cdot\ar[ur]}$$



As usual, the dotted lines on the above quivers represent the existence of the Auslander-Reiten sequences. We see that $\Gamma_{R}$ contains a convex line $L$ whose vertices are marked as $\bullet$ and arrows are labelled by the letters $a,b,c,d$. We denote by $D$ the convex subquiver of $\Gamma_{A}$ of the type $\widetilde{\mathbb{A}}_{3}$ whose vertices and arrows are marked on $\Gamma_{A}$ in a similar way.

We view $\Gamma_{R}$ and $\Gamma_{A}$ as bound quiver $K$-categories, bounded by all mesh-relations. More specifically, we identify $\Gamma_{R}$ with $K(\Gamma_{R}^{\op})=K(\Gamma_{R})^{\op}$ and $\Gamma_{A}$ with $K(\Gamma_{A}^{\op})=K(\Gamma_{A})^{\op}$. Observe that the group $G=\mathbb{Z}$ acts on $\Gamma_{R}$ by vertical translation, that is, all meshes lying on each of two vertical lines (on both sides of the line $L$) belong to the same orbit. Thus indeed $\Gamma_{R}\slash G\cong\Gamma_{A}$. Denote the Galois covering $\Gamma_{R}\ra\Gamma_{A}$ by $\widetilde{F}$ and observe that $\widetilde{F}$ is equivalent with $(F_{\lambda}^{\Gamma})^{\op}$. Furthermore, denote by $M(L)$ the linear $\Gamma_{R}$-module associated with the line $L$ and by $M(D)$ the $\Gamma_{A}$-module having $K$ in the vertices and identities on the arrows such that $\supp(M(D))=D$. Then $\widetilde{F}_{\bullet}(M(D))=M(L)$ and thus $\widetilde{F}_{\bullet}(M(D))$ is an indecomposable module of infinite dimension. This implies that the module $M(D)$ is of the second kind, because otherwise $\widetilde{F}_{\bullet}(M(D))$ would be of the form $\bigoplus_{g\in G}{}^{g}L$, for some indecomposable $L\in\mod(R)$. Consequently, the functor $\widetilde{F}_{\lambda}$ is not dense. Since $\widetilde{F}^{\op}\cong F_{\lambda}^{\Gamma}$, Theorem \ref{t11} yields $$(\widetilde{F}_{\lambda})^{\op}\cong(\widetilde{F}^{\op})_\lambda\cong (F_{\lambda}^{\Gamma})_{\lambda}\cong\Phi_F$$ and thus the functor $\Phi_{F}:\CF(R)\ra\CF(A)$ is not dense either.


These arguments are sufficient to show that $\Phi_{F}:\CF(R)\ra\CF(A)$ is not a dense functor. However, we would like to give more detailed description of a functor $U\in\CF(A)$ which is of the second kind. For this purpose we need more calculations. First observe that $\Gamma_{A}$ has the following form $$
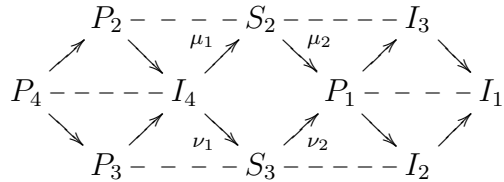
\xymatrix@C=1em@R=1em{&P_2\ar[dr]\ar@{--}[rr]&&S_2\ar[dr]^{\mu_{2}}\ar@{--}[rr]&& I_3\ar[dr]\\ P_4\ar[ur]\ar[dr]\ar@{--}[rr]&& I_4\ar[ur]^{\mu_{1}}\ar[dr]_{\nu_{1}}&& P_1\ar[ur]\ar@{--}[rr]\ar[dr]&& I_1\\& P_3\ar[ur]\ar@{--}[rr]&& S_3\ar[ur]_{\nu_{2}}\ar@{--}[rr]&& I_2\ar[ur]}$$ where $P_j,I_j,S_j$ denote the indecomposable projective, injective and simple $A$-modules, respectively, associated with the vertex $j\in\{1,2,3,4\}$ and $\mu_{1},\mu_{2},\nu_{1},\nu_{2}$ are the appropriate irreducible homomorphisms. Let us depict $\Gamma_{A}^{\op}$ in the following way $$
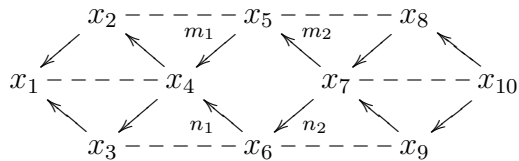
\xymatrix@C=1em@R=1em{&x_{2}\ar[dl]\ar@{--}[rr]&& x_{5}\ar[dl]_{m_{1}}\ar@{--}[rr]&& x_{8}\ar[dl]\\ x_{1}\ar@{--}[rr]&& x_{4}\ar[ul]\ar[dl]&& x_{7}\ar[ul]_{m_{2}}\ar@{--}[rr]\ar[dl]^{n_{2}}&&x_{10}\ar[dl]\ar[ul]\\& x_{3}\ar[ul]\ar@{--}[rr]&&x_{6}\ar[ul]^{n_{1}}\ar@{--}[rr]&& x_{9}\ar[ul]}$$ and set $B=K(\Gamma_{A})$. Let $P_{x_{4}}=B(-,x_{4})$ and $P_{x_{7}}=B(-,x_{7})$ be the indecomposable projective $B$-modules. Then the element $s:=m_{2}m_{1}+n_{2}n_{1}\in B(x_{4},x_{7})$ induces a homomorphism $(-,s):P_{x_{4}}\ra P_{x_{7}}$ of $B$-modules which cokernel is the $B$-module $N$ having $K$ in the vertices and identities on the arrows such that $\supp(N)=\{x_{4},x_{5},x_{6},x_{7}\}$. Obviously $N$ is the dual of the module $M(D)$ and thus it is of the second kind. Therefore we have a projective presentation $$p:\quad P_{x_{4}}\xrightarrow{(-,s)} P_{x_{7}}\ra N\ra 0$$ in $\mod(B)$ and it follows easily that $$(\eta_{A}\circ\wt{\phi}_{A})(p):\quad{}_{A}(-,I(4))\xrightarrow{{}_{A}(-,f)}{}_{A}(-,P(1))\ra(\eta_{A}\circ\wt{\phi}_{A})(N)\ra 0$$ where $f=\mu_{2}\mu_{1}+\nu_{2}\nu_{1}$ and $\eta_{A}\circ\wt{\phi}_{A}:\mod(K(\Gamma_{A}))\ra\CF(A)$ is the equivalence of Theorem \ref{t11}. Hence $$U=(\eta_{A}\circ\wt{\phi}_{A})(N)\in\CF(A)$$ is a functor of the second kind. In fact, similar arguments imply that, for any non-zero $\lambda\in K$, the functor $U_{\lambda}=\Coker{}_{A}(-,f_{\lambda})$, where $f_{\lambda}=\mu_{2}\mu_{1}+\lambda\nu_{2}\nu_{1}$, is of the second kind. These functors are pairwise non-isomorphic and so we obtain a one parameter family of functors of the second kind in the functor category $\CF(A)$. \epv


\end{exa}

\begin{exa}\label{e2} Assume that $R=\underline{KQ_{R}}\slash I_{R}$ where $Q_{R}$ is the following quiver: $$
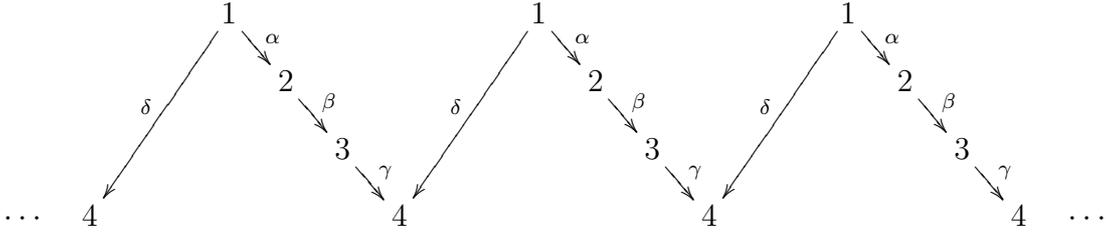
\xymatrix@C=0.8em@R=1em{&&&&1\ar[dr]^{\alpha}\ar[dddlll]_{\delta}&&&&&&1\ar[dr]^{\alpha}\ar[dddlll]_{\delta}&&&&&&1\ar[dr]^{\alpha}\ar[dddlll]_{\delta}&\\&&&&& 2\ar[dr]^{\beta}&&&&&& 2\ar[dr]^{\beta} &&&&&& 2\ar[dr]^{\beta}&&&\\&&&&&& 3\ar[dr]^{\gamma} &&&&&& 3\ar[dr]^{\gamma} &&&&&& 3\ar[dr]^{\gamma}&&\\\cdots &4 &&&&&& 4 &&&&&& 4 &&&&&& 4& \cdots}$$ and $I_{R}$ is generated by all zero relations of the form $\alpha\beta=\beta\gamma=0$. The group $G=\mathbb{Z}$ acts on $R$ by the horizontal translation and hence $R\slash G$ is the bound-quiver $K$-algebra $A=KQ_{A}\slash I_{A}$ where $Q_{A}$ is the quiver: $$
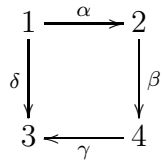
\xymatrix@C=1em@R=1em{1\ar[dd]_{\delta}\ar[rr]^{\alpha}&& 2\ar[dd]^{\beta}\\&&\\3&&4\ar[ll]^{\gamma}}$$ and $I_{A}=\langle\alpha\beta,\beta\gamma\rangle$. We obtain a Galois covering functor $F:R\ra A$ with the covering group $G=\mathbb{Z}$ where $R$ is locally representation-finite simply connected and thus $A$ is standard. The Auslander-Reiten quivers $\Gamma_{R}$ and $\Gamma_{A}\cong\Gamma_{R}\slash G$ have the following shapes, respectively: $$
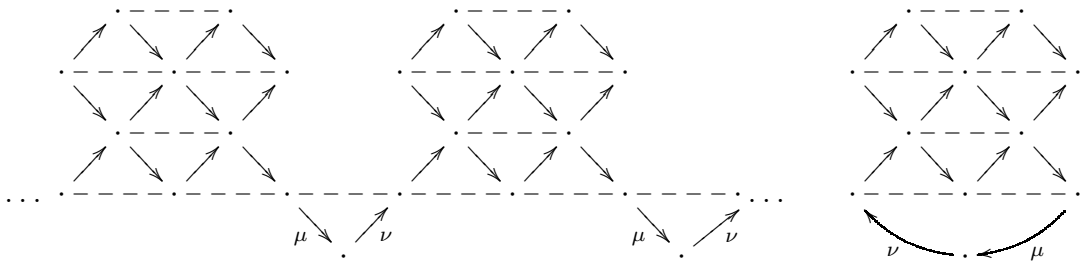
\xymatrix@C=1em@R=0.95em{\\ \\ \\ \\  \dots\\} \xymatrix@C=1em@R=1em{&\cdot\ar[dr]\ar@{--}[rr]&&\cdot\ar[dr] &&&& \cdot\ar[dr]\ar@{--}[rr]&&\cdot\ar[dr]\\
 \cdot\ar[ur]\ar[dr]\ar@{--}[rr]&&\cdot\ar[ur]\ar[dr]\ar@{--}[rr]&&\cdot &&  \cdot\ar[ur]\ar[dr]\ar@{--}[rr]&&\cdot\ar[ur]\ar[dr]\ar@{--}[rr]&&\cdot\\&\cdot\ar[ur]\ar[dr]\ar@{--}[rr]&&\cdot\ar[ur]\ar[dr] &&&& \cdot\ar[ur]\ar[dr]\ar@{--}[rr]&&\cdot\ar[ur]\ar[dr]& && & \\
 \cdot\ar[ur]\ar@{--}[rr]&&\cdot\ar[ur]\ar@{--}[rr]&&\cdot\ar[dr]_{\mu}\ar@{--}[rr]&&\cdot\ar[ur]\ar@{--}[rr]&&\cdot\ar[ur]\ar@{--}[rr]&&\cdot\ar[dr]_{\mu}\ar@{--}[rr]&&\cdot \dots\\ &&&&&\cdot\ar[ur]_{\nu}&& &&&& \cdot \ar[ur]_{\nu}&}\xymatrix@C=1em@R=1em{&\cdot\ar[dr]\ar@{--}[rr]&&\cdot\ar[dr]\\\cdot\ar[ur]\ar[dr]\ar@{--}[rr]&&\cdot\ar[ur]\ar[dr]\ar@{--}[rr]&&\cdot\\&\cdot\ar[ur]\ar[dr]\ar@{--}[rr]&&\cdot\ar[ur]\ar[dr]\\\cdot\ar[ur]\ar@{--}[rr]&&\cdot\ar[ur]\ar@{--}[rr]&&\cdot\ar@/^/[dll]^{\tiny{\mu}}\\& & \cdot\ar@/^/[ull]^{\tiny{\nu}} & &}$$ Set $B=K(\Gamma_{R}^{\op})$ and observe that $B$ is locally support-finite. Indeed, we always have $\nu\mu=0$ so the support of any indecomposable finite dimensional $B$-module is contained in a full subcategory $C$ of $B$ of the form: $$
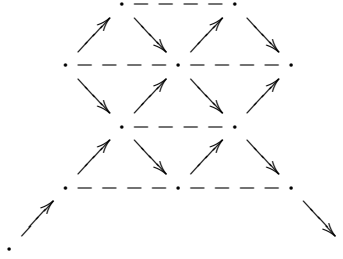
\xymatrix@C=1em@R=1em{&&\cdot\ar[dr]\ar@{--}[rr]&&\cdot\ar[dr]\\&\cdot\ar[ur]\ar[dr]\ar@{--}[rr]&&\cdot\ar[ur]\ar[dr]\ar@{--}[rr]&&\cdot\\&&\cdot\ar[ur]\ar[dr]\ar@{--}[rr]&&\cdot\ar[ur]\ar[dr]\\&\cdot\ar[ur]\ar@{--}[rr]&&\cdot\ar[ur]\ar@{--}[rr]&&\cdot\ar[dr]\\ \cdot\ar[ur]&&&&&&\cdot}$$ Denoting the Galois covering $\Gamma_{R}\ra\Gamma_{A}$ by $\widetilde{F}$, we conclude from \cite{DoSk} or \cite{DoLeSk} that the functor $\widetilde{F}_{\lambda}$ is dense. As in the previous example, we conclude that the functor $\Phi_{F}:\CF(R)\ra\CF(A)$ is also dense in this case. \epv
\end{exa}

In the above example the $K$-category $B=K(\Gamma_{R}^{\op})$ is not only locally support-finite, but also locally representation-finite. This follows directly from the results of \cite{IT,IPTZ} where the authors give sufficient criteria for Auslander algebras to be representation finite.\footnote{One can also prove directly that the subcategory $C$ of $B$ (we use the same notation as above) is representation-finite.} In the last example we show a Galois covering functor $F:R\ra A$ such that $\Phi_{F}:\CF(R)\ra\CF(A)$ is dense but the category $K(\Gamma_{R}^{\op})$ is not locally representation-finite.

\begin{exa}\label{e3} Assume that $R=\underline{KQ_{R}}\slash I_{R}$ where $Q_{R}$ is the following quiver: $$
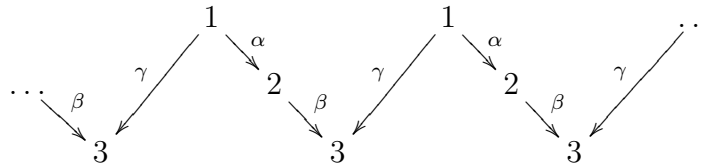
\xymatrix@C=1em@R=1em{&&&1\ar[dr]^{\alpha}\ar[ddll]_{\gamma}&&&&1\ar[dr]^{\alpha}\ar[ddll]_{\gamma}&&&&\dots\ar[ddll]_{\gamma}\\\dots\ar[dr]^{\beta}&&&&2\ar[dr]^{\beta}&&&&2\ar[dr]^{\beta}&&\\&3&&&&3&&&&3}$$ and $I_{R}$ is generated by all zero relations of the form $\alpha\beta=0$. The group $G=\mathbb{Z}$ acts on $R$ by horizontal translation and hence $R\slash G$ is the bound-quiver $K$-algebra $A=KQ_{A}\slash I_{A}$ where $Q_{A}$ is the quiver $$
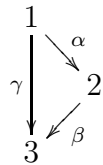
\xymatrix@C=1em@R=1em{1\ar[dr]^{\alpha}\ar[dd]_{\gamma}\\&2\ar[dl]^{\beta}\\3}$$ and $I_{A}=\langle\alpha\beta\rangle$. We obtain a Galois covering functor $F:R\ra A$ with the covering group $G=\mathbb{Z}$ where $R$ is locally representation-finite simply connected and thus $A$ is standard. The Auslander-Reiten quiver $\Gamma_{R}$ has the following shape:
$$\cdots \vcenter
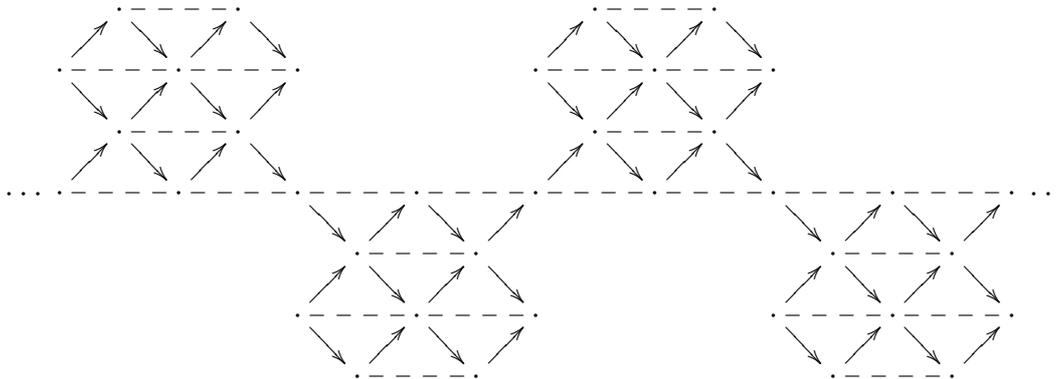
{\xymatrix@C=1.1em@R=1em{&\cdot\ar[dr]\ar@{--}[rr]&&\cdot\ar[dr] &&&&&& \cdot\ar[dr]\ar@{--}[rr]&&\cdot\ar[dr]\\
 \cdot\ar[ur]\ar[dr]\ar@{--}[rr]&&\cdot\ar[ur]\ar[dr]\ar@{--}[rr]&&\cdot &&&&  \cdot\ar[ur]\ar[dr]\ar@{--}[rr]&&\cdot\ar[ur]\ar[dr]\ar@{--}[rr]&&\cdot\\&\cdot\ar[ur]\ar[dr]\ar@{--}[rr]&&\cdot\ar[ur]\ar[dr] &&&&&& \cdot\ar[ur]\ar[dr]\ar@{--}[rr]&&\cdot\ar[ur]\ar[dr] \\
 \cdot\ar[ur]\ar@{--}[rr]&&\cdot\ar[ur]\ar@{--}[rr]&&\cdot\ar[dr]\ar@{--}[rr]&&\cdot\ar[dr]\ar@{--}[rr]&&\cdot\ar[ur]\ar@{--}[rr]&&\cdot\ar[ur]\ar@{--}[rr]&&\cdot\ar[dr]\ar@{--}[rr]&&\cdot\ar[dr]\ar@{--}[rr]&&\cdot\\
 &&&&&\cdot\ar[ur]\ar[dr]\ar@{--}[rr]&&\cdot\ar[ur]\ar[dr] &&&&&& \cdot\ar[ur]\ar[dr]\ar@{--}[rr]&&\cdot\ar[ur]\ar[dr] \\
 &&&&\cdot\ar[ur]\ar[dr]\ar@{--}[rr]&&\cdot\ar[ur]\ar[dr]\ar@{--}[rr]&&\cdot &&&& \cdot\ar[ur]\ar[dr]\ar@{--}[rr]&&\cdot\ar[ur]\ar[dr]\ar@{--}[rr]&&\cdot \\
 &&&&&\cdot \ar[ur]\ar@{--}[rr]&& \cdot\ar[ur] &&&&&& \cdot \ar[ur]\ar@{--}[rr]&& \cdot\ar[ur]}} \cdots$$ and thus the $K$-category $B=K(\Gamma_{R}^{\op})$ is strongly simply connected \cite{Sk2}. The category $B$ contains a convex subcategory of the form $$
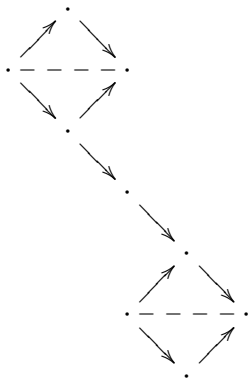
\xymatrix@C=1.1em@R=1em{&\cdot\ar[dr]\\\cdot\ar[ur]\ar[dr]\ar@{--}[rr]&&\cdot\\&\cdot\ar[ur]\ar[dr]\\&&\cdot\ar[dr]\\&&&\cdot\ar[dr]\\&&\cdot\ar[ur]\ar[dr]\ar@{--}[rr]&&\cdot\\&&&\cdot\ar[ur]}$$ which is of infinite representation type and so $B$ is not locally representation-finite. However, $B$ does not contain a hypercritical or pg-critical convex subcategory (see \cite{NoSk} for the definitions) and thus the functor $\widetilde{F}_{\lambda}$ is dense by \cite[Corollary 2.5]{Sk3} where $\widetilde{F}:\Gamma_{R}\ra\Gamma_{A}$, as usual. As before we conclude that functor $\Phi_{F}:\CF(R)\ra\CF(A)$ is dense. \epv


\end{exa}

\begin{rem}\label{r7} This section gives examples of Galois coverings $F:R\ra A$ for which the functor $\Phi_{F}:\CF(R)\ra\CF(A)$ is either dense or not. We only deal with Galois coverings of representation-finite algebras by locally representation-finite locally bounded $K$-categories. Moreover, in the final results we assume that $R$ is simply connected. These assumptions allow to view the functor $\Phi_{F}:\CF(R)\ra\CF(A)$ as a suitable push-down functor between module categories over the mesh categories of $\Gamma_{R}$ and $\Gamma_{A}$, see Theorem \ref{t11}. We are convinced that these methods can be generalized, at least to some extent, to the representation-infinite case. Indeed, it seems that one can make constructions analogues to those in Example \ref{e1} in the case of standard connected components of $\Gamma_R$. This can be used to showing examples of functors $\Phi_{F}:\CF(R)\ra\CF(A)$ which are not dense, but is not sufficient for showing the density. It is an interesting open problem to provide a criterion for the density of the functor $\Phi_{F}:\CF(R)\ra\CF(A)$ in analogy with the classical results of \cite{DoSk} and \cite{DoLeSk}. One could start with a generalization of the definition of a locally support-finite locally bounded $K$-category to the realm of categories of finitely presented functors. We believe that the following property is worth studying in this context: \begin{center}\emph{For any module $X\in\ind(R)$, the union of supports of all indecomposable functors $T\in\CF(R)$ such that $T(X)\neq 0$ is contained in the support of a hom-functor.}\footnote{The support of $T\in\CF(R)$ is defined as the class of all modules $X\in\ind(R)$ such that $T(X)\neq 0$.}\end{center} It should be emphasized, however, that this condition is useless without significant generalizations of many fundamental concepts from \cite{DoSk,DoLeSk}.
\end{rem}

\section{Some special applications}
This section is devoted to present some applications of Theorem \ref{t3} in calculating the Krull-Gabriel dimension of algebras or determining their representation types. We concentrate on two prominent classes of algebras: the \emph{tame algebras with strongly simply connected Galois coverings} \cite{Sk3} and the \emph{weighted surface algebras} \cite{ErSk1,ErSk2,ErSk3}. In the case of the first class we apply Theorem \ref{t3} (2) whereas Theorem \ref{t3} (1) in the case of the second. The main results of the section are Theorems \ref{t13} and \ref{t14}. We stress that in these theorems we are not allowed to apply Theorem \ref{t8} (i.e. our previous results from \cite{P4,P6}), because in general we consider Galois $G$-coverings $R\ra A$ for which the group $G$ is not torsion-free nor the push-down functor $F_{\lambda}$ is dense.

\subsection{Algebras with strongly simply connected Galois coverings}

We start with recalling basic definitions and facts related with strongly simply connected algebras. Assume that $A=kQ\slash I$ is a triangular algebra. Then $A$ is \textit{strongly simply connected} \cite{Sk2} if and only if the first Hochschild cohomology group $H^{1}(C,C)$ vanishes, for any convex subcategory $C$ of $A$. In the representation theory of strongly simply connected algebras, \textit{hypercritical} and \textit{pg-critical} algebras play a prominent role. The definitions are recalled below. We refer the reader to \cite[Chapter VI]{AsSiSk} or \cite{HR} for the fundamental concepts of tilting theory and to \cite[Chapters XIX, XX]{SiSk3} for the definitions of representation types and the growth of a tame algebra.

A quiver $Q$ is called \textit{minimal wild} if and only if its underlying graph is one of the following graphs: $T_{5}$, $\wt{\wt{\DD}}_{n}$ or $\wt{\wt{\EE}}_{i}$ for $i=6,7,8$, see for example \cite{Sk1}. The associated hereditary $K$-algebra $KQ$ is a \textit{minimal wild hereditary} algebra. A \textit{hypercritical algebra} is a concealed algebra of type $Q$.

Assume that $H$ is a hereditary algebra of type $\wt{\DD}_{n}$, $T$ is a tilting $H$-module without prinjective direct summands and $B=\End_{H}(T)^{\op}$ is a representation-infinite tilted algebra. We denote by $\CT(T)$ the torsion class induced by $T$. A \textit{pg-critical algebra of type I} is a matrix algebra of the form $$B[M,t]=\left[\begin{matrix}K&0&\cdots&0&0&0&0\\ K&K&\cdots&0&0&0&0\\\vdots&\vdots&\ddots&\vdots&\vdots&\vdots&\vdots\\K&K&\cdots&K&0&0&0\\K&K&\cdots&K&K&0&0\\K&K&\cdots&K&0&K&0\\M&0&\cdots&0&0&0&B\end{matrix} \right]$$ such that $t\geq 1$, $M=\Hom_{H}(T,S)$, where $S$ is an indecomposable regular $H$-module of regular length $1$ lying in a tube of rank $n-2$ such that $S\in\CT(T)$, and any proper convex subcategory of $B[M,t]$ is of polynomial growth. A \textit{pg-critical algebra of type II} is a matrix algebra of the form $$B[N]=\left [\begin{matrix} K & 0\\ N & B\end{matrix}\right]$$ such that $N=\Hom_{H}(T,R)$, where $R$ is an indecomposable regular $H$-module of regular length $2$ lying in a tube of rank $n-2$ such that $R\in\CT(T)$, and any proper convex subcategory of $B[N]$ is of polynomial growth. A \textit{pg-critical algebra} is a pg-critical algebra of type I or II.

\begin{rem}\label{r8}
It is known that hypercritical algebras are strictly wild and are classified by quivers and relations in \cite{Un}. In turn, pg-critical algebras are tame of non-polynomial growth \cite[Proposition 2.4]{Sk1} and are classified by quivers and relations in \cite{NoSk}. It follows by \cite[Theorem 7.1]{KaPa2} that hypercritical algebras (see also \cite[10.3]{Pr} and \cite{P5,GP} for wild algebras in general) and pg-critical algebras both have Krull-Gabriel dimension undefined.\footnote{In fact, Theorem 7.1 of \cite{KaPa2} shows a stronger result that algebras from these classes possess \emph{independent pairs of dense chains of pointed modules} \cite{KaPa1,KaPa2} which further implies that their lattices of pointed modules are wide \cite[Section 3]{KaPa2}. This yields their Krull-Gabriel dimensions are undefined, see \cite[Chapters 7,13]{Pr2}.}
\end{rem}

We need to recall the definition of weakly periodic modules, in particular the linear modules. Assume that $R$ is a locally bounded $K$-category and $G$ is a group of $K$-linear automorphisms of $R$ acting freely on the objects of $R$. Given an $R$-module $M$ we denote by $G_{M}$ the \emph{stabilizer} $\{g\in G\mid {}^{g}M\cong M\}$ of $M$. An indecomposable $R$-module $M$ in $\Mod(R)$ is \textit{weakly $G$-periodic} \cite[2.3]{DoSk} if and only if $\supp (M)$ is infinite and $(\supp M)\slash G_{M}$ is finite. 

Assume that $D$ is a full subcategory of $R$ and $g\in G$. Then $gD$ denotes the full subcategory of $R$ formed by all objects $gx$, $x\in D$. The set $\{g\in G\mid gD=D\}$ is the \textit{stabilizer} $G_{D}$ of $D$. A \textit{line} in $R$ is a convex subcategory of $R$ which is isomorphic to the path category of a linear quiver, i.e. a quiver of type $\AA_{n}$, $\AA_{\infty}$ or ${}_{\infty}\AA_{\infty}$. A line $L$ is \textit{$G$-periodic} if and only if the stabilizer $G_{L}$ is nontrivial. Observe that in this case the quiver of the line $L$ is of the type ${}_{\infty}\AA_{\infty}$.\footnote{It is well known that a $G$-periodic line $L$ in $R$ induces a \textit{band} in $R/G$, and every band induces a 1-parameter family of indecomposable $R/G$- modules, see for example \cite{DoSk}. The length of the band corresponding to $L$ equals the number of the $G_L$-orbits of the vertices of $L$.} 

Assume that $L$ is a $G$-periodic line in $R$ {and let $Q_L$ be the full subquiver of $Q$, whose vertices are the objects of $L$}. Then the \textit{canonical} weakly $G$-periodic $R$-module is the module $M_{L}$ such that $M_{L}(x)=K$ for $x\in Q_{L}$, $M_{L}(x)=0$ for $x\notin Q_{L}$ and $M_{L}(\alpha)=id_{K}$ for any arrow $\alpha$ in $Q_{L}$. Note that $G_{M_{L}}\cong G_{L}\cong\ZZ$. If $M$ is a weakly $G$-periodic $R$-module and $M\cong M_{L}$ for some $G$-periodic line $L$ in $R$, then $M$ is called \textit{linear}.

Assume that $R=\und{kQ}\slash I$ is a triangular locally bounded $k$-category. Then $R$ is \textit{strongly simply connected} if and only if the following two conditions are satisfied: 
\begin{itemize}
  \item $R$ is intervally finite,
  \item every finite convex subcategory of $R$ is strongly simply connected.
\end{itemize}

Tame algebras having strongly simply connected Galois coverings is a wide and important class, containing in particular all special biserial algebras \cite[5.2]{DoSk}. This class is studied in depth by A. Skowro\'nski in \cite{Sk3}. Here we only mention the following main result of this paper which we directly apply below. 

\begin{thm}\label{t12}\textnormal{(\cite[Theorem 2.6]{Sk3})} Assume that $R$ is a strongly simply connected locally bounded $K$-category, $G$ an admissible group of $K$-linear automorphisms of $R$ and let $A=R\slash G$. Then the following assertions hold:
\begin{enumerate}[\rm(1)]
  \item The algebra $A$ is of polynomial growth if and only if the category $R$ does not contain a convex subcategory which is hypercritical or pg-critical, and the number of $G$-orbits of $G$-periodic lines in $R$ is finite.
  \item The algebra $A$ is of domestic type if and only if the category $R$ does not contain a convex subcategory which is hypercritical, pg-critical or tubular, and the number of $G$-orbits of $G$-periodic lines in $R$ is finite.
\end{enumerate}
\end{thm}

The following theorem applies Theorem \ref{t3} (2).

\begin{thm}\label{t13} Assume that $R$ is a strongly simply connected locally bounded $K$-category, $G$ an admissible group of $K$-linear automorphisms of $R$ and $A=R\slash G$. If $\KG(A)$ is finite, then $A$ is of domestic type.
\end{thm}

\begin{proof} Assume that $\KG(A)$ is finite. Then Theorem \ref{t3} (2) yields $\KG(R)$ is finite (note that $F_{\lambda}$ is not dense in general, see \cite[Corollary 2.5]{Sk3}) and we show that in this case $R$ does not contain a convex subcategory which is hypercritical, pg-critical or tubular, and the number of $G$-orbits of $G$-periodic lines in $R$ is finite. First recall that if $B$ is a convex subcategory of $R$, then $\KG(B)\leq\KG(R)$ by Lemma \ref{0l2}. Therefore, if $R$ contains a convex subcategory $B$ which is hypercritical or pg-critical, then $\infty=\KG(B)\leq\KG(R)$ (see Remark \ref{r8}) and thus $\KG(R)=\infty$. We conclude the same if $R$ contains a convex subcategory which is tubular, because tubular algebras have Krull-Gabriel dimension undefined \cite{Geigle1985}. Finally, assume that the number of $G$-orbits of $G$-periodic lines in $R$ is infinite. Without loss of generality we assume that $R$ does not contain a convex subcategory which is hypercritical or pg-critical. Then it follows from \cite[Lemma 6.1 (3), Theorem 6.2]{KaPa2} and their proofs that $A=R\slash G$ has a factor algebra $C$ which is string of non-domestic type. Hence $\KG(C)=\infty$ by \cite{Sch2} and since $\KG(C)\leq\KG(A)$ by Lemma \ref{0l2} we get $\KG(A)=\infty$, contrary to our assumption that $\KG(A)$ is finite. These arguments show that we can apply Theorem \ref{t12} (2) and conclude that $A$ is of domestic type. 
\end{proof}

\begin{rem}\label{r9} It seems that our techniques are not sufficient to prove the converse of Theorem \ref{t13}. Such a result, together with the above theorem, implies that Prest's conjecture holds for the class of algebras with strongly simply connected Galois coverings. However, recall that it is proved in \cite{We} that for a strongly simply connected algebra $A$ the following statements are equivalent: (1) $\KG(A)=2$, (2) $\KG(A)$ exists and (3) $A$ is of domestic type. This result applies the first version of Theorem \ref{t12} (2), proved only for strongly simply connected algebras (and hence without the redundant assumption that the number of $G$-orbits of $G$-periodic lines in $A$ is finite). We believe that the three equivalences can be generalized to the case of algebras with strongly simply connected Galois coverings, perhaps applying also the results of \cite{LPP} where the authors prove that string algebras of domestic type have finite Krull-Gabriel dimension.
\end{rem}

\subsection{Weighted surface algebras}

In the final part of the section we show an application of Theorem \ref{t3} (1) to \emph{weighted surfaces algebras}, introduced in \cite{ErSk1}. The class of weighted surface algebras is a large class of symmetric periodic algebras of period 4 (with few exceptions), associated to triangulations of real compact surfaces. We refer to \cite{ErSk1,ErSk2,ErSk3} for the background and motivation to study this class. 

Our goal is to show that weighted surface algebras of two families, $D(\lambda)^{(1)}$ and $D(\lambda)^{(2)}$ \cite{ErSk2}, have Krull-Gabriel dimension undefined. This is only to give a flavor of our main results from \cite{EJ-PP} where we prove, in particular, that all weighted surface algebras have Krull-Gabriel dimension undefined.

The family $D(\lambda)^{(1)}$, $\lambda \in K^*=K\setminus\{0\}$, is given by $2$-regular quiver $Q(D(\lambda)^{(1)})$:
\[
  \xymatrix{
    1
    \ar@(ld,ul)^{\xi}[]
    \ar@<.5ex>[r]^{\beta}
    & 3
    \ar@<.5ex>[l]^{\alpha}
    \ar@(ru,dr)^{\gamma}[]
  }
\]
permutation $f$ with orbits
$(\xi\ \beta\ \alpha)$, $(\gamma)$,
and the relations:
\begin{align*}
&&
 \alpha \xi &= \lambda \gamma \alpha \beta \gamma \alpha ,
 &
 \xi \beta &= \lambda \beta \gamma \alpha \beta \gamma ,
 &
 \gamma^2 &= \lambda \alpha \beta \gamma \alpha \beta ,
 &
 \beta \alpha &=  \xi ,
  &
 \\
 \alpha \xi^2 &= 0 ,
 &
 \xi \beta \gamma &= 0 ,
 &
 \gamma^2 \alpha &= 0 ,
 &
 \xi^2 \beta &= 0 ,
 &
 \gamma \alpha \xi &= 0 ,
 &
 \beta \gamma^2 &= 0.
\end{align*}

\noindent The family $D(\lambda)^{(2)}$, $\lambda \in K^*$, is given by $2$-regular quiver $Q(D(\lambda)^{(2)})$:
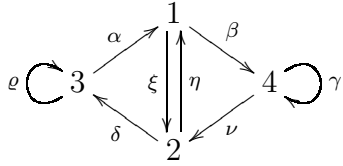
\[
  \xymatrix@R=1pc{
   & 1 \ar[rd]^{\beta} \ar@<-.5ex>[dd]_{\xi}  \\
   3   \ar@(ld,ul)^{\varrho}[]  \ar[ru]^{\alpha}
   && 4   \ar@(ru,dr)^{\gamma}[]  \ar[ld]^{\nu}   \\
   & 2 \ar[lu]^{\delta}  \ar@<-.5ex>[uu]_{\eta}
  }
\]
permutation $f$ with orbits
$(\alpha\ \xi\ \delta)$, $(\beta\ \nu\ \eta)$, $(\varrho)$, $(\gamma)$
and the set of relations:
\begin{align*}
 \alpha \xi &= \lambda \varrho \alpha \beta \gamma \nu ,
 &
 \xi \delta &= \lambda \beta \gamma \nu \delta \varrho ,
 &
 \varrho^2 &= \lambda \alpha \beta \gamma \nu \delta ,
 &
 \delta \alpha &=  \eta ,
 &
 \alpha  \xi \eta &= 0 ,
 \\
 \nu \eta &= \lambda \gamma \nu \delta \varrho \alpha ,
 &
 \eta \beta &= \lambda \delta \varrho \alpha \beta \gamma ,
 &
 \gamma^2 &= \lambda \nu \delta \varrho \alpha \beta ,
 &
 \beta \nu &=  \xi ,
 &
 \xi \eta \beta &= 0 ,
 \\
 \xi \delta \varrho &= 0 ,
 &
 \nu \eta \xi &= 0 ,
 &
 \eta \beta \gamma &= 0 ,
 &
 \varrho^2 \alpha &= 0 ,
 &
  \gamma^2 \nu &= 0 ,
 \\
 \varrho \alpha \xi &= 0 ,
 &
 \eta \xi \delta &= 0 ,
 &
 \gamma \nu \eta &= 0 ,
 &
 \delta \varrho^2 &= 0 ,
 &
 \beta \gamma^2 &= 0 .
\end{align*}

The following theorem is a part of \cite[Theorem 4.3]{EJ-PP}

\begin{thm}\label{t14} We have $\KG(D(\lambda)^{(1)})=\KG(D(\lambda)^{(2)})=\infty$.
\end{thm}

\begin{proof} In the proof we follow the lines of \cite[Lemma 6.9]{ErSk2} where it is shown that $D(\lambda)^{(1)}$, $D(\lambda)^{(2)}$ are of non-polynomial growth. We apply freely Lemma \ref{0l2}. 

Denote by $(Q^{(i)}, I^{(i)})$ the bound quiver of an algebra $D(\lambda)^{(i)}$, for $i \in \{1,2\}$. Let  $A^{(i)}$ be the quotient algebra of $D(\lambda)^{(i)}$ given by the Gabriel quiver $Q^{(i)}$ and the ideal of relations generated by zero-relations $\omega=0$ and $v=0$ for any commutativity relation $\omega-v$ in $I^{(i)}$. It follows from the proof of \cite[Lemma 6.9]{ErSk2} that $A^{(i)}$ admits a Galois covering $F^{(i)}\colon R \rightarrow A^{(i)}$ by a locally bounded $K$-category $R$ such that $R$ contains a finite convex subcategory $B$ which is pg-critical algebra. Hence we have the following diagram:
      \vspace{-0.3cm}\[
    \begin{tikzpicture}
    [scale=.4]
    \node (D) at (4.6,4) {$D(\lambda)^{(i)}$};
    \node (R) at (-4,0) {$R$};
    \node (A) at (2,0) {$R/G^{(i)}$};
    \node (A') at (4.6,0) {$=A^{(i)}$};
    \node (B) at (-4,-3.4) {$B$};
    \draw[->,thick]
    (B) edge node [left] {$\iota$} (R)
    (R) edge node [above] {$F$} (A)
    (D) edge node [right] {$p$} (A')
    ;
    \end{tikzpicture}\vspace{-0.3cm}
    \] 
where $p \colon D(\lambda)^{(i)} \rightarrow A^{(i)}$  is a surjection of $K$-algebras and $\iota \colon B \rightarrow R$ is an embedding of locally bounded $K$-categories. Since $B$ is pg-critical, we get $\KG(B) =\infty$, see Remark \ref{r8}, and we conclude that $\KG(R) =\infty$, because $\KG(B)\leq \KG(R)$. Applying Theorem \ref{t3} (1) to the Galois covering $F^{(i)}\colon R \rightarrow A^{(i)}$ we get $\KG(A^{(i)})=\infty$ and so $\KG(D(\lambda)^{(i)})=\infty$, for $i \in \{1,2\}$ and $\lambda \in K^*$.
\end{proof}

\section{On the conjecture of Prest}

We finish the paper with few remarks concerning possible applications of our results, especially Theorem \ref{t3}, in the conjecture of M. Prest on Krull-Gabriel dimension. Since the general form of the conjecture seems to be out of reach, we believe one should restrict considerations to the case of algebras $A$ having Galois $G$-coverings $R\ra A$ such that $G$ acts freely on the isomorphism classes in $\ind(R)$ (for example, $G$ may be torsion-free) and every weakly $G$-periodic $R$-module is linear.\footnote{Recall that this is typically the case of algebras with strongly simply connected Galois coverings, string algebras in particular, see \cite{Sk3} and \cite{KaPa3}.} Indeed, based on the renowned result \cite[Theorem 3.6]{DoSk}, we get the following theorem which in our opinion should attract much more attention in this context (see \cite[Theorem 2.4]{Sk3} for a special version). 

We denote by $\CL_0$ a fixed set of representatives of all $G$-orbits of $G$-periodic lines in $R$. As usual, $K[T,T^{-1}]$ is the algebra of \emph{Laurent polynomials}. 

\begin{thm}\label{t15} Assume that $R$ is a locally bounded $K$-category over algebraically closed field $K$, $G$ an admissible group of $K$-linear automorphisms of $R$ and $F:R\ra A\cong R\slash G$ the Galois covering. Assume that the group $G$ acts freely on the isomorphism classes in $\ind(R)$ and every weakly $G$-periodic $R$-module is linear. The following assertions hold:
\begin{enumerate}[\rm(1)]
  \item Any module $Z\in\ind(A)$ of the second kind is of the form $V\otimes_{K[T,T^{-1}]}F_{\lambda}(M_{L})$ for some $L\in\CL_{0}$ and indecomposable finite dimensional $K[T,T^{-1}]$-module $V$.
  \item We have $$\Gamma_{A}=(\Gamma_{R}\slash G)\vee (\bigvee_{L\in\CL_{0}\snull}\Gamma_{K[T,T^{-1}]})$$ where $\Gamma_{K[T,T^{-1}]}$ denotes the Auslader-Reiten quiver of the category of finite dimensional $K[T,T^{-1}]$-modules.
\end{enumerate}
\end{thm} We recall that if $L\in\CL_0$, then $G_{M_{L}}=G_{L}\cong\ZZ$. Hence the group algebra $KG_{L}$ is isomorphic with $K[T,T^{-1}]$ and the canonical action of $G_{L}$ on $L$ gives a left $K[T,T^{-1}]$-module structure on the module $F_{\lambda}(M_{L})$. We refer to \cite{DoSk} for the details.

Theorem \ref{t15} shows deep connections between representation theories of $R$ and $A$ in the above situation. In the context of Prest's conjecture, it particularly suggests to look for relations between the finiteness of $\KG(R)$ and a condition expressing \emph{domesticity of $(\Gamma_{R}\slash G)$}, equivalently, \emph{domesticity of the category of modules of the first kind}. Since such concepts are not standard, we believe that the following definition could be of some use. The definition is based on the well known definition of domestic representation type, see \cite[XIX 3.12]{SiSk3} and also a generalization in \cite[XIX 3.1]{SiSk3}. 

\begin{df}\label{d2} Assume that the field $K$ is algebraically closed and let $R$ be a locally bounded $K$-category. A full subcategory $\CA\subseteq\mod(R)$ is of \emph{domestic representation type} if an only if there are $K[t]$-$R$-bimodules ${}_{K[t]}N^{(1)}_{R},\cdots,{}_{K[t]}N^{(n)}_{R}$ that are finitely generated and free as left $K[t]$-modules such that for any integer $d\geq 1$ all but a finite number of isomorphism classes of indecomposable modules in $\CA$ of dimension $d$ are isomorphic to $R$-modules of the form $X\otimes_{K[t]}N^{(i)}_{R}$ where $i=1,\cdots,n$ and $X$ is a finite dimensional indecomposable $K[t]$-module.
\end{df}

We pose the following two problems. The first one is a version of Prest's conjecture for locally bounded $K$-categories, but in a restricted setting.

\begin{prob}\label{pro1} Assume that $R=\und{KQ_{R}}\slash I_{R}$ is a locally bounded $K$-category over algebraically closed field $K$ and $G$ is an admissible group of $K$-linear automorphisms of $R$. Let $F:R\ra A\cong R\slash G$ be the associated Galois covering. Moreover, assume that the following conditions are satisfied: 
\begin{itemize}
  \item the quiver $Q_R$ is a tree,
  \item the group $G$ acts freely on the isomorphism classes in $\ind(R)$,
  \item every weakly $G$-periodic $R$-module is linear. 
\end{itemize} Under these assumptions, is the category of $\mod(A)_{1}$ of $A$-modules of the first kind of domestic representation type if and only if $\KG(R)$ is finite? \epv
\end{prob}

The second problem may be seen as a possible path of generalizing Theorem \ref{t8} to the realm of locally bounded categories which are not necessarily locally support-finite.

\begin{prob}\label{pro2} Under the same assumptions as in Problem \ref{pro1}, is $\KG(R)$ finite if and only if $\KG(A)$ is finite? Does the equality $\KG(R)=\KG(A)$ hold? \epv
\end{prob}

While the above problems are undoubtedly challenging, they seem accessible thanks to restrictions imposed on $R$.  We demonstrate below how positive solutions to these problems could be useful in the verification of the conjecture of Prest.

\begin{dis}\label{dis1} Assume that the field $K$ is algebraically closed. Let $A$ be a finite dimensional $K$-algebra and $F:R\ra A$ the universal Galois covering. Recall that then the quiver $Q_R$ of $R$ is a tree. Assume that the conditions of Theorem \ref{t15} are satisfied which in particular means that $$\Gamma_{A}=(\Gamma_{R}\slash G)\vee(\bigvee_{L\in\CL_{0}\snull}\Gamma_{K[T,T^{-1}]}).$$ Assume that $\KG(A)$ is finite. This should imply that the set $\CL_0$ of representatives of all $G$-orbits of $G$-periodic lines in $R$ is finite. Indeed, in principle it should be possible to repeat arguments from Section 6 of \cite{KaPa3} to show that if $\CL_0$ is infinite, then $A$ possesses a quotient algebra which is string of non-domestic type.\footnote{It could be necessary to assume that $R$ is not of wild type to draw this conclusion, but in this case $A$ is also wild, which contradicts with the finiteness of $\KG(A)$.} Then Theorem \ref{t3} implies that $\KG(R)$ is finite and if the solution to Problem \ref{pro1} is positive, we get that the category $\mod(A)_{1}$ is of domestic type. Then the shape of $\Gamma_{A}$ should allow to conclude that $A$ itself is of domestic representation type. On the other hand, if $A$ is of domestic type, then the shape of $\Gamma_{A}$ implies that $\CL_{0}$ is finite and $\mod(A)_{1}$ is of domestic type as well. If the solution to the Problem \ref{pro1} is positive, then $\KG(R)$ is finite. In turn, if solution to the Problem \ref{pro2} is positive, we get that $\KG(A)$ is finite. 
\end{dis}

\begin{rem}\label{r10} In private discussions, A. Skowro\'{n}ski suggested that if there were any counterexamples to the general form of Prest's conjecture, one should try look at the algebras studied in \cite{BDS}. These are special classes of domestic but quite complex algebras, with many nonperiodic Auslander-Reiten components. Such algebras could be candidates for domestic algebras with the Krull-Gabriel dimension undefined.
\end{rem}

\bibsection

\end{document}